\documentclass[10pt,a4paper,reqno]{amsart}             
\usepackage{enumitem}
\usepackage{mathabx}

\def\MC{\multicolumn{1}{c}{}}

\usepackage[english]{babel}
\usepackage[utf8]{inputenc}
\usepackage{hyperref}
\usepackage{latexsym}
\usepackage{verbatim}
\usepackage{calligra}
\usepackage{rotating}
\usepackage{color,bbm}
\usepackage{amssymb,amsfonts,amsmath,stmaryrd,esint}
\usepackage[toc]{appendix}
\usepackage[T1]{fontenc}
\usepackage{soul}

\newcommand{\LandauO}{\mathcal{O}}

\newcommand{\GK}{\mathbb{K}}
\newcommand{\GF}{\mathbb{F}}

\newcommand{\Cgf}{{\sf C}}
\newcommand{\Dgf}{{\sf D}}
\newcommand{\Egf}{{\sf E}}
\newcommand{\Ggf}{{\sf G}}

\newcommand{\Lgf}{{\sf L}}
\newcommand{\Pgf}{{\sf P}}
\newcommand{\Qgf}{{\sf Q}}

\newcommand{\Rgf}{{\sf R}}
\newcommand{\Sgf}{{\sf S}}

\newcommand{\Dnn}{{\mathcal D}}
\newcommand{\Enn}{{\mathcal E}}
\newcommand{\Fnn}{{\mathcal F}}
\newcommand{\Pnn}{{\mathcal P}}

\newcommand{\Knn}{{\mathcal K}}
\newcommand{\Lnn}{{\mathcal L}}
\newcommand{\Mnn}{{\mathcal M}}
\newcommand{\Znn}{{\mathcal Z}}

\newcommand{\Brm}{\mathrm{B}}
\newcommand{\Erm}{\mathrm{E}}
\newcommand{\Grm}{\mathrm{G}}
\newcommand{\Hrm}{\mathrm{H}}
\newcommand{\Srm}{\mathrm{S}}
\newcommand{\Trm}{\mathrm{T}}
\newcommand{\Xrm}{\mathrm{X}}
\newcommand{\Yrm}{\mathrm{Y}}

\newcommand{\al}{\alpha}

\newcommand{\Xz}{\chi}

\DeclareMathOperator{\Pol}{Pol}

\newcommand{\Cyl}{{\mathcal C}}
\newcommand{\Strip}{{\mathcal S}}

\newtheorem{Theorem}{Theorem}[section]
\newtheorem{Lemma}[Theorem]{Lemma}
\newtheorem{Proposition}[Theorem]{Proposition}
\newtheorem{Definition}[Theorem]{Definition}
\newtheorem{Corollary}[Theorem]{Corollary}

\newtheorem{Prediction}[Theorem]{Prediction}

\theoremstyle{definition}
\newtheorem{Remark}[Theorem]{Remark}

\def\qed{$\hfill{\vrule height 3pt width 5pt depth 2pt}$}
\def\qee{$\hfill{\Box}$}

\newcommand{\beq}{\begin{equation}}
  \newcommand{\eeq}{\end{equation}}
\def\emm#1,{{\em #1}}

\catcode`\@=11
\def\section{\@startsection{section}{1}%
  \z@{.7\linespacing\@plus\linespacing}{.5\linespacing}%
  {\normalfont\bfseries\scshape\centering}}

\def\subsection{\@startsection{subsection}{2}%
  \z@{.5\linespacing\@plus\linespacing}{.5\linespacing}%
  {\normalfont\bfseries\scshape}}

\def\subsubsection{\@startsection{subsubsection}{3}%
  \z@{.5\linespacing\@plus\linespacing}{-.5em}%{.5\linespacing}%
  {\normalfont\bfseries}}
\catcode`\@=12

\newcommand{\ns}{\mathbb{N}}%{\mbox{\bbold N}}

\newcommand{\zs}{\mathbb{Z}}%{\mbox{\bbold Z}}

\newcommand{\qs}{\mathbb{Q}}%{\mbox{\bbold Q}}

\newcommand{\rs}{\mathbb{R}}%{\mbox{\bbold R}}
\newcommand{\cs}{\mathbb{C}}%{\mbox{\bbold C}}
\newcommand{\fps}{formal power series}
\newcommand{\bm}[1]{\mbox{\boldmath \ensuremath{#1}}}

\DeclareMathOperator{\vv}{v}
\DeclareMathOperator{\ff}{f}
\DeclareMathOperator{\ee}{e}
\DeclareMathOperator{\id}{id}
\DeclareMathOperator{\od}{od}

\DeclareMathOperator{\iq}{iq}
\DeclareMathOperator{\oc}{oc}

\newcommand{\gf}{generating function}
\newcommand{\gfs}{generating functions}

\newcommand{\om}{\omega} 
\renewcommand{\epsilon}{\varepsilon}

\graphicspath{{Figures/}}

\begin{document}
% --------------------------------------
% Front Matter
% --------------------------------------
\title{Refined enumeration of planar Eulerian orientations}

\author[M. Bousquet-M\'elou]{Mireille Bousquet-M\'elou}

\author[A. Elvey Price]{Andrew Elvey Price}

\thanks{MBM was partially supported by the ANR projects DeRerumNatura (ANR-19-CE40-0018) and Combiné (ANR-19-CE48-0011). AEP was partially
  supported by ACEMS in the form of a top up scholarship and travel stipend, a 2017 Nicolas Baudin travel grant, an Australian government research training program scholarship, the ERC project COMBINEPIC (Grant Agreement No. 759702) and the ANR project IsOMa (ANR-21-CE48-0007). MBM and AEP were also partially supported by the ANR project CartesEtPlus (ANR-23-CE48-0018).}

\address{MBM: CNRS, LaBRI, Universit\'e de Bordeaux, 351 cours de la
  Lib\'eration,  F-33405 Talence Cedex, France}

\address{AEP: CNRS, Institut Denis Poisson, Universit\'e de Tours, Parc Grandmont, 37200, Tours, France} 
\email{bousquet@labri.fr, andrew.elvey@univ-tours.fr}

\maketitle

\begin{abstract}
  We address the enumeration of Eulerian  orientations of quartic (i.e., 4-valent) planar maps according to three parameters: the 
  number of vertices, the number of \emm alternating, vertices (having in/out/in/out incident edges), and the number of clockwise oriented faces. This is a refinement of the six vertex model studied by Kostov, then Zinn-Justin and Elvey Price, where one only considers the first two parameters. Via a bijection of Ambj\o rn and Budd, our problem is equivalent to the enumeration of  Eulerian partial orientations of \emm general, planar maps, counted by the 
  number of edges, the number of undirected edges, and the number of vertices.

  We first derive from combinatorial arguments a system of functional equations characterising the associated trivariate series $\Qgf(t,\omega,v)$. We then derive from this system a  compact characterisation of this series. We use it to determine $\Qgf(t,\omega,v)$  in three two-parameter cases. The first two cases  correspond to setting the variable $\omega$ counting alternating vertices (or undirected edges after the AB bijection) to $0$ or $1$: when $\omega=0$ we count Eulerian orientations of general planar maps by edges and vertices, and when $\omega=1$ we count Eulerian orientations of quartic maps by vertices and clockwise faces. The final forms of these two series, namely $\Qgf(t,0, v)$ and $\Qgf(t,1,v)$,  refine those obtained by the authors in an earlier paper for $v=1$.  The third case that we solve, namely $v=1$ (but $\omega$ arbitrary), is the standard six-vertex model, for which we provide a new proof of the formula of Elvey Price and Zinn-Justin involving Jacobi theta functions.
  This new  derivation has the advantage that it remains purely in the world of formal power series, not relying on complex analysis. Our results also use a more direct approach to solving the functional equations, in contrast to the guess and check approaches used in previous work. 
\end{abstract}

%%%%%%%%%%%%%%%%%%%%%%%%%%%%%%%%%%%%%%%%%%%%%%%%%%%%%%%%%%%%%%%% 
\section{Introduction}
%%%%%%%%%%%%%%%%%%%%%%%%%%%%%%%%%%%%%%%%%%%%%%%%%%%%%%%%%%%%%%%% 
The enumeration of \emm planar maps, (connected multigraphs embedded in the sphere) is a classical topic in enumerative combinatorics. It started in the sixties, with the seminal work of Tutte based on recursive constructions (see e.g.~\cite{tutte-triangulations,tutte-census-maps}). About fifteen years later, the topic had a second, and independent birth in physics, where new approaches related to matrix integrals were developed~\cite{BIPZ,BIZ}. Since then, the theory of maps has been enriched by enlightening bijections with families of trees, which explain combinatorially why many families of maps are counted by simple numbers and/or simple generating functions (see e.g.~\cite{Sch97,BDG-planaires,bouttier-mobiles}). Beyond the raw enumeration of maps, both the combinatorics and the physics communities showed a very early interest in counting \emm maps equipped with an additional structure, such as a spanning tree, a colouring, a self-avoiding walk or a classical statistical mechanics model~\cite{DK88,Ka86,mullin-boisees,lambda12,tutte-dichromatic-sums}. This paper follows this line of research, by studying planar maps equipped with an \emm Eulerian orientation, of its edges: at each vertex, one finds as many incoming as outgoing edges (Figure~\ref{fig:defs}, right).

This question was first raised a few years ago by Bonichon et al.~\cite{BoBoDoPe}. The question then attracted more interest~\cite{elvey-guttmann17}, and  was finally solved by the authors in a 2020 paper~\cite{mbm-aep1}, both on general planar maps and on \emm quartic, maps (those in which all vertices have degree $4$). It turned out that the latter problem, and in fact a refinement of it known as the \emm $6$-vertex model,, had been studied and solved much earlier in theoretical physics, {in a different form, by Kostov~\cite{kostov} (although his solution appeared to be incorrect). The second author and Zinn-Justin then revisited Kostov's approach and derived the correct solution of the 6-vertex model, in terms of Jacobi's theta functions~\cite{elvey-zinn}. In this refined version,  every vertex at which the in/out orientations alternate is assigned
a weight $\om$. A non-trivial calculation allowed the authors of~\cite{elvey-zinn} to recover the solution obtained in~\cite{mbm-aep1} when $\om=1$, which takes a completely different, much simpler form. 

\begin{figure}[ht]
  \centering
  \includegraphics[width=12cm]{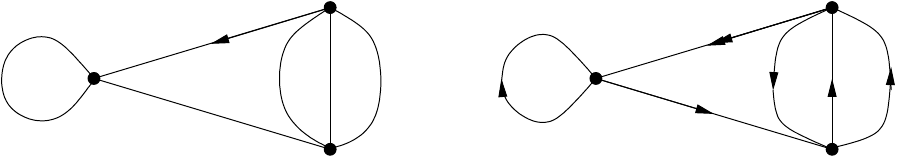}
  \caption{Left: a rooted planar map, which is 4-valent (or:
    quartic). Right: the same map, equipped with an Eulerian
    orientation.}
  \label{fig:defs} 
\end{figure}

In this paper, we refine the above enumeration results for quartic Eulerian orientations, by recording, in addition to the total number of vertices and the number of alternating vertices, the number of clockwise oriented faces. Via a bijection of Ambj\o rn and Budd~\cite{ambjorn-budd}, this problem is equivalent to the enumeration of  Eulerian \emm partial, orientations of \emm general, planar maps, counted by the 
number of edges, the number of undirected edges, and the number of vertices. In a \emm partial, Eulerian orientation, only some edges are oriented, and one requires again that the number of in- and out-edges coincide at each vertex. The case $\om=0$ corresponds to the enumeration of Eulerian orientations of planar maps, counted by edges and vertices.

We address this three-variate counting problem through recursive constructions of quartic Eulerian orientations. These constructions differ from those of earlier papers. They yield a system of functional equations for series that involve, in addition to the three variables corresponding to the three parameters of interest (denoted $t$, $\om$ and $v$), two additional variables~$x$ and~$y$, called \emm catalytic,. Manipulating this system, we derive a new characterisation of the solution, involving this time a single catalytic variable $x$. Finally, we use this to determine the solution in three two-parameter cases:
\begin{itemize}
\item {\bf The case  $\bm{\om=0}$:}  we then count Eulerian orientations of general planar maps by edges and vertices (Theorem~\ref{thm:general}); this refines~\cite{mbm-aep1}, where only the edge number was recorded.
\item {\bf The case  $\bm{\om=1}$:} we then count Eulerian orientations of quartic maps by vertices and clockwise faces.  (Theorem~\ref{thm:quartic}); this refines~\cite{mbm-aep1}, where only the vertex number was recorded.
\item {\bf The case $\bm{v=1}$:} this is the standard six-vertex model (Theorem~\ref{thm:allomega}), and we recover the result of~\cite{elvey-zinn} in terms of theta functions.
\end{itemize}
Our derivations differ significantly from those of~\cite{mbm-aep1,elvey-zinn}. The equations we start from are different, and our solution  does not involve any guessing, nor complex analytic tools or arguments, for instance the classical \emm one-cut assumption, that appears frequently in physics solutions of map problems.

\begin{figure}[ht]
  \centering
  \includegraphics[scale=0.7]{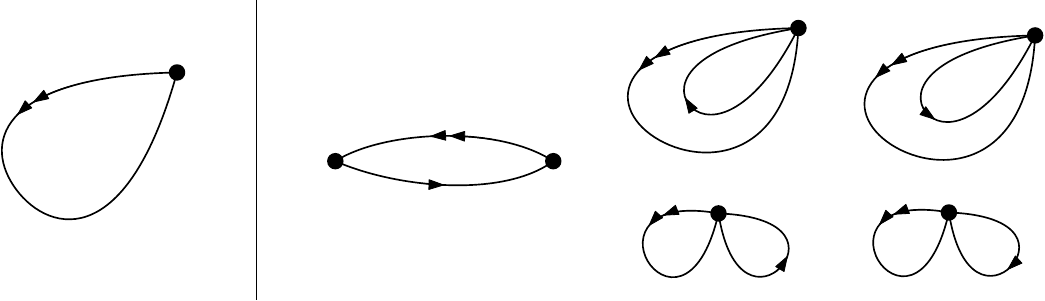}
  \caption{The planar Eulerian orientations with at most
    two edges.
    % , in agreement with       $\Ggf%(t,v)      = tv+ \left(v +4\right) t^{2} v + \LandauO(t^3)$ (the variables $t$ and $v$ count edges and vertices, respectively).
    On the right are the four quartic Eulerian
    orientations with one vertex.
    % , in agreement with $\Qgf_1%(t,v)      =(v+3)tv+\LandauO(t^2)$ (here, $t$ counts vertices while $v$ counts clockwise oriented faces) and $\widetilde Q= 2(1+\omega) t + \LandauO(t^2)$ (where $t$ counts vertices and $\omega$ alternating vertices).
  }
  \label{fig:small}
\end{figure}

Let us now state our results in these three two-parameter cases, beginning with Eulerian orientations of general planar maps (case $\om=0$). As usual in map enumeration, the oriented maps that we count are \emm rooted,, meaning that one edge is marked. Complete definitions will be given in Section~\ref{sec:defs}. 

\begin{Theorem}\label{thm:general}
  Let $\Rgf_0(t,v)\equiv \Rgf_0$ be the unique \fps\ in $t$ with constant term $0$ satisfying
  \beq\label{t-R0-first}
  t=\sum_{n,k\geq0}\frac{1}{n+1}{2n\choose n}{2n+k\choose k}{2n+k\choose n}t^k (v-1)^k\,\Rgf_0^{n+1} . %(tv-t)^k.
  \eeq
  This series has polynomial coefficients in $v$. Then the \gf\ of rooted planar Eulerian orientations, counted by edges (variable $t$) and vertices (variable $v$) is
  \[
    \Ggf%(t,v)
    = -\frac {v} 2+\frac{1}{2t^2}\sum_{{n,k\geq0},{~n+k>0}}\frac{1}{n+1}{2n\choose n}{2n+k\choose k}{2n+k-1\choose n} t^k(v-1)^k\, \Rgf_0^{n+1}.
  \]
  This series is differentially algebraic, that is, satisfies (non linear) differential equations in $t$ and~$v$  (of order~$3$ and $4$, respectively).

  The series $\Qgf_0:=2\Ggf$ also counts, by vertices and clockwise faces, quartic  Eulerian orientations having no alternating vertex.
\end{Theorem}
The first coefficients of $\Rgf_0$ and $\Ggf$ are:
\begin{align*}
  \Rgf_0&=t -\left(1+v \right) t^{2}-\left(1+3 v \right) t^{3}-\left(3+14 v+3 v^{2} \right) t^{4}+\LandauO(t^5),
  \\
  \Ggf&= tv+ v\left(v +4\right) t^{2} +v\left(v +10\right) \left(v +2\right) t^{3} +
        v\left(v^{3}+24 v^{2}+115 v +112\right) t^{4}  +\LandauO(t^5).
\end{align*}
Figure~\ref{fig:small} confirms the value of the coefficients of $t$ and $t^2$ in $\Ggf$. Note that when $v=1$, the only non-zero summands in the expressions of Theorem~\ref{thm:general}  are obtained for $k=0$. One then realises that
  \[
    % \Ggf=  -\frac {1} 2+\frac{1}{4t^2}(t-\Rgf_0).
    2t^2(2\Ggf+1) = t -\Rgf_0.
  \]
  This  simpler expression is the one given in~\cite{mbm-aep1}. No similar reduction seems to occur for a general value of $v$, but we will see that a natural counterpart is
  \beq\label{Q0-alt}
  2t^2(2\Ggf+v) = t -\Rgf_0 + \sum_{n,k\ge 0} \frac{1}{n+1}{2n\choose n}{2n+k\choose k}{2n+k+1\choose n} t^{k+1}(v-1)^{k+1}\, \Rgf_0^{n+1}.
  \eeq
  This expression follows in a straightforward fashion from the theorem.
  We give in Proposition~\ref{prop:trees0} a simple combinatorial interpretation in terms of trees of the series $t-\Rgf_0$ and $t^2(2\Ggf+v)$. A bijection remains to be found.

\smallskip
Our second result  (the case $\om=1$), for quartic orientations counted by vertices and clockwise oriented faces, takes a strikingly similar form.

\begin{Theorem}\label{thm:quartic}
  Let $\Rgf_1(t,v)\equiv \Rgf_1$ be the unique \fps\ in $t$ with constant term $0$ satisfying
  \beq\label{t-R1-first} 
  t=\sum_{n,k\geq0}\frac{1}{n+1}{2n\choose n}{2n+k\choose k}{3n+2k\choose n+k} t^k (v-1)^k\, \Rgf_1^{n+1}.
  \eeq
  This series has polynomial coefficients in $v$. Then the \gf\ of rooted planar quartic Eulerian orientations, counted by vertices (variable $t$) and clockwise oriented faces (variable $v$) is
  \[
    \Qgf_1%(t,v)
    = - v+\frac{1}{t^2}\sum_{{n,k\geq0},{~n+k>0}}\frac{1}{n+1}{2n\choose n}{2n+k\choose k}{3n+2k-1\choose 2n+k} t^k(v-1)^k\,\Rgf_1^{n+1}.
  \]
  This series is differentially algebraic, that is, satisfies (non linear) differential equations in $t$ and~$v$  (of order~$3$ and $4$, respectively).
\end{Theorem}
The first coefficients of $\Rgf_1$ and $\Qgf$ are:
\begin{align*}
  \Rgf_1&=t -\left(2 v +1\right) t^{2}-2\left( v^{2}+4 v +1\right) t^{3}-\left(4 v^{3}+36 v^{2}+56 v +9\right) t^{4}
          +\LandauO(t^5),
  \\
  \Qgf_1&= v \left(v +3\right) t +v \left(v +6\right) \left(2 v +3\right) t^{2}
          +v \left(v +1\right) \left(5 v^{2}+61 v +135\right) t^{3}
          % + v \left(14 v^{4}+279 v^{3}+1452 v^{2}+2455 v +1134\right) t^{4}
          +\LandauO(t^4).
\end{align*}
The four quartic orientations with only one vertex are shown on the right of Figure~\ref{fig:small}. Again, the expression of $\Qgf_1$ simplifies  when $v=1$, this time to
  \[
    % \Qgf_1= - 1+\frac{1}{3t^2}(t-\Rgf_1),
    3 t^2(\Qgf_1+ 1)=t-\Rgf_1,
  \]
  which is the expression given in~\cite{mbm-aep1}. We will see that a natural counterpart %of this
  for $v$ generic is
  \beq\label{Q1-alt}
  3 t^2(\Qgf_1+ v)=t-\Rgf_1 + \sum_{n,k\ge 0} \frac 1{n+1}{2n\choose n}{2n+k\choose k}{3n+2k+1\choose n+k+1} t^{k+1} (v-1)^{k+1}\, \Rgf_1^{n+1}.
  \eeq
  We give in Proposition~\ref{prop:trees1} a simple combinatorial interpretation in terms of trees of the series $t-\Rgf_1$ and $t^2(\Qgf_1+v)$. A bijection remains to be found.

\smallskip

Our third main  result (case $v=1$) involves the following (slightly modified) theta function:
\begin{align}
  {\theta(z,q)}\equiv\theta(z)&:= \sum_{n\ge 0} (-1)^n q^{n(n+1)/2} \sin (2n+1)z \label{theta-def-init}
  \\
           &= \sin z \prod_{n\ge 1} (1-q^n) \left(1-q^n e^{-2iz}\right) \left(1-q^n e^{2iz}\right). \nonumber
\end{align}
It is convenient to   write $\omega=-2\cos(2\alpha)$, as was done in earlier papers~\cite{elvey-zinn,kostov}. This does not mean that $\om$ takes some particular value between $-2$ and $2$. The variable $\om$ remains an indeterminate, but we exploit the fact that any rational function of $\cos \alpha$ and $\sin \alpha$ that is an even function of $\al$ and has $\pi$ as a period is in fact a rational function in $\om$.

\begin{Theorem}[{\bf Thms. 1.1 and 6.1 in~\cite{elvey-zinn}}]
  \label{thm:allomega}
  Write $\omega=-2\cos(2\alpha)$ and   let $q(t,\omega)\equiv q$
  % = t+ \left( 6\,\omega+6 \right) {t}^{2}+\cdots $
  be the unique formal power series in $t$ with constant
  term $0$ satisfying
  \beq\label{t-q}
  t=  \frac{\cos\alpha}{64\sin^3\!\alpha}
  \left(
    \frac{\theta''(\alpha)}{\theta'(\alpha)} -\frac{\theta(\alpha)\theta^{(3)}(\alpha)}{\theta'(\alpha)^2}
  \right).
  \eeq
  Moreover, define the series $\widetilde \Rgf$ %(t,\omega)
  by
  \[
    \widetilde \Rgf%(t,\omega)
    =\frac{\cos^2\alpha}{96\sin^4\! \alpha}\,
    \frac{\theta(\alpha)^2}{\theta'(\alpha)^2}
    \left(
      \frac{\theta^{(3)}(0)}{\theta'(0)}
      -\frac{\theta^{(3)}(\alpha)}{\theta'(\alpha)}\right).
  \]
  Then the generating function of  rooted  planar quartic Eulerian orientations, counted by
  vertices, with a weight $\omega$ per alternating vertex is
  \[
    \widetilde\Qgf%(t,\omega)
    = \frac{t-\widetilde \Rgf}{(\omega+2)t^2}-1 .
  \]
  The series $\widetilde\Qgf$ is differentially algebraic in $t$ and $\om$. Moreover, when $\alpha \in \pi\qs$ and $\om\not\in\{2,-2\}$, one can write $t$ as a D-finite series in $\widetilde \Rgf$.
\end{Theorem}

The series $q$, $\widetilde \Rgf$ and $\widetilde\Qgf$ start as follows:
\begin{align}
  q&=t+\left(6+6 \omega \right) t^2 +\left(45 \omega^{2}+84 \omega +48\right) t^{3}+\left(378 \omega^{3}+998 \omega^{2}+1076 \omega +436\right) t^{4}+\LandauO(t^5), \nonumber\\
  \widetilde \Rgf &=t -\left(\omega+2 \right) t^{2}-2 \left(\omega +2\right) \left(1+\omega \right) t^{3}-\left(\omega +2\right) \left(9 \omega^{2}+16 \omega +10\right) t^{4}+\LandauO(t^5), \label{R-exp}
  \\
  \widetilde\Qgf
   & =\left(2+2 \omega \right) t +\left(9 \omega^{2}+16 \omega +10\right) t^{2}+\left(54 \omega^{3}+132 \omega^{2}+150 \omega +66\right) t^{3} +\LandauO(t^4).\nonumber
\end{align}

Again, the coefficient of $t$ in $\widetilde \Qgf$ is consistent with Figure~\ref{fig:small}. Note that the series $\widetilde \Qgf$ of the latter theorem must coincide, when  $\om=1$, with the series $\Qgf_1$ of Theorem~\ref{thm:quartic} taken when $v=1$. This is far from obvious, but was proved in~\cite[Sec.~7]{elvey-zinn}. It was also proved there that the series~$\widetilde \Qgf /2$  coincides, when  $\om=0$, with the series $\Ggf$ of Theorem~\ref{thm:general}, taken when $v=1$. A bijective explanation is provided in Section~\ref{sec:def-orientations}.
Accordingly,
% Moreover,
the series $\widetilde \Rgf$ coincides with $\Rgf_0$ when $\om=0$ and $v=1$, and with $\Rgf_1$ when $\om=1$ and $v=1$.

\medskip
The above three lines of  the $(\om,v)$-plane, namely $\om=0$, $\om=1$ and $v=1$,  are  so far the only cases   where the solution of our three-variate problem is complete. We make some progress extending Theorem~\ref{thm:allomega} to general $v$ in Appendix~\ref{app:complex}, using complex analysis. However we are unable to complete the solution due to the non-linearity of a certain equation when $v\neq 1$ (see Remark \ref{remark:v_not_1}). As already mentioned, we derive these solutions from a characterisation of the trivariate series $\Qgf\equiv\Qgf(t,\om,v)$ counting quartic Eulerian orientations. Let us conclude this introduction by describing this characterisation,  which  is, as far as we know, of a new type.

Recall that, given a ring~$A$ and a variable $x$, the ring of \fps\ in $x$ (resp. Laurent series in $x$) with coefficients in~$A$ is denoted $A[[x]]$ (resp. $A((x))$).

\begin{Proposition}\label{prop:M-charac} 
  There exists a unique series in $t\qs(\om,v)((x))[[t]]$, denoted $\Mnn(x)$, such that:
  \begin{itemize}
  \item the coefficient of $x^{-1}$ in $\Mnn(x)$ equals $tv$,
  \item the series
    \[
      \Fnn(x):=(x\Mnn(x)-t(v-1))(1-\om x-\Mnn(x))
    \]
    belongs to $\qs(\om,v)[[x,t]]$ (that is, has no pole at $x=0$),
  \item finally, $\Mnn(\Mnn(x))=x$. 
  \end{itemize}
  The series
  \beq\label{Q-M}
  \Qgf:=\frac{1}{t^{2}}[x^{-2}]\Mnn(x) -v
  \eeq
  is the \gf\ of quartic Eulerian orientations, counted by vertices (variable $t$), alternating vertices (variable $\om$) and clockwise oriented faces (variable $v)$.
\end{Proposition}
The series $\Mnn(x)$ %$\Fnn(x)$
and $\Qgf$ start as follows:
\begin{align*}
  \Mnn(x)&=\left(\frac{v}{x}+\frac{1}{1-x}\right) t
           +\left(\frac{v}{x^{2}}+\frac{\omega  v +1}{\left(1-x \right)^{2}}+\frac{\omega}{\left(1-x \right)^{3}}\right) t^{2}
           +\LandauO(t^3),
  \\
  \Qgf&= v \left(\omega  v +\omega +2\right) t
        +  v \left(2 \omega^{2} v^{2}+5 \omega^{2} v +2 \omega^{2}+8 \omega  v +8 \omega +2 v +8\right) t^{2}+\LandauO(t^{3}).
\end{align*}
% \begin{multline*}
%   \Fnn(x)=\frac{1-\omega  x }{1-x} t\\
%   +\frac{\left(\omega v +  v +1 \right) (\omega x^{3}-1)    -\left(\omega^{2} v +\omega^{2}+4 \omega  v +\omega +v +1\right) x^{2}    +\left(4 \omega  v +\omega +2 v +2\right) x }{\left(1-x \right)^{3}} t^{2}+\LandauO(t^{3}),
% \end{multline*}%
% \left(\omega +\frac{\omega -1}{-1+x}\right) t +\left(-\omega^{2} v -\omega  v -\omega -\frac{\omega  \left(\omega  v -2 \omega -v +2\right)}{\left(-1+x \right)^{2}}+\frac{\omega  \left(\omega -1\right)}{\left(-1+x \right)^{3}}+\frac{-2 \omega^{2} v +\omega^{2}+\omega  v -2 \omega +v +1}{-1+x}\right) t^{2}+\mathrm{O}\! \left(t^{3}\right)
As we will see in Section~\ref{sec:charac}, the series $\Mnn(x)$ has a combinatorial interpretation and thus nonnegative coefficients (once expanded in $x$).
Its coefficients are in fact polynomials in $1/x$ and $1/(1-x)$. The series $\Rgf:=\Fnn(0)$ is the one that occurs (as $\Rgf_0$, $\Rgf_1$, and $\widetilde\Rgf$) in Theorems~\ref{thm:general},~\ref{thm:quartic} and~\ref{thm:allomega}. For generic values of $v$ and $\om$, it starts as follows:
  \beq\label{F0-ser}
    \Fnn(0)=t -\left(\omega  v +v +1\right) t^{2}-\left(\omega +1\right) \left(\omega  \,v^{2}+\omega  v +3 v +1\right) t^{3}+\LandauO \left(t^{4}\right).
  \eeq

We show that $t-\Fnn(0)$ has non-negative coefficients in Remark~\ref{rem:t-F0}.

\medskip
\noindent {\bf Outline of the paper.} We begin in Section~\ref{sec:prelim} with preliminary definitions on maps, orientations, and generating functions. We introduce several families of maps of interest, and recall the properties of the Ambj\o rn and Budd bijection~\cite{ambjorn-budd}. In Section~\ref{sec:funceq} we establish, through combinatorial constructions, various functional equations for the families of maps introduced in Section~\ref{sec:prelim}. They involve, in addition of the variables $t$, $v$ and $\om$, two additional ``catalytic'' variables $x$ and~$y$. These equations are implemented in an accompanying {\sc Maple} session available in~\cite{bmep-ref-arxiv}.
% on our webpages.
In Section~\ref{sec:charac} we derive from these equations the characterisation of $\Qgf$ given in Proposition~\ref{prop:M-charac} above, which only involves \emm one, catalytic variable. Based on this characterisation, Section~\ref{sec:01}  solves the cases $\om=0$ and $\om=1$, while Section~\ref{sec:six} solves the case $v=1$.
The latter solution is extended in Appendix~\ref{app:complex} where we establish, 
based on a complex analytic argument, properties of the function that should replace the theta function of Theorem~\ref{thm:allomega} in the general case $v\neq 1$. In Section~\ref{sec:Tutte} we give an alternative solution for the doubly specialised case $v=1$ and $\om=-2\cos(k\pi/m)$ (with $\om \neq -2,2$) and prove in particular that the above series $\Mnn(x)$ and $\Fnn(x)$ are then algebraic in~$x$ and $\Rgf:=\Fnn(0)$. This also proves the last statement of Theorem~\ref{thm:allomega}. In Section~\ref{sec:DA} we prove D-algebraicity in the three cases $\om=0$, $\om=1$ and $v=1$. Finally, in Section~\ref{sec:combin} we hint at what could be more combinatorial proofs of our results. In particular, we describe two simple families of trees that are counted by the series of Theorems~\ref{thm:general} and~\ref{thm:quartic}.
Additionally, in the special case where $\om=0$ and $v=1$, we re-derive the characterisation of $\Mnn(x)$ obtained in Section~\ref{sec:charac} using purely combinatorial arguments. We conclude in Section~\ref{sec:final} with final comments. 
 
%%%%%%%%%%%%%%%%%%%%%%%%%%%%%%%%%%%%%%%%%%%%%%%%%%%%%%%%%%% 
\section{Preliminaries}\label{sec:prelim}
%%%%%%%%%%%%%%%%%%%%%%%%%%%%%%%%%%%%%%%%%%%%%%%%%%%%%%%%%%% 
\subsection{Planar maps}\label{sec:defs}
% ==================================================
% 
A \emph{planar map} is a proper
embedding of a connected planar graph in the
oriented sphere, considered up to orientation preserving
homeomorphism. Loops and multiple edges are allowed
(Figure~\ref{fig:example-map}). The \emph{faces} of a map are the
connected components of  its complement. The numbers of
vertices, edges and faces of a planar map $M$, denoted by $\vv(M)$,
$\ee(M)$ and $\ff(M)$,  are related by Euler's relation
$\vv(M)+\ff(M)=\ee(M)+2$.
The \emph{degree} of a vertex or face is the number
of edges incident to it, counted with multiplicity. A \emph{corner} is
a sector delimited by two consecutive edges around a vertex;
hence  a vertex or face of degree $k$ {is incident to}  $k$ corners. 
% Alternatively, a corner can be described as an incidence between a
% vertex and a face. 
The \emph{dual} of a
map $M$, denoted $M^*$, is the map obtained by placing a 
vertex of $M^*$ in each face of $M$ and an edge of $M^*$ across each
edge of $M$; see Figure~\ref{fig:example-map}, right. A map is %said to be
\emph{quartic} if every vertex has degree~4. Duality transforms
quartic maps  into \emm quadrangulations,, that is, maps in which every face has
degree~4. A planar map  is \emm Eulerian, if every vertex has 
even degree. Its dual, with even face degrees, is then \emm
bipartite., We call  a face
of degree~2 (resp. 4) a \emm digon, (resp. \emm quadrangle,).

\begin{figure}[h]
  \centering
  \includegraphics{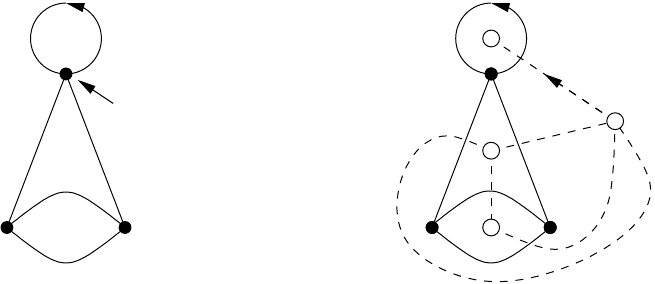}
  \caption{Left: a rooted planar map, with the root edge and root corner
    shown. Right: the dual map, in dashed edges.}
  \label{fig:example-map}
\end{figure}

For counting purposes it is convenient to consider \emm rooted, maps. 
A map is rooted by choosing an edge, called the root edge, and
orienting it. The starting point of this oriented edge is then the
\emm root vertex,, the other endpoint is the \emm co-root vertex,. The
face to the right of the root edge is the \emm
root face,. The face
to the left of the root edge is the \emm co-root face,. Equivalently, one can root the map by selecting a
corner. The correspondence between these two rooting conventions is
that the oriented root edge follows the
root corner in counterclockwise 
order around the root vertex.
In figures, we  usually choose the root face
as the infinite face (Figure~\ref{fig:example-map}). 
This explains why we often call the root face the \emm outer face,,
its degree the \emm outer degree, (denoted {$\od(M)$}), and its edges the \emm outer edges,. The other faces
are called \emm inner faces,. Similarly, we call the corners of the
outer face {\em outer corners} and all other corners {\em inner
  corners}.

From now on, every {map}
is \emph{planar} and \emph{rooted}, and these precisions will
often be omitted. Our convention for rooting the dual of a map is
illustrated on the right of Figure~\ref{fig:example-map}. Note that it makes duality of rooted
maps a transformation of order 4 rather than 2. By convention,
we include among rooted planar maps the \emph{atomic map}
having one vertex and no edge.

% ========================================  
\subsection{Orientations and labelled maps}
\label{sec:def-orientations}
% ========================================

An \emph{Eulerian  orientation} is a (rooted, planar) map in
which all edges are oriented, in such a way that at each vertex,
the in- and out-degrees coincide (Figure~\ref{fig:defs}, right). An \emph{Eulerian partial orientation} is a  map in
which \emm some, edges are oriented, with the same local condition. Note that
the underlying map of an Eulerian orientation must be Eulerian but this need not be the case for an Eulerian partial orientation. As in~\cite{mbm-aep1}, the orientation chosen for the root edge for an Eulerian orientation is required to be  consistent with its
orientation coming from the rooting, however we do not make this requirement for Eulerian partial orientations. As a result, for any fixed number of edges, the number of Eulerian partial orientations with no undirected edges is twice the number of Eulerian orientations.  A vertex is \emm alternating, if the in/out orientations of the incident edges that are oriented alternate. A face is \emm oriented clockwise, if all incident edges that are oriented are clockwise.

We find it convenient to work with
duals of Eulerian (partial) orientations, which turn out to be equivalent to
certain \emm (weakly) labelled maps,.

\begin{Definition}\label{def:labelled-map}
  A \emm labelled map, is a rooted planar map with integer labels on
  its vertices,  such that  adjacent labels differ by $\pm1$.  Such a map is necessarily
  bipartite. We also consider the atomic map, with a single unlabelled vertex,
  % (labelled $0$),
  to be a labelled map. We say that the map is \emm rooted from $\ell$ to $\ell'$, if $\ell$ (resp. $\ell'$) is the label of the root (resp. co-root) vertex.

  A \emm weakly labelled map, is a rooted planar map with integer labels on
  its vertices,  such that  adjacent labels differ by $0$, $\pm1$. An edge is \emm monochromatic, if the two incident   labels are equal. A face is \emm bicoloured, if it contains  only two distinct labels (possibly repeated). A vertex is a \emm local minimum, if no adjacent vertex has a smaller label.
\end{Definition}
A labelled map is shown in Figure~\ref{fig:labelled}. Note that in~\cite{mbm-aep1}, labelled maps were, by default, rooted from $0$ to $1$.

\begin{figure}[htb]
  \centering
  \scalebox{0.9}{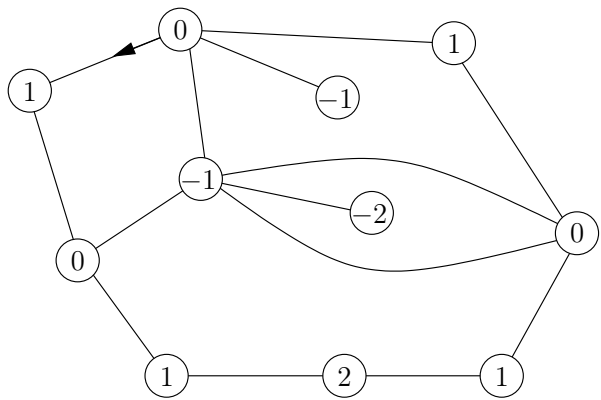}
  \caption{A labelled map, rooted from $0$ to $1$.}
  \label{fig:labelled}
\end{figure}

\begin{Lemma}\label{lem:duality}
  The duality transformation   can be extended into  a bijection $\Delta$ between Eulerian orientations
  and labelled maps rooted from $0$ to $1$, and   more generally, between Eulerian partial orientations and weakly labelled maps rooted from $0$ to $-1, 0$ or $1$.
  Classically, it transforms edges into edges, vertices into faces and vice-versa, while preserving their degrees:   if $\Delta(O)=M$,
  \[
    \ee(M)=\ee(O), \quad \ff(M)=\vv(O), \quad \vv(M)= \ff(O).
  \]
  Moreover,  the number of monochromatic edges, bicoloured faces, and local
  minima of $M$ can also be traced through the relations
  \[   
    \ee_{\mathrm{mon}}(M)=  \ee_\circ(O), \quad \ff_{\mathrm{bic}}(M)= \vv_{\mathrm{alt}}(O), \quad \vv_{\min}(M)= \ff_{\mathrm{clock}}(O),
  \]
  involving the number of  undirected edges, of alternating vertices  and of clockwise %(partially)
  oriented faces of $O$. We hope our notation to be self-explanatory.
\end{Lemma}

The map $\Delta(O)$ is obtained by labelling the faces of $O$ as described in
Figure~\ref{fig:duality}, and then taking the classical dual. In other words,  a partial Eulerian orientation of
edges of a map gives a height function on its faces, or equivalently,
on the vertices of its dual. Height functions on regular grids, like the  square lattice, are much studied as models of  discrete random surfaces, expected to  converge to the Gaussian free field~\cite{chandgotia,glazman-manolescu,duminil2024delocalization}.
This construction was already used in the case of  Eulerian orientations (all edges being oriented)
in~\cite[Prop.~2.1]{elvey-guttmann17} and in~\cite{mbm-aep1}.

\begin{figure}[htb]
  \centering
  \scalebox{0.9}{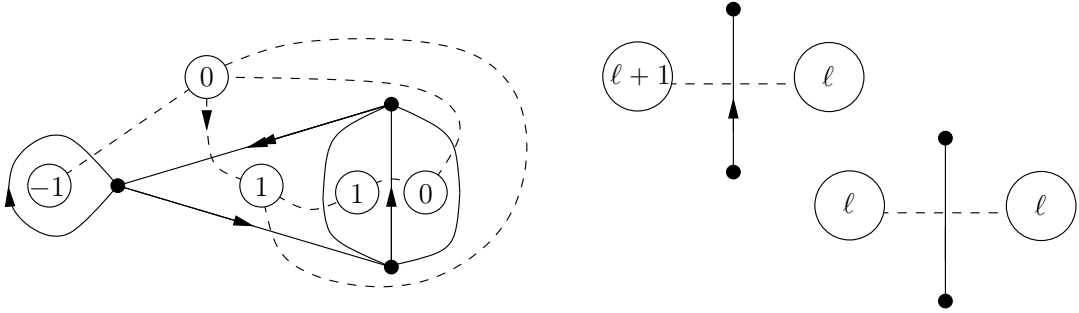}
  \caption{A partial Eulerian orientation (solid edges) and the corresponding
    dual weakly labelled map (dashed edges). The labelling rule is shown on
    the right.}     
  \label{fig:duality}
\end{figure}

Another, more subtle   bijection, due to Ambj\o rn and Budd~\cite{ambjorn-budd}, plays an important role in this work. It relates labelled quadrangulations (equivalently, Eulerian orientations of quartic maps) to weakly labelled maps. It creates an edge in each face, as described in Figure~\ref{fig:AB}. The following statement is Theorem~1 in~\cite{ambjorn-budd}, restricted to  labelled quadrangulations rooted from $0$ to $1$.

\begin{figure}[htb]
  \centering
  \scalebox{0.8}{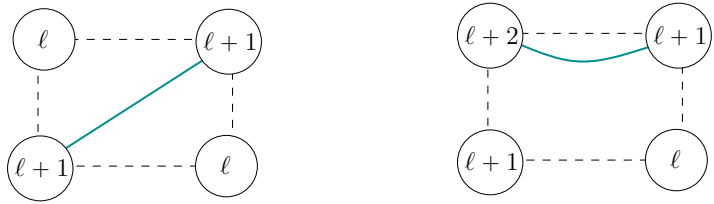}
  \caption{The Ambj\o rn and Budd construction.}     
  \label{fig:AB}
\end{figure}

\begin{Proposition} \label{prop:bij}
  There exists a bijection $\Phi$ {that sends} labelled quadrangulations rooted from $0$ to $1$ to  weakly labelled   maps 
  with root vertex labelled $1$. Moreover, if $\Phi(Q)=M$, then the
  number of edges, monochromatic edges, and faces in $M$ are given by
  \[
    \ee(M)=\ff(Q),\qquad \ee_{\mathrm{mon}}(M)=\ff_{\mathrm{bic}}(Q), \qquad \ff(M)=\vv_{\min}(Q),
  \]
  where $\vv_{\min}(Q)$ denotes the number of local minima in $Q$.
\end{Proposition}

Note that a simple shift in the labels transforms the weakly labelled maps rooted at $1$ of this proposition into those rooted at $0$ in Lemma~\ref{lem:duality}.

If $Q$ is a \emm colourful,\, labelled quadrangulation, i.e., a quadrangulation in which every face contains three distinct labels, then $\Phi(Q)$ is a labelled map, rooted either from $1$ to $2$, or from $1$ to~$0$. Each of these two classes is in bijection with general Eulerian orientations, by duality.

\begin{Corollary}\label{cor:AB}
  There is a $2$-to-$1$ correspondence $\overline\Phi$ between colourful labelled quadrangulations~$Q$  rooted from $0$ to $1$ and general Eulerian orientations $M$, such that, if $\overline\Phi(Q)=M$,
  \[
    \ee(M)=\ff(Q),  \qquad \vv(M)=\vv_{\min}(Q).
  \]
\end{Corollary}

On one occasion (proof of Proposition~\ref{prop:symmetry}) we will use a slight generalization of the bijection of Proposition~\ref{prop:bij}, which  applies to labelled maps in which all faces are quadrangles, except for the root face which has arbitrary (even) degree $2d$. It is a special case of a much more general construction that applies to all labelled maps~\cite[Sec.~2.2]{bouttier-fusy-guitter}.

The following table, illustrated by Figure~\ref{fig:bijections}, summarises the correspondence between statistics of interest under the duality map and the Ambj\o rn-Budd bijection. We will work mostly with labelled quadrangulations (column 2), and then state our results in terms of orientations using this dictionary.

\def\MC{\multicolumn{1}{c}{}}
\def\MCdl{\multicolumn{2}{c}{\ \hskip -9mm $\stackrel{\text{duality}}{\curvearrowleftright}$}}
\def\MCdr{\multicolumn{2}{c}{\ \hskip -1mm $\stackrel{\text{duality}}{\curvearrowleftright}$}}
\def\MCab{\multicolumn{2}{c}{\ \hskip 2mm $\stackrel{\text{Ambj\o rn-Budd}}{\curvearrowleftright}$}}

\medskip

\begin{table}[h]
  \centering
  \begin{tabular}{|c|c|c|c|c|c}
    \MC &  \MCab & \MC\\
    [\dimexpr -\normalbaselineskip-4pt]
    \MCdl  & \MCdr \\
    \hline
    quartic       &labelled        &weakly labelled &  Eulerian partial & variable\\
    orientations  &quadrangulations&maps            &orientations&\\
    \hline \hline
    vertices  &faces & edges & edges & $t$\\
    \hline
    alternating &bicoloured & monochromatic& undirected& $\om$\\
    vertices  &faces & edges&edges&\\
    \hline
    clockwise  & local & faces & vertices&$v$\\
    faces         & minima&&&\\
    \hline
  \end{tabular}
  \medskip
  \caption{Parameter correspondences between orientations and labelled maps. For the objects to be in bijection, we restrict to labelled quadrangulations rooted from $0$ to $1$ and weakly labelled maps rooted from $0$ to $-1$, $0$ or $1$.}
  \label{tab:bij}
\end{table}

\begin{figure}[htb]
  \centering
  \scalebox{0.4}{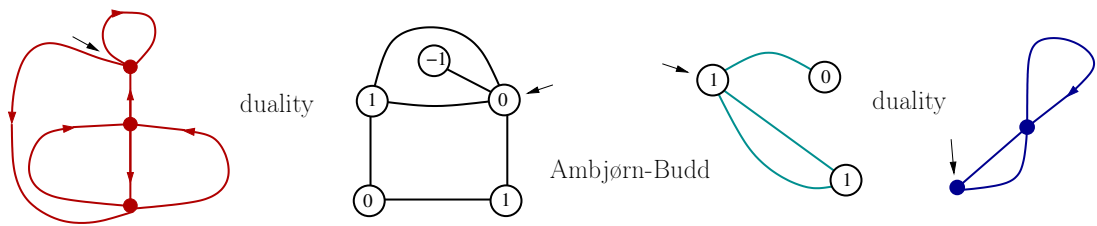}
  \caption{A quartic Eulerian orientation and the corresponding labelled quadrangulation, weakly labelled map and  Eulerian partial orientation. The root corners are indicated by black arrows.}     
  \label{fig:bijections}
\end{figure}

\medskip

  \begin{Remark}[{\bf 3-coloured quadrangulations}]
  As observed in~\cite[Sec.~2.2]{mbm-aep1}, there is a simple bijection between labelled quadrangulations and quadrangulations equipped with a proper $3$-colouring of the vertices. It simply consists in taking the labels modulo $3$. Hence our results can also be stated in these terms. In particular, the specialization $\Qgf(t,\om, 1)$ counts $3$-coloured quadrangulations by faces and bicoloured faces (the colours of the root edge being fixed). Even at $\om=1$, this series is transcendental, as proved in~\cite{mbm-aep1}. This contrasts with the classes of $3$-coloured triangulations and $3$-coloured general maps, which have algebraic \gfs~\cite{bernardi-mbm-alg}.
\end{Remark}

\begin{Remark} We note that labelled quadrangulations with no bicoloured faces (i.e., $\omega=0$) are in bijection with rectilinear disks recently studied by  Budd \cite{budd-flat}. 
\end{Remark}

% =================================================================
\subsection{Patches}
\label{sec:patches}
% =================================================================

We first recall the definitions of three families of labelled maps already used in~\cite{mbm-aep1}.

\begin{Definition}\label{def:patches}
  A {\em patch} is a labelled map, rooted from $0$ to $1$, in which each inner face has degree~$4$,
  and the vertices around the outer face are alternately labelled $0$ and
  $1$. We include the atomic map in the set of patches. For $\ell \in \zs$, an \emm $\ell$-patch, is a labelled map obtained by adding $\ell$ to all labels of a patch.

  A \emm C-patch, is a patch satisfying two additional
  conditions:  all neighbours of the root vertex are labelled
  $1$, and the root corner is the only outer corner at the root
  vertex. By convention, the atomic patch is \emm not, a C-patch.

  {\em D-patches} resemble patches but may include digons.  More
  precisely, a D-patch is a labelled map, rooted from $0$ to $1$, in which each inner face 
  has degree $2$ or $4$, those of degree $2$ being incident to
  the root vertex, and the vertices around the outer face
  are alternately labelled $0$ and~$1$. We also require that all
  neighbours of the root vertex are labelled~$1$. We include the atomic map in the set of D-patches.
\end{Definition}

We now introduce a new family of labelled maps, which are reminiscent of lattices or maps with \emm Dobrushin boundary conditions, considered in the Ising model~\cite{chen-turunen1,eynard-book,velenik}.

\begin{Definition}\label{def:E-patches}
  An \emm E-patch, is a {labelled map} such that each inner face has degree $4$,  the outer face has degree say $2m$, and there exists $k \in \llbracket 0, m\rrbracket$
  such that the  outer corners, taken  in counterclockwise order starting at the root corner,  have labels $0, \ell_{1}, 0, \ell_3, \ldots,0, \ell_{2m-1}$
  where $\ell_{2j+1}=1$ for $0\leq j< k$ and %the label
  $\ell_{2j+1}=-1$ for $k\leq j<m$. We include the atomic map in the set of E-patches.
\end{Definition}
If an E-patch satisfies $k=m$ in this definition then it is a patch, while if it satisfies $k=0$ then it is, up to a shift of the root edge, a $(-1)$-patch.

% =========================================
\subsection{Formal power series}
\label{sec:fps}
% =========================================

Let $A$ be a commutative ring and $x$ an indeterminate. We denote by
$A[x]$ (resp. $A[[x]]$) the ring of polynomials (resp. \fps) in $x$
with coefficients in $A$. If $A$ is a field,  then $A(x)$ denotes the field
of rational functions in $x$, and $A((x))$ the field of Laurent series in $x$, that is, series of the form
\[ \sum_{n \ge n_0} a_n x^n,
\]
with $n_0\in \zs$ and $a_n\in A$.  The coefficient of $x^n$ in a   series $F(x)$ is denoted by $[x^n]F(x)$.

This notation is generalised to polynomials, fractions
and series in several indeterminates. For instance, the  \gf\ of
Eulerian orientations, counted by edges (variable $t$) and vertices
(variable $v$)  belongs to
$\qs[v][[t]]$.
% For a multivariate series, say $F(x,y) \in \qs[[x,y]]$, the notation $[x^i]F(x,y)$ stands for the \emm series, $F_i(y)$ such that $F(x,y)=\sum_j F_j(y) x^j$. It should not be mixed up with the coefficient of $x^iy^0$ in $F(x,y)$, which we denote by $[x^i y^0]F(x,y)$. 
If $F(x,x_1, \ldots, x_d)$ is a series in the $x_i$'s whose 
coefficients are Laurent series in $x$, say
\[
  F(x,x_1, \ldots, x_d)= \sum_{i_1, \ldots, i_d} x_1^{i_1} \cdots
  x_d^{i_d}
  \sum_{n \ge n_0(i_1, \ldots, i_d)} a(n, i_1, \ldots, i_d) x^n,
\]
then we define the \emm non-negative part of $F$ in $x$, as the
following \fps\ in $x, x_1, \ldots, x_d$:
\[
  [x^{\ge 0}]F(x,x_1, \ldots, x_d)= \sum_{i_1, \ldots, i_d} x_1^{i_1} \cdots
  x_d^{i_d}
  \sum_{n \ge 0} a(n, i_1, \ldots, i_d) x^n.
\]
We define similarly the \emm positive part, of $F$ in $x$, denoted
$[x^{>0}]F$. We will use, more generally, the notation $[x^{>i}]F$ for any integer $i$ when extracting coefficients with a power of $x$ larger than $i$.

If $A$ is a field,  a power series $F(x) \in A[[x]]$
is \emm algebraic, (over $A(x)$) if it satisfies a
non-trivial polynomial equation $P(x, F(x))=0$ with coefficients in
$A$. It is \emm differentially algebraic, (or \emm D-algebraic,) if it satisfies a non-trivial polynomial
differential equation $P(x, F(x), F'(x), \ldots, F^{(k)}(x))\\ =0$ with
coefficients in $A$. It is \emm D-finite, if it satisfies a \emm
linear, differential equation with coefficients  in $A(x)$. For
multivariate series, D-finiteness and D-algebraicity require the
existence of a differential equation \emm in each variable,.  {We refer to~\cite{Li88,lipshitz-df} for
  general results on D-finite series, and to~\cite[Sec.~6.1]{BeBMRa17} for
  D-algebraic series.}

%======================================================================
\subsection{Four generating functions}
\label{sec:GF}
%======================================================================
We will count the four families of labelled maps from Section~\ref{sec:patches} (patches, C-patches, D-patches and E-patches) according to three main statistics corresponding to the second column of Table~\ref{tab:bij}:
\begin{itemize}
\item the number of inner quadrangles (variable $t$)
\item the number of inner bicoloured quadrangles (variable $\om$)
\item the number of vertices that are local minima, including the unique vertex in the case of the atomic map (variable $v$).
\end{itemize}
Moreover, we will also record, depending on the class of maps under consideration, one or two additional parameters, which are required to write functional equations (such parameters are sometimes called \emm catalytic,). More precisely, for all types of patches defined above, a variable $y$ records the number of outer corners labelled $1$. For patches, C-patches and D-patches this is also the number of outer corners labelled $0$, or equivalently half the outer degree. Finally, another variable $x$ records either the degree of the root vertex (for C-patches)  or the number of
inner digons (for D-patches), or the number of outer corners labelled $-1$ (for E-patches).

We denote by $\Pgf(y)$, $\Cgf(x,y)$, $\Dgf(x,y)$ and $\Egf(x,y)$ the corresponding \gfs. Despite the notation, these also depend on $t$, $\om$, and $v$. More precisely, if we denote by $\GK$ the field $\qs(\om,v)$ of rational functions in $\om$ and $v$, then
\begin{itemize}
\item $\Pgf(y)$ belongs to $\GK[[y,t]]$,
\item $\Cgf(x,y)$ belongs to $\GK[x][[y,t]]$,
\item $\Dgf(x,y)$ and $\Egf(x,y)$
  belong to $\GK[[x,y,t]]$.
  % \item $\Egf(x,y)$ belongs to $\GK[x][[y,t]]$.
\end{itemize}
This is justified in~\cite[Sec.~3]{mbm-aep1} for the first three series. For the fourth one, namely $\Egf(x,y)$, we simply note that there are finitely many labelled maps with $n$ inner faces, all of degree $4$, and an outer face of bounded degree. The outer degree is twice the sum of the exponents of $x$ and~$y$ in the series $\Egf(x,y)$. Finally, note that we could replace the field $\GK=\qs(\om,v)$ above by the polynomial ring $\qs[\om,v]$ (or even $\mathbb{Z}[\om,v]$). 

%%%%%%%%%%%%%%%%%%%%%%%%%%%%%%%%%%%%%%%%%%%%%%%%%%%%%%%%%%% 
\section{ Functional equations}
\label{sec:funceq}
%%%%%%%%%%%%%%%%%%%%%%%%%%%%%%%%%%%%%%%%%%%%%%%%%%%%%%%%%%% 

We first introduce an important decomposition of labelled maps,
namely the extraction and contraction of a \emm subpatch,. This type of decomposition already appears in~\cite{elvey-guttmann17} and~\cite{mbm-aep1}, but  we use it on different maps here.

% ==================================================
\subsection{Subpatch extraction}
% ==================================================

\begin{Definition} \label{def:sub}
  Let $M$ be a  labelled map with root vertex $v$ labelled $\ell$ and root corner $c$.
  Let $M'$ be the maximal (connected)  submap of $M$ that contains  $v$ and consists of vertices
  labelled $\ell$ or less. Let $S$ be the submap of $M$ {that} contains $M'$
  and all edges and vertices within its boundary (assuming the
  root face of $M$ is drawn as the infinite face). The map $S$ is the {\em %minus-
    subpatch} of $M$. We root it at the corner inherited from $c$.
 
  We define the {\em contracted map} $C$ of $M$ to be the map remaining from $M$ when all of $S$ is contracted to a single vertex labelled $\ell$. The root corner of $C$ is inherited from $c$ (see Figure~\ref{fig:minus_subpatch_and_contraction}).
\end{Definition}

\begin{figure}[htb]
  \centering
  \scalebox{0.7}{\includegraphics{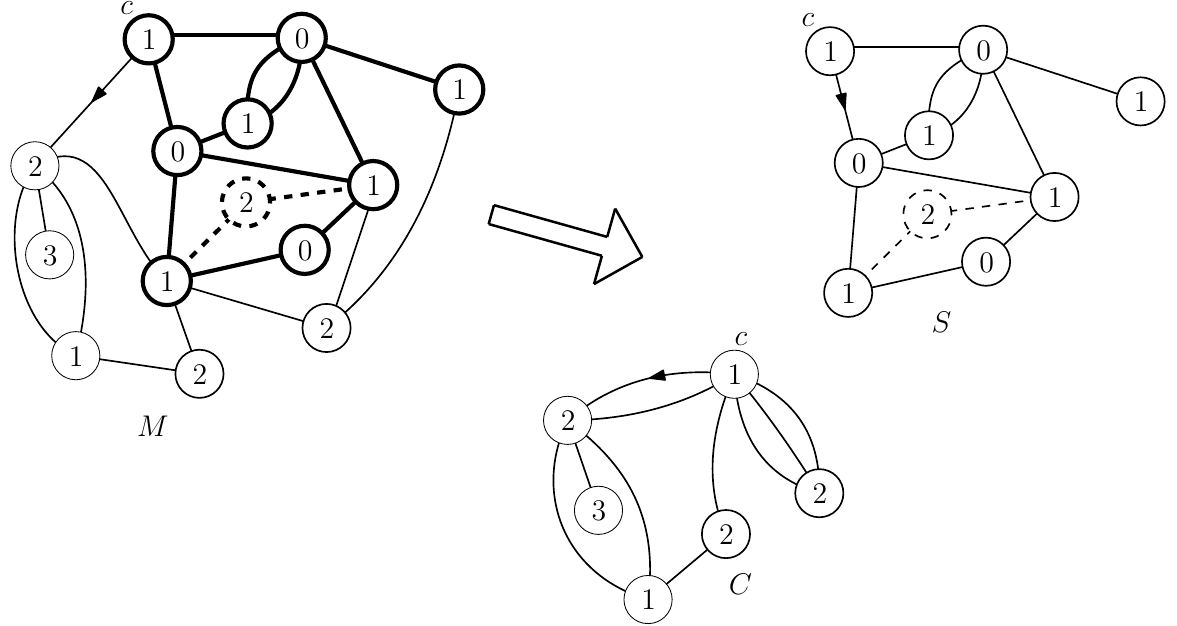}} 
  \caption{A labelled map $M$ with root corner $c$ labelled $\ell=1$, along with the corresponding %minus-
    subpatch $S$ and contracted map $C$. The map $M'$ is formed by the solid edges of $S$.}
  \label{fig:minus_subpatch_and_contraction}
\end{figure}

A very special case is when $v$ is a local minimum, that is, when all its neighbours have label $\ell+1$: then $S$ is the atomic map, and $C$ is $M$. Note that we use the word \emm subpatch, for $S$ even though $S$ may not be a patch in the sense of Definition~\ref{def:patches}. 

We will apply the subpatch extraction/contraction to maps $M$ satisfying additional properties, under which the correspondence $M \mapsto (S,C)$ becomes one-to-one. For the moment, let us study some properties of this transformation.

\begin{Proposition}\label{prop:subpatch}
  Let $M$ be a labelled map  in which  all inner faces are either quadrangles or digons, and  all inner digons lie in the subpatch $S$ of $M$. 
  Assume moreover that the root vertex~$v$ has label $\ell$, and that the outer corners satisfy the (shifted) Dobrushin condition: starting from the root corner $c$ and moving in counterclockwise order around $M$, the corners have labels $\ell, \ell_1, \ell, \ell_3, \ldots , \ell, \ell_{2m-1}$ where $\ell_{2j+1}=\ell+1$ for $0\le j <k$ and $\ell_{2j+1}=\ell-1$ for $k\le j <m$, for some $k\in \llbracket 0,m\rrbracket$.

  Then all inner faces of the subpatch $S$ and of the contracted map $C$ have degree $2$ or $4$. All digons of $C$ are incident to the root vertex. The outer corners of $S$ are labelled alternately $\ell$ and $\ell -1$, with $\ell$ at the root corner. The outer corners of  $C$ are labelled alternately $\ell$ and $\ell +1$, with~$\ell$ at the root corner. The root vertex of $C$ is only adjacent to vertices labelled $\ell+1$.

  Moreover, the number of inner digons, of inner quadrangles, and of bicoloured inner quadrangles of $M$  are given by
  \[
    \id(M)=\id(S), \quad    \iq(M)=\iq(S)+\iq(C) + \id(C),  \quad \iq_{\mathrm{bic}}(M)= \iq_{\mathrm{bic}}(S)+\iq_{\mathrm{bic}}(C),
  \]
  the number of  local minima of $M$ is
  \[
    \quad \vv_{\min}(M)= \vv_{\min}(S)+\vv_{\min}(C)-1,
  \]
  and finally the numbers of outer corners labelled $\ell+1$ and $\ell-1$ in $M$ are 
  \[
    \oc_{\ell+1}(M) =  \oc_{\ell+1}(C),\quad \oc_{\ell-1}(M)=\oc_{\ell-1}(S)-\id(C).
  \]
  Again, we hope the notation to be self-explanatory.
  % where $\dig(C)$ denotes the number of digons of $C$.  
\end{Proposition}
\begin{proof}
  With use the notation of Definition~\ref{def:sub}. It follows immediately from the definition of $M'$ and $S$ that they share
  the same outer face. Moreover, all inner faces of $S$ are also inner faces
  of~$M$. Hence they have degree $2$ or $4$. All outer vertices of $S$ must also be vertices of
  $M'$, so their labels are no more than $\ell$, by definition of $M'$.

  Let us prove that  the outer labels of $S$
  cannot be less than  $\ell-1$. For any outer vertex $u$ of
  $S$, there is some face $f$ of $M$,  containing $u$, that is not
  a face of $S$. If $f$ is
  the outer face of $M$, then by the Dobrushin assumption, the label of $u$ is $\ell$ or $\ell-1$. Otherwise $f$ is an inner quadrangle of~$M$,
  so it must contain a vertex $u'$
  with label at least $\ell+1$ (otherwise $f$ would be contained in~$S$). Since
  $u$ and $u'$ are both incident to face $f$,  and $u$ has  label at most $\ell$, this  label can only be $\ell$ or~$\ell-1$. Hence the outer
  vertices of $S$ are  labelled $\ell$ or $\ell-1$. By definition of labelled maps, the labels alternate on the outer face of $S$.

  Now, every edge in $M$ that connects a vertex in $S$
  to a vertex not in $S$ must have endpoints labelled $\ell$ (in $S$) and $\ell+1$
  (not in $S$).
  When contracting $S$ into a single (root) vertex, all neighbours of this vertex thus have label $\ell+1$.  Observe that all outer vertices of $M$ labelled $\ell-1$ belong to the subpatch $S$, and end up being contracted in $C$. Moreover, all  outer vertices of $C$  were outer vertices in $M$. They thus have label $\ell$ or $\ell+1$, alternately.

  Let us now discuss the inner face degrees in $C$. By assumption, before the contraction of $S$ they are all equal to $4$. When $S$ is contracted into a single vertex, the inner quadrangles of $M\setminus S$ that contain an outer corner of $S$ labelled $\ell-1$ become digons. Before the contraction they have labels $\ell, \ell-1, \ell, \ell+1$, and the vertex labelled $\ell+1$ is the only one that is not in $S$. Conversely, any outer corner of $S$ that has label $\ell-1$ and is not an outer corner of $M$ gives rise to a digon of $C$.

  Let us now examine the quantitative statements of the proposition. The first three follow from the the above discussion on face degrees, since the quadrangles that get contracted into a digon are not bicoloured. For the fourth one, we observe that the root vertex of $C$, which results from the  contraction of $S$, is a local minimum in $C$, but not in $M$, unless $S$ is the atomic map. But in this case (and in this case only), this vertex is counted as a minimum of $S$, and the result holds as well. The fifth statement results from the discussion on the outer corners of $C$. The final one translates the fact that all outer corners labelled $\ell-1$ in $S$ either come from an outer corner of $M$ or give rise to a digon in $C$.
\end{proof}

\begin{Proposition}\label{prop:subpatch-bij} 
  The transformation $M \mapsto (S,C)$ of Definition~\ref{def:sub}, restricted to labelled maps satisfying the conditions of Proposition~\ref{prop:subpatch} (first paragraph), forms a bijection from these maps to pairs $(S,C)$  of labelled maps satisfying the conditions of Proposition~\ref{prop:subpatch} (second paragraph), and the additional condition $\oc_{\ell-1}(S)\ge \id(C)$. 
\end{Proposition}
\begin{proof}
  We will argue that:
  \begin{enumerate}
  \item   we can reconstruct $M$ from its subpatch $S$ and contracted map $C$,
  \item  this construction, applied to any pair $(S,C)$ satisfying the above conditions, gives a labelled map $M$ that also satisfies the conditions of the proposition,
  \item  finally, the subpatch and contracted map of  $M$ are indeed $S$ and $C$.
  \end{enumerate}
  Let us begin with  the first point, which is the key one. There are two special cases: if $M$ has no outer corner labelled $\ell+1$, then $S=M$ and $C$ is the atomic map; if the root vertex of $M$ is only adjacent to vertices labelled $\ell+1$, then $S$ is the atomic map and $C=M$. These are the only cases where $C$ or $S$ is atomic.

  So let us assume that neither $C$ nor $S$ is atomic. Let $d$ be the number of inner digons of $C$. Proposition~\ref{prop:subpatch} tells us that $d$ is not larger than the number of outer corners labelled $\ell-1$ in $S$. Observe that if we label these outer corners $c_1, \ldots, c_j$, 
  in counterclockwise order starting from the root, then those that belong to an inner quadrangle of $M\setminus S$ (and give rise to a digon of $C$) are $c_1, \ldots, c_d$. This observation tells us how to reconstruct $M$ from $S$ and $C$. Let us still call~$v$ the root vertex of $C$.  Let $e_1, e_2, \ldots, e_k$ be the 
  edges attached to $v$ in $C$, in counterclockwise order around $v$,
  starting from the root corner (Figure~\ref{fig:decontract}).
  We now  erase the vertex $v$ from $C$, so that
  the %half-
  edges $e_i$ are dangling. We connect them to the 
  outer corners of $S$ labelled $\ell$ in the following way: we first attach $e_1$ to
  the root corner of $S$, and then proceed  counterclockwise around
  $S$, connecting $e_{i+1}$ to the next corner of $S$ labelled
  $\ell$ if  $e_i$ and $e_{i+1}$ form  an inner digon of $C$, and to
  the same corner as $e_i$ otherwise. We thus reconstruct $M$ from $S$ and $C$, proving Point (1).

  Let us address Point (2). We start from $S$ and $C$ satisfying the conditions of the proposition, and try to apply the above construction. The condition $\oc_{\ell-1}(S)\ge \id(C)$ implies that we will not run out of outer corners labelled $\ell-1$ in $S$, hence the construction terminates and gives a map $M$. Moreover, all inner digons of $C$  give rise to quadrangles, so all inner faces of $M\setminus S$ are quadrangles. The other inner faces (of $C$ and $S$) keep their degrees. Since all neighbours of the root vertex of $C$ were labelled $\ell+1$, the only new edges that we have created have label $\ell+1$ (on the $C$ side) and $\ell$ (on the $S$ side), so that $M$ is indeed a labelled map. Finally, the outer corners of $M$ that have label $\ell-1$ are, with the above notation, the outer corners $c_{d+1}, \ldots, c_j$ of $S$. They are the last outer corners with label $\ell\pm 1$ around $M$ in counterclockwise order. The other such corners were outer corners of $C$ and have label $\ell+1$. This proves that the shifted Dobrushin condition holds.

  Point (3) should now be clear, because in the construction of $M$, we have separated $S$ from~$C$ by edges labelled $(\ell, \ell+1)$.
\end{proof}

\begin{figure}[htb]
  \centering
  \scalebox{0.7}{\includegraphics{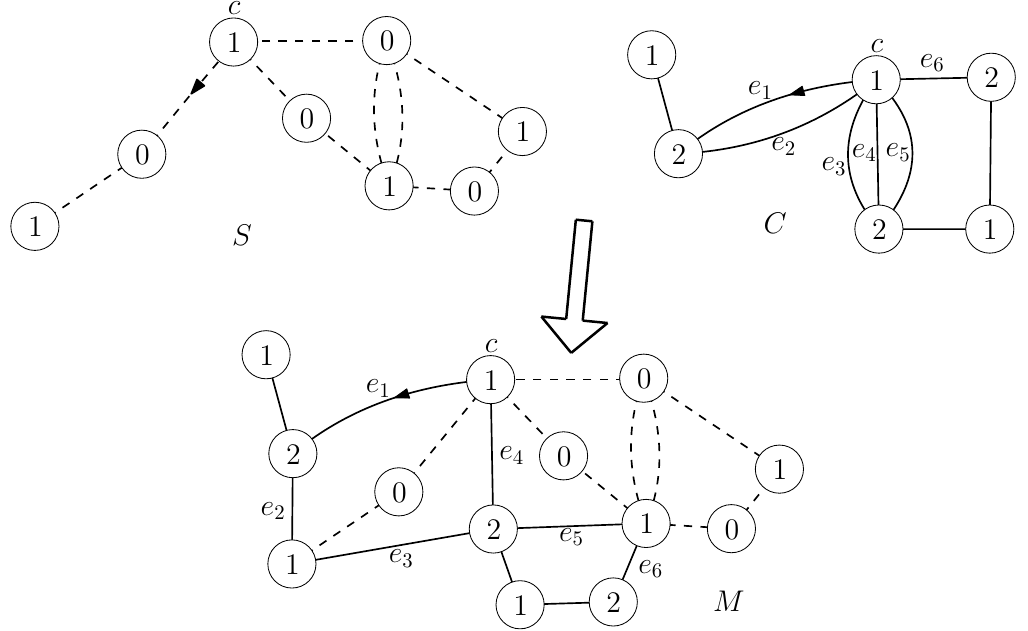}}
  \caption{How to reconstruct the labelled map $M$ from the
     subpatch $S$ (dashed edges) and the contracted map
    $C$. Here, the label of the root vertex $v$ is $\ell=1$, the subpatch~$S$ has outer degree $2j=10$, there are $d=3$ inner digons in $C$, and the root vertex  of $C$  has degree $k=5$.}
  \label{fig:decontract}
\end{figure}

% ============================================================
\subsection{Refinement of earlier equations}
% ============================================================
Here, we examine to what extent the system of four equations established in~\cite[Thm.~3.1]{mbm-aep1} for $\om=v=1$ can  be refined so as to include the variables $\om$ and $v$, that is, the number of bicoloured faces and local minima. It turns out that only the final equation, giving the coefficient of $y$ in $\Dgf(x,y)$ must be significantly modified.  We will only use one equation of this system,
% not use this system,
but we found it natural to discuss to what extent this earlier approach generalises.

\begin{Proposition}\label{prop:early-systemQ}
  The \gfs\ $\Pgf(y)$,
  $\Cgf(x,y)$ and $\Dgf(x,y)$ defined in Section~\ref{sec:GF} satisfy the following system of equations:
  \begin{align}
    \Pgf(y)&=v-1+\frac{1}{yv}[x^1]\Cgf(x,y),\nonumber\\
    \Dgf(x,y)&=\frac{v}{1-\frac 1 v \Cgf\left(\frac{1}{1-x},y\right)},\nonumber\\
    \Dgf(x,y)&=v+ {y}\, [x^{\geq0}]\left(\Dgf(x,y)\left(\frac 1 v [y^{1}]\Dgf(x,y)+\frac{1}{x}\Pgf\left(\frac{t}{x}\right)\right)\right),\label{eq:DP1}
  \end{align}
  with the initial condition
  \[
    (1-x)  [y^{1}]\Dgf(x,y)= v+\om t [y^2] \Dgf(x,y)+ \frac t v [y^1]\left( \Dgf(x,y)[z^1]\Dgf\!\left(\frac t y,z\right)\right).
  \]
  Moreover, these four equations have a unique solution if we require that $\Pgf$, $\Cgf$ and $\Dgf$ belong  respectively to $\GK[[y,t]]$,
  $\GK[x][[y,t]]$ and $\GK[[x,y,t]]$, where we recall that $\GK=\qs(\om, v)$.

  The \gf\ $\Qgf\equiv\Qgf(t,\om,v)$ that counts (by faces, bicoloured faces and local minima) labelled quadrangulations rooted from $0$ to $1$ is
  \beq\label{Q-P-sol}
  \Qgf=[y^1]\Pgf(y)-v.
  \eeq
  This is also the \gf\ of quartic Eulerian orientations, counted by vertices ($t$), alternating vertices ($\om$) and clockwise faces ($v$).
\end{Proposition}

\begin{proof}[Proof of Proposition~\ref{prop:early-systemQ}]
  The proof of uniqueness follows the same lines as in~\cite[Thm.~3.1]{mbm-aep1}. More precisely, we prove that one can compute by induction on $N\ge0$ the coefficients of all monomials $y^j t^n$ in $\Pgf$ (resp. $\Cgf$, $\Dgf$) such that  $j+n<N$ (resp. $j+n\le N$ for $\Cgf$ and $\Dgf$). For instance, when $N=0$ there is nothing to say about $\Pgf$, the third equation gives $[y^0]\Dgf=v$ (so that in particular, $[y^0t^0]\Dgf=v$), and then the second equation gives $[y^0]\Cgf=0$.

  It must be noted that since $\Dgf(x,y)=v+\LandauO(y)$,  the third equation gives a tautology if we extract the coefficient of $y$. This is why the initial condition is needed.

  Now let us prove that the series  $\Pgf(y)$, $\Cgf(x,y)$ and $\Dgf(x,y)$ defined in Section~\ref{sec:GF} satisfy the system.
  The proof of the first three equations mimics what was done in~\cite{mbm-aep1}. The only slight difference arises in the proof of the third equation, where we have to count the so-called \emm minus-patches,,  defined in~\cite{mbm-aep1} as the maps obtained by replacing each label $\ell$ by $-\ell$ in a patch.  This transformation does not preserve the number of local minima, but we can alternatively describe   minus-patches as the maps obtained by replacing each label $\ell$ by $\ell-1$ in a  patch, and moving the root edge one step in clockwise order around the outer face (see Figure~\ref{fig:patch_minuspatch}). This description proves that the \gf\ of minus-patches is $\Pgf(y)$. 

  Finally, the fourth equation giving the initial condition $[y^1]\Dgf(x,y)$ is \emm not, a refinement of the initial condition used in~\cite{mbm-aep1}, which  seems to be incompatible  with both $\om$ and $v$. The new initial condition is obtained by extracting the coefficient of $y^1$ in a new equation that we establish below, namely~\eqref{D-eq-new}.
\end{proof}

\begin{figure}[ht]
  \setlength{\captionindent}{0pt}
  \includegraphics[scale=0.7]  {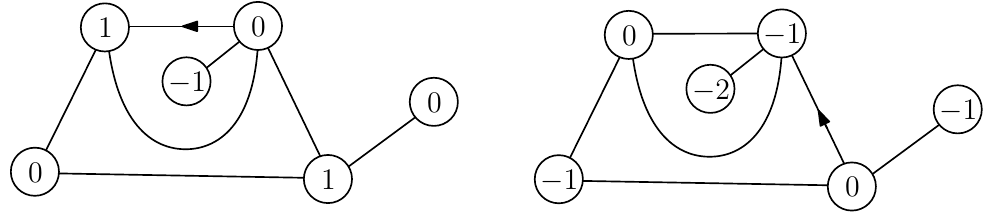} 
  \caption{A patch (left) and the corresponding minus-patch (right).}
   \label{fig:patch_minuspatch}
\end{figure}

\begin{Remark}\label{rem:ringD}
  Recall that $\Cgf(x,y)$ lies  in $\GK[x][[y,t]]$.  It follows from the second equation above that %$\Cgf(x,y)$
  $\Dgf(x,y)$ belongs to $\GK[1/(1-x)][[y,t]]$, while we had so far identified it as  a series of the larger ring $\GK[[x,y,t]]$.
\end{Remark}

%============================================================
\subsection{A smaller system}
%============================================================
Here we give a second equation between $\Pgf$ and $\Dgf$, which, combined with~\eqref{eq:DP1}, allows us to characterise the series~$\Qgf$ in terms of these two series only.

\begin{Proposition}\label{prop:new-systemQ}
  The above defined series $\Pgf(y)$ and $\Dgf(x,y)$ are related by the following equation:
 \beq \label{D-eq-new}
    %\Dgf(x,y)&=v+ {y}\, [x^{\geq0}]\left(\Dgf(x,y)\left(\frac 1 v [y^{1}]\Dgf(x,y)+\frac{1}{x}\Pgf\left(\frac{t}{x}\right)\right)\right), \nonumber\\
    (1-x) (\Dgf(x,y)-v)=  y\Dgf(x,y) (\Pgf(y)+1-v) +\om \frac ty [y^{>1}] \Dgf(x,y)+ \frac t v [y^{>0}]\left( \Dgf(x,y) [z^1]\Dgf\!\left(\frac t y,z\right)\right). 
  \eeq
  Moreover, this equation, combined  with~\eqref{eq:DP1} and the initial condition $\Pgf(0)=v$, defines $\Pgf(y)$ and $\Dgf(x,y)$ uniquely in $\GK[[y,t]]$ and $\GK[[x,y,t]]$, respectively.

  Recall that the \gf\ of quartic Eulerian orientations is
  $%\[
  \Qgf=[y^1]\Pgf(y)-v.
  $%\]
\end{Proposition}

\begin{figure}[ht]
  \setlength{\captionindent}{0pt}
  \includegraphics[scale=0.7]  {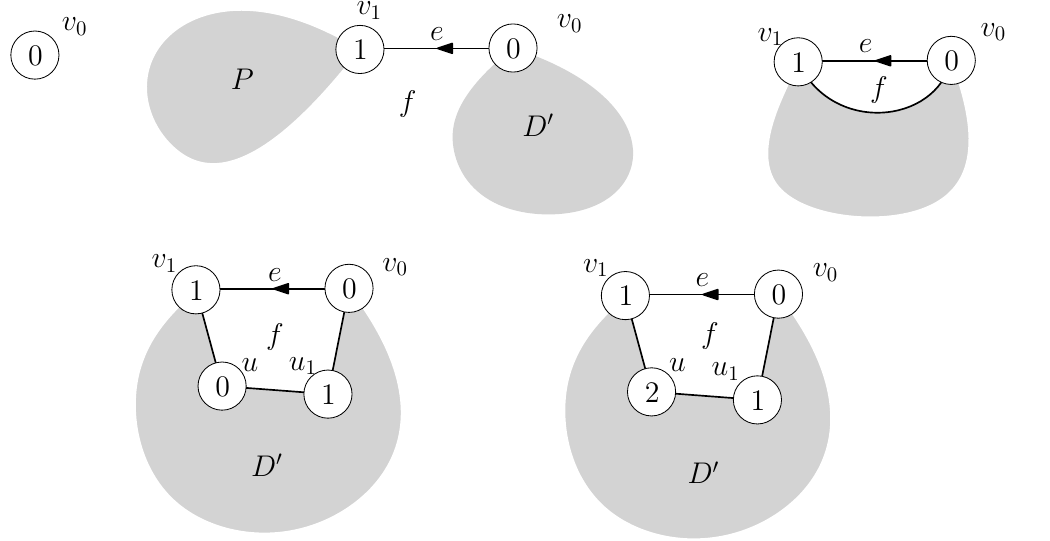} 
  \caption{The five different types of D-patches.
     In the last two cases
     vertices $v_{0},v_{1},u,u_{1}$ with the same label may coincide.}
   \label{fig:newDproof_cases}
\end{figure}

\begin{proof}
  We will start by proving the equation. Let $D$ be any D-patch. We will consider five cases, illustrated in
  Figure~\ref{fig:newDproof_cases}. In the first case,~$D$ is atomic, 
  and contributes $v$ to $\Dgf(x,y)$. For the other cases, let
  $v_{0}$, $v_{1}$, $e$ and $f$ be the root vertex, co-root vertex, root edge and co-root face of~$D$.

  In the second case  $e$ is a bridge, which means that $f$ is the outer face. Let $P$ and $D'$ be the two connected components remaining when $e$ is removed from $D$, with $D'$ containing $v_{0}$. We root $P$ at the edge toward $v_{1}$ following $e$ in clockwise order around $v_{1}$, and we root $D'$ away from~$v_{0}$ at the edge following $e$ in counterclockwise order around $v_{0}$. Then $D'$ is any D-patch and $P$ is any patch. The local minima of $D$ are exactly the local minima of $D'$ and $P$, except when $P$ is the atomic patch, in which case $v_{1}$ is a local minima in $P$ but not $D$. Hence the contribution from this case is
  \[
    y\Dgf(x,y)(\Pgf(y)+1-v).
  \]

  In the third case, $f$ is an inner digon. Removing the root edge from $D$ yields a D-patch with one fewer digon but all other parameters unchanged. The only restriction on this new D-patch is that it cannot be the atomic patch. Hence this case contributes
  \[
    x(\Dgf(x,y)-v).
  \]

  In the final two cases, $f$ is an inner quadrangle of $D$. Let $v_{0}$, $v_{1}$, $u$ and $u_{1}$ be the vertices around~$f$ in counterclockwise order. Since $D$ is a D-patch, $v_{0}$ is only adjacent to vertices labelled~$1$, so $u_{1}$ must be labelled $1$. Hence     $u$ is labelled either $0$ or $2$. 
  %In each case, let $D'$ be the map remaining when $e$ is removed from~$D$. The root edge of $D'$ is chosen to be the edge of $f$ from $v_{0}$ to $u_{1}$.

  In the fourth case, {where $u$ is labelled $0$,} let $D'$ be the map remaining when $e$ is removed from~$D$, rooted at the edge of $f$ from $v_{0}$ to $u_{1}$. Since $u$ is labelled $0$, the map $D'$ is a D-patch. Clearly $D'$  has outer degree at least $4$, but subject to this restriction, $D'$ can be any D-patch. Conversely, if we reconstruct $D$ from $D'$,
   one inner (bicoloured) quadrangle is added,
   the outer degree decreases by $2$, while the number of local minima and the number of inner digons does not change. Hence the contribution from this case is
  \[
    \om \frac{t}{y}(\Dgf(x,y)-y[y^1]\Dgf(x,y)-v) = \om \frac{t}{y} [y^{>1}] \Dgf(x,y) .
  \]

  In the fifth and final case, {where $u$ is labelled $2$,}  let $D'$ be the map remaining when $e$ is removed from~$D$, rooted at the edge of $f$ from $u_{1}$ to $u$.
  % $u$ is labelled $2$. Let us move the root corner of $D'$ one step further in counterclockwise order. The resulting map $D'$ has root vertex $u_1$.
  The outer face of $D'\equiv M$ satisfies the shifted Dobrushin condition of Proposition~\ref{prop:subpatch}, with $\ell=1$, $\oc_{\ell+1}(D')=1$ and $\oc_{\ell-1}(D')\ge 1$. We  apply to $D'$ the construction from Definition~\ref{def:sub}. Given that all inner digons of $D$ are incident to~$v_0$, they lie in the subpatch~$S$ of $D'$, as required in Proposition~\ref{prop:subpatch}. This proposition tells us that the submap $S$, once we have moved the root corner one step back in clockwise order, has outer corners labelled $0$ and $1$ alternately. Moreover, its inner faces have degree $4$ and~$2$. Now, due to the properties of $D$, all inner digons of $S$ are incident to the root vertex, and all neighbours of this root have label $1$. Hence $S$ (with shifted root corner) is a D-patch. By Proposition~\ref{prop:subpatch}, the contracted map $C$ has outer degree $2$, and, upon subtracting $1$ to all labels, is D-patch. Conversely, starting from a D-patch $S$ and another D-patch $C$ with outer degree $2$, such that $\id(C)<\oc_0(S)$ (the inequality is strict because we want $D'$ to have at least one outer corner labelled $0$),  the construction of Proposition~\ref{prop:subpatch-bij} first gives a labelled map $D'$, and, after adding a new root edge $e$ (thereby creating a new inner quadrangle that contains the only outer corner of $D'$ labelled $2$), one obtains a D-patch $D$ of the fifth type. The statistics of $D'$ (and then $D$) are derived from Proposition~\ref{prop:subpatch}, and this gives the \gf\ of the fifth case as
  \[
    \frac{t}{v}[y^{>0}]\left(\Dgf\left(x,y\right)[z^1]\Dgf\left(\frac{t}{y},z\right)\right).
  \]
  Adding the contributions from the five cases yields~\eqref{D-eq-new}.

  \medskip  Let us finally prove uniqueness of the solution. For $j, n \ge 0$, let us denote by $p_{j,n}$ (resp. $d_{j,n}$) the coefficient of $y^jt^n$ in $\Pgf(y)$ (resp. $\Dgf(x,y)$). Note that $d_{j,n}$ is a \fps\ in $x$. We will prove by   induction on $N\ge 0$ that
  \begin{itemize}
  \item $p_{j,n}$ is completely determined for $j+n<N$,
  \item $d_{j,n}$ is completely determined for $j+n \le N$.
  \end{itemize}
  When $N=0$, there is nothing to prove for $\Pgf$. From~\eqref{eq:DP1}, we know that  $\Dgf-v$ is a multiple of $y$. That is, not
  only $d_{0,0}=v$, but in addition $d_{0,n}=0$ for  $n\ge 1$. 
  Now assume that the induction hypothesis holds for some $N\ge 0$, and let us prove it for $N+1$. 

  We begin with the series $\Dgf$. Of course it suffices to
  determine the coefficients $d_{j,n}$ for $j+n=N+1$. We have already explained that
  $d_{0,N+1}=0$, so we take $j\ge 1$. The equation~\eqref{eq:DP1}
  expresses $d_{j,N+1-j}$ in terms of the series $d_{j-1,m}$ (for $m\le
  N+1-j$), $p_{k, \ell}$ (for $k+\ell \le N+1-j$) and $d_{1,m}$ (for
  $  m\le N+1-j$). If $j\ge 2$, these series are known, by the induction
  hypothesis, and thus $d_{j,N+1-j}$ is completely determined for $j\ge 2$. For $j=1$, we will determine $d_{j,N+1-j}=d_{1,N}$ using~\eqref{D-eq-new}.
 
  Let us  extract the coefficient of $y^1 t^N$ in~\eqref{D-eq-new} or, equivalently, the coefficient of~$t^N$ in the fourth equation of Proposition~\ref{prop:early-systemQ}. On the left-hand side we obtain  $(1-x)d_{1,N}$, which is the series that we wish to determine. On the right-hand side we first have  $\delta_{N,0}v$,
    % $[t^N]\Dgf(x,0)(\Pgf(0)+1-v)$, which is simply $v\delta_{N,0}$ (by the initial condition $\Pgf(0)=v$ and the first equation). The
    then the second term gives $\om d_{2,N-1}$ (which we have
    determined using~\eqref{eq:DP1}). For the third term, recall that  $d_{1,n}$ is a series in $x$, say $d_{1,n}=\sum_{k\ge 0} d_{k,1,n} x^k$. Extracting the coefficient of $y^1 t^N$ gives, up to a factor $1/v$, a sum of terms $d_{j,m}d_{k,1,n}$ for $j=k+1$ and $m+k+n=N-1$, that is, $m+j+n=N$. Hence $j+m\le N$ and
    % ,    since  $j=k+1 \ge 1$,    we also have
    $1+n \le N$, so all these numbers $d_{j,m}$ and $d_{k,1,n}$ are
    known by the induction hypothesis. Hence we now know $d_{1,N}$ as well, and hence all coefficients $d_{j,n}$ for $j+n \le N+1$.

  Next, consider the series $\Pgf$ and the coefficient $p_{j,N-j}$. The initial condition $\Pgf(0)=v$  tells us that $p_{0,n}= v\delta_{n,0}$, so let us assume that $j\ge 1$. Let us extract from~\eqref{D-eq-new} the coefficient of $y^{j+1}t^{N-j}$. The left-hand side gives $(1-x)d_{j+1,N-j}$, which is known. On the right-hand side, the first term gives
    \[
      d_{j,N-j}+ v p_{j,N-j}+ \sum_{i+k=j, m+n=N-j}d_{i,m} p_{k,n},
    \]
    where $i$ and $k$ in the sum are non-zero. The only unknown coefficient here is  $p_{j,N-j}$, which we want to determine. The second term on the right-hand side involves only $d_{j+2,N-j-1}$, which is known. Finally, for the last term in the right-hand side, we get again a sum of terms $d_{i,m}d_{k,1,n}$ for $i=k+j+1$ and $m+k+n=N-j-1$, that is, $m+i+n=N$. All these terms are known by the induction hypothesis.
   This concludes our induction.
\end{proof}

In the following lemma, we use a decomposition similar to the one of Figure~\ref{fig:newDproof_cases},
starting with a patch rather than a D-patch, to deduce an equation between $\Pgf$, $\Dgf$ and the {\gf}~$\Egf$ of E-patches (see Definition~\ref{def:E-patches}). 

\begin{Lemma}\label{NewPequation}
  The generating functions $\Pgf(y)$, $\Dgf(x,y)$ and $\Egf(x,y)$ satisfy the equation
  \[
    \Pgf(y)=v+y\Pgf(y)(\Pgf(y)+1-v)
    +\frac{\om t}{y}[y^{>1}] \Pgf(y)
     +t[y^{>0}x^{1}]\left(\Egf(x,y)+\Egf(y,x)\right).
  \]
\end{Lemma}

\begin{Remark}
  In the next subsection   we will express $\Egf$ in terms of $\Pgf$ and $\Dgf$ (see Proposition~\ref{prop:symmetry}). This will allow us to rewrite the above equation in terms of two series only:
  \beq\label{PD-add}
  \Pgf(y)=v+y\Pgf(y)(\Pgf(y)+1-v)
  +\frac{\om t}{y}[y^{>1}] \Pgf(y)
  +\frac {2t} v [y^{>0}]\left(\Pgf(y) [z^1] \Dgf\!\left( \frac t y, z\right)\right).
  \eeq
\end{Remark}

\begin{proof}
We   apply the same case analysis   to a patch $P$
  that  we applied to $D$ in the proof of~\eqref{D-eq-new}  (see Figure~\ref{fig:newDproof_cases}).  There are three differences: % in the proof:
  \begin{itemize}
  \item Case 3 does not occur as now the face $f$ cannot be a digon,
  \item In each case the map $P'$ obtained by deleting $e$ %$D'$
    does not contain any inner digon, 
  \item In each case the vertices adjacent to the root vertex, apart from those on the outer face, may be labelled $1$ or $-1$. In particular, in the final case, the face $f$ could also be labelled $0,1,0,-1$, in counterclockwise order.
  \end{itemize} 
  The fact that the third case does not occur simply means that we do not have any counterpart to the term $x(\Dgf(x,y)-v)$, which appears on the left-hand side of~\eqref{D-eq-new}. The second and third differences explain why  the series $\Dgf(x,y)$ that occurs in~\eqref{D-eq-new} for the description of Cases $2$ and~$4$ is replaced by $\Pgf(y)$ above.

  The  final difference is in the fifth case,  where the inner quadrangle  $f$ to the left of the root edge $e$ can be labelled  $0,1,2,1$ or $0,1,0,-1$, in counterclockwise order. If it is labelled $0,1,0,-1$, then let us root $P'$ 
 at the edge of $f$ that goes from $u$ to $v_1$ (in the notation of Figure~\ref{fig:newDproof_cases}). Then~$P'$ can be any E-patch with one outer corner labelled $-1$ and at least one outer corner labelled~$1$. Hence the contribution from this subcase is $t[y^{>0}x^{1}]\Egf(x,y)$.
   If $f$ is labelled $0,1,2,1$, then we root $P'$ at the edge of $f$ that goes from $u_1$ to $u$, and then we subtract $1$ from all labels. This gives an E-patch with exactly one outer corner labelled $1$, and at least one outer corner labelled~$-1$. Moreover, the number of outer corners labelled $-1$ in this E-patch, which is recorded by the first variable in $\Egf(\cdot, \cdot)$, gives the number of outer corners labelled $0$ in the original patch $P$, or equivalently the number of outer corners labelled $1$ in $P$. Since we want to record them with the variable $y$ rather than $x$, the contribution of this final subcase  is $t[y^{>0}x^{1}]\Egf(y,x)$.
\end{proof}

%==================================================================
\subsection{A Dobrushin-like class of labelled maps}
%==================================================================
Let us now consider the family of E-patches, defined in Definition~\ref{def:E-patches}, which did not appear in our earlier paper. We will express its
% their
\gf\ $\Egf$ in terms of $\Pgf$ and $\Dgf$.

\begin{Proposition}\label{prop:symmetry}
  The \gf\ $\Egf(x,y)$ that counts E-patches is given by
  \beq\label{E-def}
  \Egf(x,y) = \frac 1 v [x^{\ge 0}] \left (\Dgf\left(\frac t x, y \right) \Pgf(x)\right).
  \eeq
  Moreover the series $\Egf(x,y)$  is symmetric in $x$ and $y$.
\end{Proposition}
\begin{Remark}
  We recall that $\Pgf(x)$ belongs to $\GK[[x,t]]$ while $\Dgf(x,y)$ belongs to $\GK[[x,y,t]]$. The extraction of the non-negative part in~\eqref{E-def} thus yields a series in $\GK[[x,y,t]]$, which is indeed the nature of $\Egf(x,y)$, as argued in Section~\ref{sec:GF}.
\end{Remark}

\begin{proof}[Proof of Proposition~\ref{prop:symmetry}]
 We can apply to any E-patch $M$ the subpatch extraction/contraction of Definition~\ref{def:sub}, with $\ell=0$. Proposition~\ref{prop:subpatch-bij} then implies that the set of E-patches is in bijection with the set of pairs $(S,C)$ consisting of a patch $S$ (with labels shifted by $-1$) and a D-patch $C$ satisfying $\id(C)\leq \oc_{-1}(S)$. Proposition~\ref{prop:subpatch} then gives the desired expression of the \gf\ $\Egf(x,y)$.

\begin{figure}[ht]
   \scalebox{0.67}{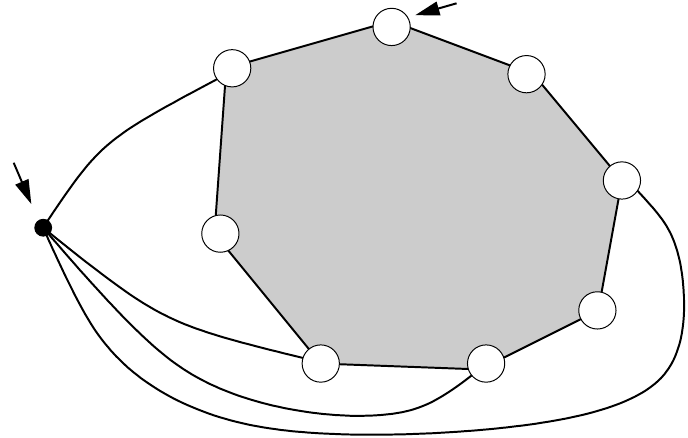}
  \caption{Adapting the Ambj\o rn-Budd transformation to the outer face of an E-patch.}
    \label{fig:AB-extended}
\end{figure}

  Let us now prove that $\Egf(x,y)=\Egf(y,x)$. Here we need an extension to E-patches of the Ambj\o rn-Budd bijection $\Phi$, schematised by Figure~\ref{fig:AB}, which is so far defined for labelled quadrangulations (rooted from $0$ to $1$) only. Given an E-patch $E$, we construct $M:=\Phi(E)$ by creating one new vertex and a collection of edges, as follows:
    \begin{itemize}
    \item we create an edge in each finite face of $E$     by applying the construction of Figure~\ref{fig:AB},
    \item we create a new unlabelled vertex $u$ in the infinite face of $E$ (Figure~\ref{fig:AB-extended}), which we join by an edge to all outer corners of $E$ labelled $1$, and to all outer corners labelled $0$ that are followed, in counterclockwise order around $E$, by a corner labelled $-1$,
       \item we  erase all edges of $E$, as well as its local minima (they have not been connected to any edge), and we root $M$ at $u$ so that the root face of $M$ contains the root corner of $E$.
    \end{itemize}
    This is a special case of the bijection of~\cite[Sec.~2.2]{bouttier-fusy-guitter}. It sends E-patches to rooted planar maps in which all vertices, apart from the root vertex $u$, are labelled, in such a way that
    \begin{itemize}
    \item   labels differ by $0$, $\pm1$ along edges,
      \item the root vertex $u$ is incident, in clockwise order starting from the root corner, to a (possibly empty) sequence of vertices labelled $1$, then a (possibly empty) sequence of vertices labelled $0$. 
    \end{itemize}
  The series $E(x,y)$ counts these maps by edges not attached to $u$ (variable $t$), monochromatic edges ($\om$), faces ($v$), edges with endpoint $0$ attached to $u$ ($x$) and edges of endpoint $1$ attached to $u$ ($y$). Then the symmetry in $x$ and $y$ can be seen by taking one of these maps, replacing each label $\ell$ with {$1-\ell$} and finally reflecting the map. 
\end{proof}

\begin{Remark}
  At $v=1$ there is a more direct way to see the symmetry of $\Egf$: simply take an E-patch $E$ and replace each label $\ell$ with $-\ell$ (and change the root in some canonical way). This does not work for general $v$, however, as this transformation does not preserve the number of local minima.
\end{Remark}

\begin{Remark}
  In light of the previous remark, we observe a surprising symmetry between local maxima and local minima, which we describe below. Fix a sequence $L=\ell_{0},\ell_{1},\ldots,\ell_{2k}$ with $\ell_{j+1}-\ell_{j}=\pm1$ and $\ell_{2k}=\ell_{0}$. Denote by $\Egf_{L}(t,\omega,v_{1},v_{2})$ the generating function of
  % counting
  labelled maps with outer degree $2k$ whose outer corners are labelled $\ell_{0},\ell_{1},\ldots,\ell_{2k}$ in counterclockwise order and whose inner faces are quadrangles, counted according to the following statistics:
  \begin{itemize}
  \item the number of inner quadrangles (variable $t$),
  \item the number of inner bicoloured quadrangles (variable $\omega$),
  \item the number of vertices that are local minima (variable $v_{1}$),
  \item the number of vertices that are local maxima (variable $v_{2}$).
  \end{itemize}
  The sign flip used in the remark above combined with $\Egf(x,y)=\Egf(y,x)$ implies that $\Egf_{L}(t,\omega,v,1)=\Egf_{L}(t,\omega,1,v)$ when $L$ is a Dobrushin sequence as in Definition~\ref{def:E-patches}.
  Empirically we observe a much stronger symmetry, namely that for any sequence $L$ we have $\Egf_{L}(t,\omega,v_{1},v_{2})=\Egf_{L}(t,\omega,v_{2},v_{1})$, however we do not have a proof. We have checked by brute force that $[t^n]\Egf_{L}(t,\omega,v_{1},v_{2})=[t^n]\Egf_{L}(t,\omega,v_{2},v_{1})$ whenever $n+k\leq 8$.
\end{Remark}

%=============================================================
\subsection{A change of variables}
%=============================================================
As in~\cite{mbm-aep1}, the form of our functional equations suggests certain changes of variables. We thus introduce
\beq\label{change}
\Pnn(y):=t\Pgf(yt), %\quad  \Cnn(x,y):= \frac 1 v \Cgf(x,ty),
\quad \Dnn(x,y):=\frac{1}{v}\Dgf(x,yt), \quad
\Enn(x,y):= t \Egf(xt,yt).
\eeq 
These three series belong respectively to $\GK[y][[t]]$, %$\GK[x,y][[t]]$,
$\GK[y][[x,t]]$ (and more precisely to $\GK[y, 1/(1-x)][[t]]$, see Remark~\ref{rem:ringD}), and finally $\GK[x,y][[t]]$. Moreover, $\Enn(x,y)$ is symmetric in $x$ and $y$.  The equations of Propositions~\ref{prop:new-systemQ} (including~\eqref{eq:DP1}) and~\ref{prop:symmetry} read
\allowdisplaybreaks
\begin{align}
  \Dnn(x,y) &= 1+ y [x^{\ge 0}] \left( \Dnn(x,y) \left( \frac 1 x \Pnn\left(\frac 1 x\right)+ [y^1] \Dnn(x,y)\right)\right), \label{Dnn-short}\\
  (1-x) \left( \Dnn(x,y)-1\right) & = y \Dnn(x,y) \left( \Pnn(y) +t-tv\right) +  [y^{>0}] \left( \Dnn(x,y) \left(  \frac \om y+ [z^1] \Dnn\!\left( \frac 1 y, z\right)\right)\right),\label{Dnn-long}
  \\                                    
  \label{Enn-def}
  \Enn(x,y)&=  [x^{\ge 0}] \left (\Dnn\left(\frac 1 x, y \right) \Pnn(x)\right).
\end{align}
Equation~\eqref{PD-add} rewrites as
\begin{equation}\label{Pnn}
  \Pnn(y)=tv+y\Pnn(y) \left( \Pnn(y)-t(v-1)\right)+\frac \om y [y^{>1}] \Pnn(y) + 2 [y^{>0}] \left( \Pnn(y) [z^1]\Dnn\!\left( \frac {1} y, z\right)\right).
\end{equation}

We note that the \gf\ of patches is $\Pgf(y)=\Egf(0,y)$. Equivalently, we have  $\Pnn(y)=\Enn(0,y)$, and Eq.~\eqref{Enn-def} gives yet another expression of $\Pnn(y)$:
\[
  \Pnn(y)= %\frac 1 v
  [x^0] \left( \Dnn\!\left(\frac 1 x,y\right)\Pnn(x)\right).
\]
We will use  in Section~\ref{sec:six} the following byproduct:
\beq\label{P1-prod}
[y^1]\Pnn(y)= %\frac 1 v
[x^0y^1] \left( \Dnn\!\left(\frac 1 x,y\right)\Pnn(x)\right).
\eeq

Finally, recall that the \gf\ of quartic Eulerian orientations that we want to determine, given by~\eqref{Q-P-sol}, is 
\beq\label{Q-nn}
\Qgf= \frac 1 {t^2} [y^1] \Pnn(y) -v.
\eeq

%%%%%%%%%%%%%%%%%%%%%%%%%%%%%%%%%%%%%%%%%%%%%%%%%%%%%%%%%%%%%%%%%%%% 
\section{A new characterisation of the series $\bm{\Qgf(t,\omega,v)}$}
\label{sec:charac}
%%%%%%%%%%%%%%%%%%%%%%%%%%%%%%%%%%%%%%%%%%%%%%%%%%%%%%%%%%%%%%%%%%%% 

In this section, we derive from the functional equations of the previous section  a new characterisation of the series $\Qgf\equiv\Qgf(t, \om, v)$, which will be the starting point of the proofs of Theorems~\ref{thm:general} to~\ref{thm:allomega}.
This characterisation is simpler than the functional equations in that it involves series with a single ``catalytic'' variable, $x$. The main series \emm mixes, positive and negative powers of the variable~$x$ (hence the notation $\Mnn$), and is defined by:
\beq\label{M-def}
\Mnn(x)=\frac{1}{x}\Pnn\left(\frac{1}{x}\right)+[y^{1}]\Dnn(x,y).
\eeq
This  is an element of $\GK[ \frac 1 x , \frac 1{1-x}][[t]]$, which we can see as an element of $\GK((x))[[t]]$, that is, a series in~$t$ whose coefficients are Laurent series in $x$ (themselves with coefficients in $\qs[\om,v]$). Note also that $\Mnn$ has no constant term in $t$. The first two coefficients read
\beq\label{M-expansion}
\Mnn(x)= \left(\frac{v}{x}+\frac{1}{1-x}\right) t +
\left(\frac{v}{x^{2}}+\frac{1+\omega  v }{\left(1-x \right)^{2}}+\frac{\omega}{\left(1-x \right)^{3}}\right) t^{2}+ \LandauO(t^3).
\eeq
We actually give two equivalent formulations of our
characterisation of $\Mnn(x)$: the first one (Proposition~\ref{prop:M})   will be used in Section~\ref{sec:01} to
solve the cases $\om=0$ and $\om=1$, while the second
(Proposition~\ref{prop:F}) will be used in Section~\ref{sec:six} to
solve the case $v=1$. Finally, Proposition~\ref{prop:M-charac} in the
introduction is a mere reformulation of Proposition~\ref{prop:F}, and will be used in Appendix~\ref{app:complex}.

\begin{Proposition}\label{prop:M} The series
  \[
    \Znn(x,y):=  \frac{\left ( 1- \frac 1 x \Mnn(y)\right)\left(1-x-{\om}{y}+\frac t y (v-1)-\Mnn(y)\right)}
    {\left ( 1- \frac 1 y \Mnn(x)\right)\left(1-y-{\om}{x}+\frac t x (v-1)-\Mnn(x)\right)},
  \]
  which belongs \emm  a priori, to $\GK[x,\frac 1 x, \frac 1 {1-x}, y, \frac 1 y, \frac 1 {1-y}, \frac 1 {1-y-{\om}{x}}][[t]]$, actually lies in the smaller ring $\GK[x, \frac 1 {1-x}, y, \frac 1 {1-y}, \frac 1 {1-y-{\om}{x}}][[t]]$. That is, expanding $ \Znn(x,y)$ in $t$, then in $x$ and $y$ (in any order), yields a series of $\GK[[x,y,t]]$.

  Moreover, this property, combined with the initial condition $[x^{-1}]\Mnn(x)=tv$,
  % and the fact that $\Mnn$ is a multiple of $t$,
  defines $\Mnn(x)$ uniquely in the ring $t\GK((x))[[t]]$.

  Finally, the \gf\ of quartic Eulerian orientations
  can be recovered from $\Mnn$ by
  \beq\label{Q-M}
  \Qgf=\frac{1}{t^{2}}[x^{-2}]\Mnn(x) -v.
  \eeq

  This characterisation of $\Mnn(x)$ and $\Qgf$ also holds for numeric values of $v$ and $\om$.
\end{Proposition}
The first terms of $\Znn$ are found to be
\beq\label{Z-expansion}
\Znn(x,y)= \frac{1- x-\omega y  }{1-y-\omega  x }
+\frac{\omega   \left(\omega -1\right)\left(x -y \right) }{\left(1-y-\omega  x  \right)^{2} \left(1-x \right) \left(1-y \right)} t
-\frac{\omega  \left(\omega -1\right) \left(x -y \right)\Pol(x,y)}{\left(1-x \right)^{3} \left(1-y \right)^{3} \left(1-y-\omega  x \right)^{3}} t^{2}+\LandauO(t^{3})
\eeq
for some polynomial $\Pol(x,y)$ in $\GK[x,y]$. Observe the factors $\om \left(\omega -1\right)$ in the second and third terms. We will prove  in Section~\ref{sec:01} that these factors persist in higher order terms, and use this to solve the cases $\omega=0$ and $\omega=1$ explicitly.

\begin{proof}
  Equations~\eqref{Dnn-short} and~\eqref{Enn-def} (the latter written at $1/x$ rather than $x$), combined with the definition~\eqref{M-def} of $\Mnn(x)$, give
  \[
    \Dnn(x,y)+\frac{y}{x}\Enn\!\left(\frac{1}{x},y\right)=1+y\,\Dnn(x,y)\Mnn(x).
  \]
Moreover, Equation~\eqref{Dnn-long} reads:
  \[
    (1-x) \Dnn(x,y)=1-x
    +[y^{>0}]\left(\Dnn(x,y)\left(\frac{\om}{y}+yt(1-v)+\Mnn\left(\frac{1}{y}\right)\right)\right).
  \]
  The first identity holds in $\GK[y]((x))[[t]]$, while the second involves series in the ring $\GK((1/y))[[x,t]]$. Both rings are subrings of $\GK((1/y))((x))[[t]]$.
  Equivalently,
  \[
    \Dnn(x,y)\left ( 1- y \Mnn(x)\right) = 1-\frac{y}{x}\Enn\!\left(\frac{1}{x},y\right)
  \]
  and
  \[
    [y^{>0}]\left[\Dnn(x,y)\left(1-x-\frac{\om}{y}-yt(1-v)-\Mnn\left(\frac{1}{y}\right)\right)\right]=0.
  \]
  Let us now replace $y$ by $1/y$ in these two equations:
  \[
    \Dnn\left(x,\frac 1 y\right)\left ( 1- \frac 1 y \Mnn(x)\right) = 1-\frac{1}{xy}\Enn\!\left(\frac{1}{x},\frac 1 y\right)
  \]
  and
  \[
    [y^{<0}]\left[\Dnn\left(x,\frac 1 y\right)\left(1-x-{\om}{y}-\frac t y (1-v)-\Mnn\left(y\right)\right)\right]=0,
  \]
  with all series considered as elements of $\GK((y))((x))[[t]]$. Note that there trivially holds
  \[
    [x^{<0}]\left[\Dnn\left(x,\frac 1 y\right)\left(1-x-{\om}{y}-\frac t y (1-v)-\Mnn\left(y\right)\right)\right]=0.
  \]
  In other words, the series
  \[
    \Dnn\left(x,\frac 1 y\right)\left(1-x-{\om}{y}-\frac t y (1-v)-\Mnn(y)\right)=
    \frac{\left(1-\frac{1}{xy}\Enn\left(\frac{1}{x},\frac 1 y\right)\right)\left(1-x-{\om}{y}-\frac t y (1-v)-\Mnn\left(y\right)\right)}{\left ( 1- \frac 1 y \Mnn(x)\right)}
  \]
  actually belongs to $\GK[[y]][[x]][[t]]=\GK[[x,y,t]]$. Note that all series involved here lie in $\GK[x,\frac 1 x, \frac 1 {1-x}, y, \frac 1 y, \frac 1 {1-y}][[t]]$, so that expanding them first in $x$ and then in $y$, or vice-versa, gives the same series. In particular, we can exchange $x$ and $y$ in the above  identity. Remembering that $\Enn(x,y)=t \Egf(xt,yt)$ is symmetric in $x$ and $y$ (Proposition~\ref{prop:symmetry}), we finally conclude that the ratio $\Znn(x,y)$ defined in the proposition,
   expanded in $\GK((y))((x))[[t]]$ (or in $\GK((x))((y))[[t]]$, as argued), contains no negative exponents in $x$ or $y$.

  \medskip

\noindent{\bf Initial condition and uniqueness.}  By definition~\eqref{M-def} of $\Mnn(x)$, the coefficient of $x^{-1}$ in this series is
  \[
    [x^{-1}]\Mnn(x)=\Pnn(0)=t\Pgf(0)=tv.
  \]
  We will now show that this condition, combined with the above property of $\Znn(x,y)$, 
  uniquely defines $\Mnn(x)$ as an element of $t\GK((x))[[t]]$. So let $M(x)$ be a series satisfying all these conditions, and consider the modified ratio
  \beq\label{Z-bar}
  \overline \Znn(x,y):= \frac{1-y-\om x}{1-x-\om y}\, \Znn(x,y),
  \eeq
 (with $\Mnn$ replaced by $M$), which expands as a series in $t$ with constant term $1$, and coefficients in $\GK[[x,y]]$.
  Then we have in particular the following equation:
  \[
    [y^{-1}]\frac{(1-x-\om y)}{(1-x)(1-\om y)}\log(\overline \Znn(x,y))=0.
  \]
  Using the expressions of $\overline \Znn$ and $\Znn$, and simplifying, yields the equation:
  \[
    0=[y^{-1}]\frac{(1-x-\om y)}{(1-x)(1-\om y)}\left(\log\left(1-\frac{1}{x}M(y)\right)
      +\log\left(1-\frac{t(1-v)+yM(y)}{y(1-x-\omega y)}\right)
      -\log\left(1-\frac{1}{y}M(x)\right)\right),
  \]
   because neither $\log(1-y-\om x)$, nor the second term involving $M(x)$ contribute to the coefficient of $y^{-1}$. 
  Now we write each term $\log(1-u)$ as the sum $-\sum_{n\ge 1} u^n/n$. The three terms obtained for $n=1$ contribute
  \begin{multline*}
    -  \frac{(1-x-\om y)}{(1-x)(1-\om y)}[y^{-1}]\left( \frac 1 x M(y) + \frac{t(1-v) +yM(y)}{y(1-x-\omega y)} - \frac 1 y M(x)\right)\\
    =[y^{-1}]\left(\frac{(1-x-\om y)M(x)}{y(1-x)(1-\om y) }
      -\frac{M(y)}{x\left(1-x \right) }
      -\frac{t\left(1-v\right) }{ y\left(1-x \right) \left(1- \omega y \right)}\right) \\
    \hskip 25mm = M(x) +\frac{t(xv-x-v)}{x(1-x)} \hskip 18mm \text{because} \ \  [y^{-1}]M(y)=tv \ \ \text{by assumption}.\hskip 20mm
  \end{multline*}
  Including now all values of $n$ gives:
  \begin{multline}\label{Mnn-rec}
    M(x)=\frac{t(x+v-xv)}{x(1-x)}+\sum_{n\ge2}\frac{x\om^{n-1}M(x)^{n}}{n(1-x)}\\
    +[y^{-1}]\sum_{n\ge2}\left(\frac{(1-x-\om y)M(y)^{n}}{nx^{n}(1-x)(1-\om y)}+\frac{\left(\frac{t(1-v)}{y}+M(y)\right)^{n}}{n(1-x)(1-\om y)(1-x-\om y)^{n-1}}\right).
  \end{multline}
  This equation, plus the fact that $M$ has, by assumption, no constant term in $t$,  recursively determines the coefficient of $t^n$ in $M(x)$. This proves the second statement of the proposition.

  \medskip
  Finally, the expression~\eqref{Q-M} of $\Qgf$ follows from its earlier expression~\eqref{Q-nn} in terms of $\Pnn$ and  the definition~\eqref{M-def} of $\Mnn$.
\end{proof}

\begin{Remark}
Equation~\eqref{Mnn-rec} gives a reasonably efficient algorithm for computing the coefficients of $\Mnn(x)$ and hence of $\Qgf$, which only requires handling series in \emm one, catalytic variable.
\end{Remark}

\medskip

We will reformulate the above characterisation of $\Mnn$ in terms of the following series:
\beq\label{FM}
\Fnn(x):=(x\Mnn(x)-t(v-1))(1-\om x-\Mnn(x)).
\eeq
A priori, this is a series in $t$ with coefficients in $\GK[x, 1/x, 1/(1-x)]$ (and no constant term in $t$), but in fact a stronger property holds.

\begin{Proposition}\label{prop:F}
  The series $\Fnn(x)$ belongs to $\GK[x, 1/(1-x)][[t]]$. 
  Moreover,
  \beq\label{kernel-F}
  \Fnn(\Mnn(x))= \left(x\Mnn(x)-{t(v-1)}\right)(1-x-\omega \Mnn(x)).
  \eeq
  The series $\Fnn(x)$ is the only element of $t\GK[[x,t]]$ such that the (unique) series $\Mnn(x)$ that  solves the quadratic equation~\eqref{FM} and 
  is a multiple of~$t$,  that is,
  \beq\label{MF-z}
  \Mnn(x)
  %= \frac 1 2 \left( 1-\om x + \frac{t(v-1)}x    -  \left( 1- \om x - \frac{t(v-1)}x\right)    \sqrt {1- \frac{4\Fnn(x)}{x\left(1-\om x - \frac {t(v-1)}x\right)^2}}\right),
=\frac{x(1-\om x) + t(v-1)}{2x} - \frac{x(1-\om x)-t(v-1)}{2x}  \sqrt {1- \frac{4x\Fnn(x)}{\left(x(1-\om x) -t(v-1)\right)^2}},
  \eeq
  also satisfies~\eqref{kernel-F} and the initial condition $[x^{-1}]\Mnn(x)=tv$.

  Again, this characterisation of $\Mnn(x)$ also holds for numeric values of $v$ and $\om$. 
\end{Proposition}
\begin{proof}
  To prove the first statement, we get back to the definition of $\Mnn$ in terms of $\Pnn$ and $\Dnn$ (see~\eqref{M-def}) and use the functional equation~\eqref{Pnn}, taken at $y=1/x$. After a few reductions, we obtain
    \beq\label{F-expr}
      \Fnn(x)=t(1-\om x) -\om [x^{-1}] \Pnn\!\left( \frac 1 x \right) - 2 [x^{\ge 0}] \Pnn\!\left( \frac 1 x \right) \Dnn_1(x) + x\Dnn_1(x)\left(1-\om x - \Dnn_1(x)\right)+ t(v-1) \Dnn_1(x),
    \eeq
    where $\Dnn_1(x)$ stands for the coefficient of $y^1$ in $\Dnn(x,y)$. This proves that $\Fnn(x)$ contains no negative power of $x$. The fact that it is a multiple of~$t$ is already clear on its definition~\eqref{FM} in terms of $\Mnn$, since $\Mnn$ itself is a multiple of~$t$.

  To prove  the second statement, we note that the series $\Znn(x,y)$ of Proposition~\ref{prop:M} can be written
  \[ 
    \Big( (xy-t(v-1)) (1-x-\om y) -\Fnn(y)\Big)/  \Big( (xy-t(v-1)) (1-y-\om x) -\Fnn(x)\Big).
  \] 
  More precisely, the numerator of $\Znn(x,y)$ is the numerator of the above expression divided by~$xy$.
  This gives
  \beq\label{Z-def}
  (xy-t(v-1)) (1-x-\om y) -\Fnn(y)=  \Big( (xy-t(v-1)) (1-y-\om x) -\Fnn(x)\Big) \Znn(x,y),
  \eeq
  where we recall that $\Znn(x,y) \in \GK[[x,y,t]]$. Recall moreover that $\Fnn(y) \in \GK[[y,t]]$. Let us now replace $y$ by $\Mnn(x)$ in the above identity: since $\Mnn(x)$ is a multiple of~$t$, this substitution yields well defined series in $\GK((x))[[t]]$. By definition of $\Fnn(x)$, the right-hand side then vanishes. Hence the left-hand side vanishes too, and this gives~\eqref{kernel-F}. Note that this is an application of the so-called \emm kernel method,, see e.g.~\cite{hexacephale,bousquet-petkovsek-1,banderier-flajolet}.

  \medskip 
  Let us finally prove uniqueness.  Assume that a series $F(x)$ of $t\GK[[x,t]]$ satisfies the conditions stated in the lemma. We want to prove that the corresponding series $M(x)$ satisfies the conditions of Proposition~\ref{prop:M}, so that $M$, and thus $F$, are uniquely defined. Since the initial condition on $M$ holds by assumption, it suffices to prove that the series $Z(x,y)$ defined by~\eqref{Z-def} (with $\Fnn$ replaced by $F$ of course) belongs to $\GK[[x,y,t]]$. A priori, due to the form of the denominator of $Z(x,y)$, the coefficient of $t^n$ in this series is an element of $\GK[[x,y]]$, divided by $x^n y^n$.

   It follows from~\eqref{MF-z} that $M(x)$ is a series of $t/x\GK[[t/x,x]]$.  By assumption, the left-hand side of~\eqref{Z-def} vanishes at $y=M(x)$. Hence
  \[
    \frac{  (xy-t(v-1)) (1-x-\om y) -F(y)}{y-M(x)}
  \]
  expands as a \fps \  in $y$ with coefficients in $\GK[[x,t/x]]$. Since $F(y)$ is a multiple of~$t$, this still holds if we divide the above expression by $x$. Hence  the same holds for
  \[
    Z(x,y)=  \frac{  (y-\frac{t(v-1)}x) (1-x-\om y) -\frac 1 x F(y)}{(y-M(x))\left(1- y-\om x + \frac{t(v-1)} x -M(x)\right)},
  \]
  which thus belongs to $\GK[[x,t/x,y]]$. That is, the coefficient of $t^n$ in $Z(x,y)$ is a series of $\GK[[x,y]]$ divided by $x^n$.

  Let us also note that
  \[
    Z(x,y)=  \frac 1{Z(y,x)}.
  \]
  We have just proved that $Z(y,x)$ belongs to $\GK[[x,y,t/y]]$. Since the coefficient of $t^0$ in $Z(y,x)$  is $({1-y-\om x})/({1-x-\om y})$, the above formula shows that $Z(x,y)$ belongs to $\GK[[x,y,t/y]]$ as well. Hence the coefficient of $t^n$ in $Z(x,y)$ has no pole at $x=0$ either, and finally $Z(x,y)$ lies in $\GK[[x,y,t]]$. By Proposition~\ref{prop:M}, the series $M(x)$ coincides with $\Mnn(x)$ and hence  $F(x)$ coincides with $\Fnn(x)$.
\end{proof}

Let us now prove that the characterisation of $\Mnn(x)$ given in the introduction, namely Proposition~\ref{prop:M-charac}, is a reformulation of the above proposition.

\begin{proof}[Proof of Proposition~\ref{prop:M-charac}] The first two
  properties of Proposition~\ref{prop:M-charac} come from
  Proposition~\ref{prop:F}, while  the expression~\eqref{Q-M} of
  $\Qgf$ in terms of $\Mnn$ is borrowed from Proposition~\ref{prop:M}. Hence it suffices to show that $\Mnn^{(2)}(x):=\Mnn(\Mnn(x))$ is well defined, that $\Mnn^{(2)}(x)=x$, and that the conditions of Proposition~\ref{prop:M-charac} uniquely define $\Mnn(x)$. % and $\Fnn(x)$.
  
  Starting with the first two conditions of
  Proposition~\ref{prop:M-charac}, observe that, since $\Mnn(x)$ is a
  multiple of $t$, the same holds for $\Fnn(x)$. As
  already observed in the previous proof, the form~\eqref{MF-z} implies that
  $\Mnn(x)\in t/x \GK[[x,t/x]]$. Since moreover
  $[t^{1}x^{-1}]\Mnn(x)=v$, we have $t/\Mnn(x)\in x\GK[[x,t/x]]$. This implies that the substitution $\Mnn(\Mnn(x))=\Mnn^{(2)}(x)$ is well defined and lies in $x\GK[[x,t/x]]$. 

  We now use Eq.~\eqref{kernel-F} from  Proposition~\ref{prop:F}:
  \begin{align*}
    \Fnn(\Mnn(x))&= \left(x\Mnn(x)-{t(v-1)}\right)(1-x-\omega \Mnn(x))\\
                 &   = \left(\Mnn(x)\Mnn^{(2)}(x)-t(v-1)\right)\left(1-\om \Mnn(x)-\Mnn^{(2)}(x)\right),
  \end{align*}
  by definition of $\Fnn$. Comparing both expressions shows that
    \[
      \text{either} \qquad  \Mnn^{(2)}(x)=x \qquad \text{or} \qquad      x+ \Mnn^{(2)}(x)= 1-\om \Mnn(x)+ \frac{t(v-1)}{\Mnn(x)}.
    \]
    Let us think of the latter identity as relating series in $x$ and $t':=t/x$, and extract the coefficient of $x^0 t'^0$:  this gives $0=1$, a contradiction. We thus conclude that $\Mnn^{(2)}(x)=x$.

  It remains to show that the conditions of Proposition~\ref{prop:M-charac} uniquely define $\Mnn(x)$. % and $\Fnn(x)$.
  But they imply
  the conditions of Proposition~\ref{prop:F}, which we have proved to characterise  $\Fnn(x)$ and $\Mnn(x)$. 
\end{proof}

\begin{Remark}
    So far we have expressed the series $\Qgf$ in terms of the coefficient of $x^{-2}$ in $\Mnn(x)$:
    \[
      t^2(\Qgf+v)=[x^{1}]\Pnn(x) =[x^{-2}]\Mnn(x).
    \]
    There exists an alternative expression, which involves the constant terms of $\Fnn(x)$ and $\Mnn(x)$, and will be especially convenient in the case $v=1$ (Section~\ref{sec:six}). Namely,
    \beq
    \label{Q-alt}
    (\om+2) t^2 (\Qgf+v)= t- \Fnn(0) +t(v-1) [x^0]\Mnn(x).
    \eeq
    To prove this, we extract the coefficient of $x^0$ in~\eqref{F-expr}. This gives:
    \begin{align*}
      \Fnn(0)&= t-\om [x^1] \Pnn(x)-2 [x^0]\Pnn\!\left( \frac 1 x \right)\Dnn_1(x)+t(v-1)[x^0]\Dnn_1(x)\\
               %[x^0](1-\omega x) \Pnn(1/x) -2 [x^0] \Pnn(1/x) \Dnn_1(x) -t(v-1)+ t(v-1) [x^0]\Dnn_1(x)\\
             &=t -(\omega+2) [x^1]\Pnn(x)+ t(v-1) [x^0]\Mnn(x),
    \end{align*}
    where we have used~\eqref{P1-prod} and then~\eqref{M-def}.
    This gives the desired expression~\eqref{Q-alt} of $\Qgf$.
  \end{Remark}

\begin{Remark}\label{rem:t-F0}
We will use~\eqref{Q-alt} to show that $t-\Fnn(0)$ has non-negative coefficients (see~\eqref{F0-ser}). By~\eqref{Q-P-sol},~\eqref{M-def} and~\eqref{change}, we have
\[
  t-\Fnn(0)=(\om+2) t^2 [y^1]\Pgf(y)+t^{2}(1/v-1) [y^1]\Dgf(0,y),
\]
so it suffices to show that the coefficients of $\Pgf(y)-\Dgf(0,y)$ are non-negative. Indeed, this is clear from Definition~\ref{def:patches} as $\Dgf(0,y)$ counts $D$-patches with no digons, which are patches whose root vertex is a local minimum, whereas $\Pgf(y)$ counts patches without this restriction.
\end{Remark}

%%%%%%%%%%%%%%%%%%%%%%%%%%%%%%%%%%%%%%%%%%%%%%%%%%%%%%%%%%%%%%%%% 
\section{The cases $\bm {\omega=0}$ and $\bm{\omega=1}$}
\label{sec:01}
%%%%%%%%%%%%%%%%%%%%%%%%%%%%%%%%%%%%%%%%%%%%%%%%%%%%%%%%%%%%%%%%% 
We will now solve the cases where the ratio $\Znn(x,y)$ of Proposition~\ref{prop:M} reduces to its constant term (in $t$), or equivalently, where the modified ratio $\overline \Znn(x,y)$ defined in~\eqref{Z-bar} equals $1$.  The first terms in the expansion~\eqref{Z-expansion} of $\Znn(x,y)$ show that this can only happen for $\om=0$ or $\om=1$. In this section, we prove that $\overline \Znn(x,y)$ is indeed $1$ for these two values of $\om$, and use this to solve these two cases completely. In the special case where $v=1$ we recover the main two results of~\cite{mbm-aep1}. In contrast to this earlier solution, our approach here does not require any guessing, whereas it was central in~\cite{mbm-aep1}.

Before we dig into this, observe that $\overline \Znn(x,y)=1$ means that
the following expression is symmetric in $x$ and $y$:
\beq\label{symmetriceq0}
(1-y-\om x)\left(1-\frac{1}{x}\Mnn(y)\right)\left(1-x-\omega y+\frac{t(v-1)}{y}-\Mnn(y)\right).
\eeq
Multiplying by $xy$ and cancelling some symmetric terms, we see that this is equivalent to the $x/y$-symmetry of 
\begin{equation}\label{symmetriceq}
  (1-y-\om x)\left(y\Mnn(y)-t(v-1)\right)\left(1-\om y-\Mnn(y)\right).
\end{equation}
Now only the first factor involves $x$.

In this section we denote by $\GK$ the field $\qs(v)$.

% ============================================================
\subsection{Solution for $\bm{\om=0}$: proof of Theorem~\ref{thm:general}}
\label{sec:om0}
% ============================================================

For $\om=0$, the  expression~\eqref{symmetriceq} does not depend on $x$. Hence it  is symmetric if and only if  it does not depend on $y$ either. Hence, in this case, we should have
\begin{equation}\label{omis0eqforH}
  (1-y)\left(y\Mnn(y)-t(v-1)\right)\left(1-\Mnn(y)\right)=\Rgf_0,
\end{equation}
for  some series  $\Rgf_0$ in $\GK[[t]]$.
Using the first terms in the expansion of $\Mnn(x)$ (see~\eqref{M-expansion} and our {\sc Maple} session~\cite{bmep-ref-arxiv}), we find that the series $\Rgf_0$ should start as follows:
\beq\label{Rgf-ser}
\Rgf_0= t -\left(1+v\right) t^{2}-\left(1+3 v\right) t^{3}+\LandauO(t^4).
\eeq
We observe that these first terms do not depend on $y$ indeed, which is encouraging.

Now let $R=\sum_{n\ge 1} r_n t^n$ be \emm any, series in $ t\GK[[t]]$. Consider the following quadratic equation in $M(y)$, derived from~\eqref{omis0eqforH}:
\beq\label{eq-M}
(1-y)\left(yM(y)-t(v-1)\right)\left(1-M(y)\right)=R.
\eeq
It admits two solutions. Both are \fps\ in $t$, but one of them has a non-zero constant term (more precisely, $1$), and, 
since our series $\Mnn(y)$ is a multiple of $t$, we are not interested in it. The other solution has constant term $0$ and can be written  as a series in $R$, $y$ and $t(v-1)$ as follows:
\begin{align}
  \label{My-exp}
  M(y)
  %&=\frac{y+t(v-1)}{2y} \left( 1- \sqrt{1-4y \frac{t(v-1)+R/(1-y)}{(y+t(v-1))^2}}\right)\\
  & = \frac{y+t(v-1)}{2y}- \frac{y-t(v-1)}{2y}\sqrt{1- \frac{4yR}{(1-y)(y-t(v-1))^2}}\\
      &=
        \frac{tv-t}{y}+\sum_{n,k,j\geq0}\frac{1}{n+1}{2n\choose n}{2n+k\choose k}{n+j\choose n}t^k(v-1)^k R^{n+1}y^{j-n-k-1}. \nonumber 
\end{align}
Since we have taken $R$ with no constant term,  this is a series of $t\GK((y))[[t]]$, and more precisely, a series of $t/y \GK[[y,t/y]]$.

Now, following the conditions of Proposition~\ref{prop:M}, we would like the coefficient of $y^{-1}$ in $M(y)$ to be $tv$. Equivalently,
\beq\label{t-R0}
t=\sum_{n,k\geq0}\frac{1}{n+1}{2n\choose n}{2n+k\choose k}{2n+k\choose n} t^k(v-1)^kR^{n+1}.
\eeq
But this condition defines a \emm unique, series $R$ with constant term $0$, which actually starts like the series $\Rgf_0$ (see~\eqref{Rgf-ser}).

Now let us fix $\Rgf_0\equiv R$ to be the unique solution of the above equation, and define $M(y)$ by~\eqref{My-exp}. We can now work our way backwards. Equation~\eqref{eq-M} holds, which means that Expressions~\eqref{symmetriceq} and~\eqref{symmetriceq0} (with $\Mnn(y)$ replaced by $M(y)$ and $\om$ by $0$) are symmetric in $x$ and $y$, so that the expression $\Znn(x,y)$ of Proposition~\ref{prop:M} is reduced to its constant term. In particular, it expands as a series of $\GK[[x,y]]$, and the first condition of  Proposition~\ref{prop:M} is satisfied by $M(y)$. By construction, $M(y)$  belongs to $\GK((y))[[t]]$, has no constant term in $t$, and  $[y^{-1}]M(y)=tv$. By Proposition~\ref{prop:M}, our series $M$ coincides with the series $\Mnn$ defined by~\eqref{M-def}, and we have solved the case $\om=0$.

According to Proposition~\ref{prop:M},  the series $\Qgf_0$ that counts quartic Eulerian orientations with no alternating vertex by vertices ($t)$ and clockwise faces ($v$) is given by
\begin{align}
  \Qgf_0&=-v+\frac{1}{t^2}[y^{-2}]M(y) \nonumber \\
        &  = -v+\frac{1}{t^2}\sum_{{n,k\geq0},{~n+k>0}}\frac{1}{n+1}{2n\choose n}{2n+k\choose k}{2n+k-1\choose n} t^k(v-1)^k \Rgf_0^{n+1}. \label{Q0-R}
\end{align}
Recall that by duality (Lemma~\ref{lem:duality}), this is also the \gf\ of \emm colourful, labelled quadrangulations rooted from $0$ to $1$, counted by faces and local minima. By the Ambj\o rn-Budd bijection (Corollary~\ref{cor:AB}), this is twice the \gf\ $\Ggf$ of Eulerian orientations counted by edges and vertices. This gives the expression of $\Ggf$ in Theorem~\ref{thm:general}.

We finally want to prove that  the series $\Rgf\equiv\Rgf_0$ is D-algebraic in $t$ and  $v$, and that the same holds for $\Qgf_0$ (and thus for $\Ggf=\Qgf_0/2$). The argument is a bit more involved than in~\cite{mbm-aep1}, because the equation~\eqref{t-R0} that defines $\Rgf$ is of the form
  \beq\label{t-Omega0}
  t=\Omega(\Rgf,t(v-1)),
  \eeq
  rather than $t=\Omega(\Rgf)$ in our previous paper. We refer to Section~\ref{sec:DA-om0} for the proof.
This completes the proof of Theorem~\ref{thm:general}. \qed

\begin{Remark}\label{rem:om0}
  Note that in this case ($\om=0$), the series $\Fnn(y)$ defined by~\eqref{FM} is simply $\Rgf_0/(1-y)$ (see~\eqref{omis0eqforH}), hence a very simple rational series (in $y$). In particular, the alternative form of $\Qgf_0$ given by~\eqref{Q-alt} specialises to~\eqref{Q0-alt}, where we recall that $2\Ggf=\Qgf_0$.
  \end{Remark}

 We now discuss the {dominant singularity}, or critical point,  of $\Qgf_0(t,v)$ for fixed $v\in(0,\infty)$.
 
 \begin{Prediction}[{\bf Critical point}]
   \label{pred:0}
 Fix $v\in(0,\infty)$. Let $y_{c}\in(1/4,1)$ be the unique value satisfying
\beq\label{yc-char}
\frac{2\pi v}{v-1}=
%{2 \arccos \frac{1-3y_c}{2y_c^{3/2}}
6\arctan(\sqrt{4y_{c}-1})
+\frac{(1-y_{c})\sqrt{4y_{c}-1}}{y_{c}(2y_{c}-1)}.
\eeq
Then the dominant singularity of $\Qgf_0(t,v)$, as a function of $t$, is
$t_{c}=\frac{y_{c}(2y_{c}-1)}{v-1}$. At this point, $\Rgf=y_{c}(1-y_{c})^3:=\Rgf_c$.
\end{Prediction}
Note that exactly one value $y_c\in(1/4,1)$ satisfies~\eqref{yc-char}, because the left-hand side of~\eqref{yc-char} decreases from $0$ to $-\infty$ for $v \in (0,1)$, and then from $+\infty$ to $2\pi$ for $v \in (1, +\infty)$, while the right-hand side behaves similarly in terms of $y$, with the discontinuity  at $y_c=1/2$. Of course we take $y_c=1/2$ when $v=1$, and in this limit our prediction coincides with the known value $t_c=1/(4\pi)$ (see~\cite{mbm-aep1}). We have also checked it numerically by estimating $t_c$ from the first $38$ coefficients of $\Qgf_0$; see Figure~\ref{fig:radius0}. 

\begin{figure}[htb] 
  \centering
  \includegraphics[width=40mm]{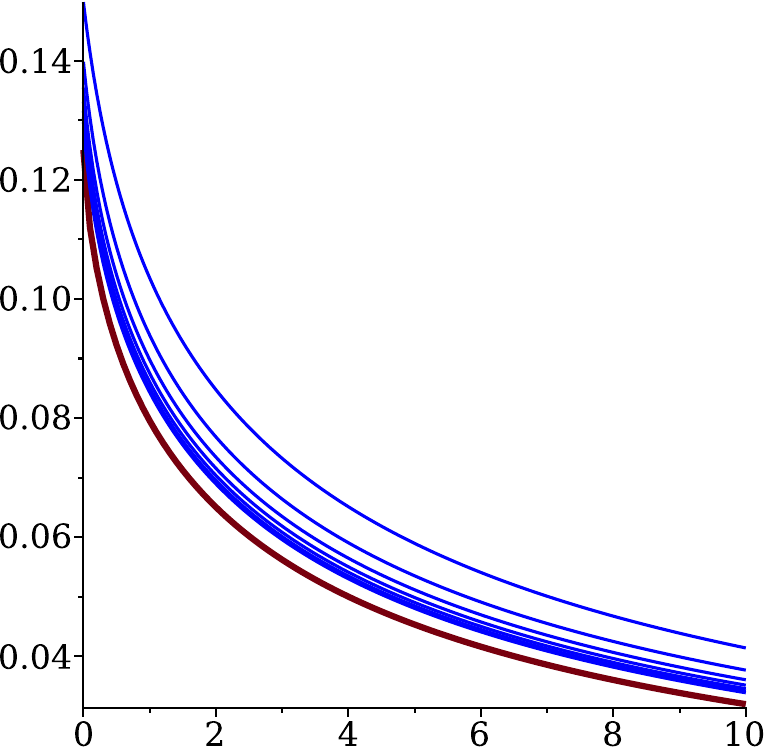}
  \caption{Bottom line: our prediction for the radius $t_c$ of $\Qgf_0$ in terms of $v$. Top lines: estimates of this radius obtained from ratios of the coefficients of $t^n$ and $t^{n+1}$ in $\Qgf_0$, for various values of $n\le 37$.}
  \label{fig:radius0}
\end{figure}

Our prediction is based on the assumption  that the discriminant occurring in $M(y)$ (see~\eqref{My-exp}) has a double root, denoted $y=y_{c}$, at the critical point $t_c$. This phenomenon is frequently observed in map enumeration, for instance when counting planar maps by edges and outer degree. This condition gives at once the above expressions of $t_c$ and $\Rgf_{c}$ in terms of % this value~
$y_c$.

Another natural way that leads to this assumption is to consider the series $M(y)$ at small values of $t>0$, recalling that $\Rgf=t+\LandauO(t^2)$. The discriminant occurring in~\eqref{My-exp} is a rational function in $y$, of degree $3$, with $3$ positive roots $Y_\pm$ and $Y_1$ that start as follows:
\[
  Y_\pm = (1\pm \sqrt v)^2t +\LandauO(t^2), \qquad Y_1=1-4t +\LandauO(t^2).
\]
The discriminant is  non-positive on the cuts $[Y_-, Y_+]$ and $[Y_1, 1]$, and positive elsewhere. Recalling that $M(y)$ is a \fps\ in $t$ with coefficients in $\qs_{\ge 0}[v]((y))$ (that is, with non-negative coefficients), we see that for small $t>0$, this series converges absolutely in the annulus $\{z : Y_+ <|z|<Y_1\}$. In particular, the initial condition $tv=[y^{-1} ] M(y)$ reads
\beq\label{minus1}
tv= \frac{1}{2\pi i}\oint M(y) dy,
\eeq
for any circle in the annulus. Our prediction is that this annulus shrinks to a circle of radius $y_c$ when $t$ approaches its critical value $t_c$.
  
In order to determine $y_c$ (and thus $t_c$), we will use the limit $t\rightarrow t_c$ of~\eqref{minus1},
% this identity,
integrating over the circle of radius $y_c$.  So assume that our assumption holds true, for some $y_c\in (1/4,1)$. Then, for  $t=t_c$ and $|y|=y_c$,  the expression~\eqref{My-exp} of $M(y)$ rewrites as
\beq\label{M0-extended}
   M(y)  =
  \frac{y+y_c(2y_c-1)}{2y}- \frac{y-y_c(2y_c-1)}{2y}\sqrt{1- \frac{4yy_c(1-y_c)^3}{(1-y)(y-y_c(2y_c-1))^2}}.
  \eeq
  The right-hand side  is a meromorphic function of $y$ on the portion of the complex plane where   $\left|(1-y)(y-y_c(2y_c-1))^2\right|>4|y|y_c(1-y_c)^3$.
  %%{(Ci-dessus, et ci-dessous, est-ce qu'il suffirait de dire que c'est une fonction méromorphe en dehors de la coupure $[(1-2y_c)^2,1]$, qu'il faut voir comme $[Y_-(t_c), 1]$, c'est-à-dire l'union des deux coupures trouvées pour $t$ petit ? et dire directement que ça se réécrit sous la forme (46) ?)}{Oui, mais il faut faire gaffe de si la racine carré et positive ou negative. Je pense que c'est ok comme il est.}
  This  portion has an unbounded component $\mathcal G(y_c)$, lying outside a pair of ``glasses'' meeting at the point $y_c$, as illustrated in Figure~\ref{fig:glasses} for various values of $y_c$ in $(1/4,1)$. The point at $y_c$ is what is left from the segment $[Y_+,Y_1]$ at $t=t_c$.  Also, for $y$ a large positive real in $\mathcal G(y_c)$, we can simplify the square root, so that the above function rewrites as:
  \beq\label{M0-simple}
  % M(y)   =
  \frac{y+y_c(2y_c-1)}{2y}- \frac{y-y_c}{2y}\sqrt{\frac{y-(1-2y_c)^2}{{y-1}}}.
\eeq
 This  is meromorphic  on $\mathbb{C}\setminus[(1-2y_c)^2,1]$, using the principal value of the square root, and the cut $[(1-2y_c)^2,1]$ does not intersect $\mathcal G(y_c)$ (Figure~\ref{fig:glasses}).  Hence the above function must coincide with~\eqref{M0-extended} in $\mathcal G(y_c)$, and with $M(y)$ when $|y|=y_c$.
Moreover, it has an explicit anti-derivative, namely
\[
  \widehat M(y):=
   y_{c}(2y_{c}-1)\log\left(
    \frac{y-1+2y_c+(y-1) \sqrt{\frac{y-(2y_{c}-1)^{2}}{y-1}}}{2y_c}\right)
    -\frac{y-1}2 \left(\sqrt{\frac{y-(2y_{c}-1)^{2}}{y-1}}-1\right),
\]
where we define $\log(\rho e^{i\theta})=\log \rho + i\theta$ for $\rho>0$ and $\theta \in(0, 2\pi)$. This choice makes the above function $\widehat M(y)$ analytic on $\cs\setminus [0, +\infty)$.

\begin{figure}[bht]
  \centering
 \includegraphics[width=35mm]{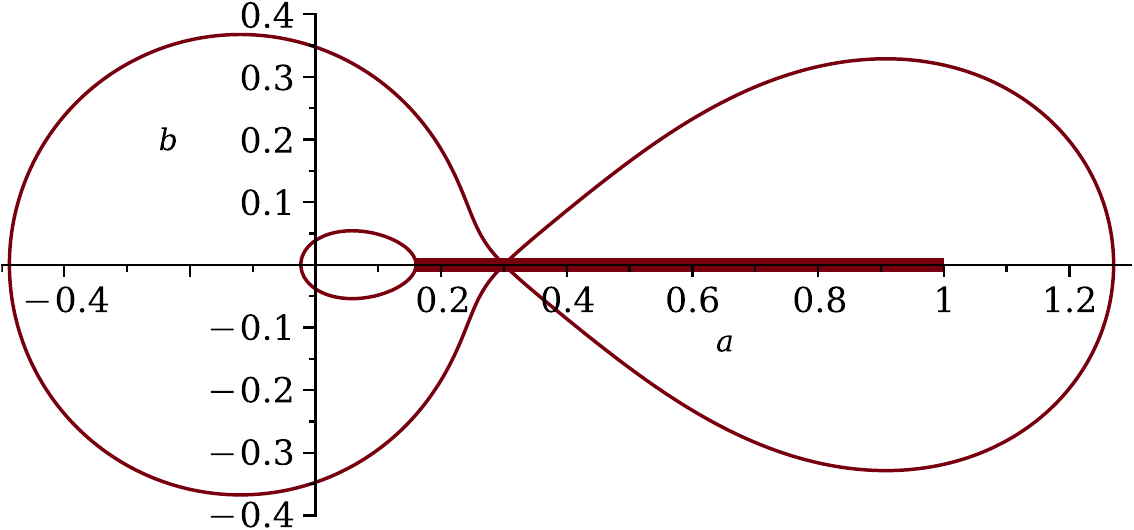}
 \includegraphics[width=35mm]{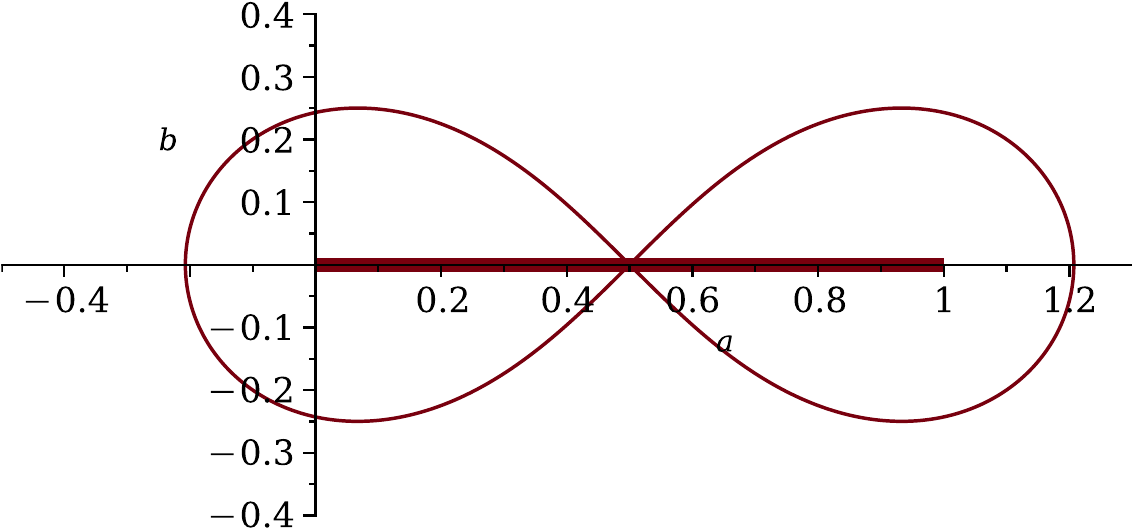}
  %\vskip -20mm
\includegraphics[width=35mm]{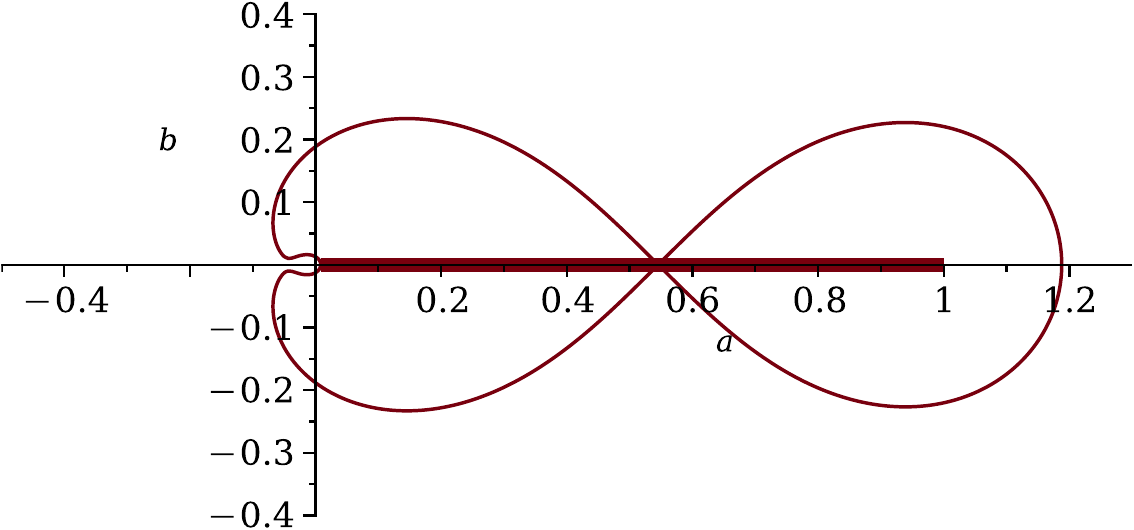}
\includegraphics[width=35mm]{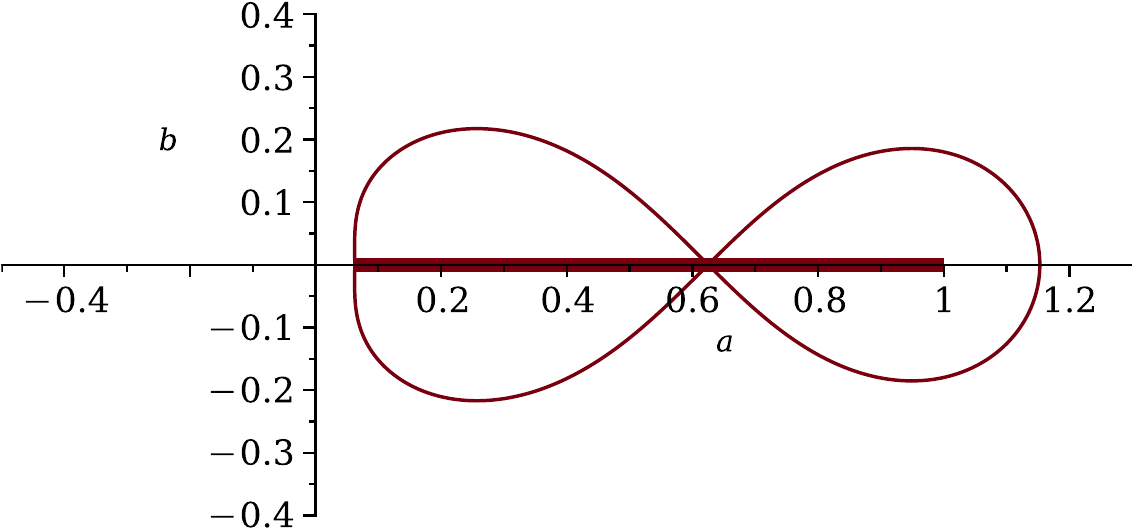}
  \caption{The function $M(y)$, for $t=t_c$, is analytic in the infinite region outside the curve. From left to right, $y_c=1/4+1/20, 1/2, 1/2+1/22, 1/2+1/8$. The cut of the function~\eqref{M0-simple} is also shown.}
  \label{fig:glasses}  
\end{figure}

We can then compute $t_{c}$ via the  integral~\eqref{minus1} taken on the circle of radius $y_{c}$ around 0: 
\begin{equation*}
  t_{c}v=y_c(2y_c-1) \frac v{v-1}%=[y^{-1}]M(y)
  =\frac{1}{2\pi i}\oint M(y) dy
  =\lim_{\epsilon\to 0^+}\frac{1}{2\pi i}\left(\widehat{M}(y_{c}-i\epsilon)-\widehat{M}(y_{c}+i\epsilon)\right).
\end{equation*}
 A careful expansion of $\widehat M(y)$ in the vicinity of $y_c$ yields~\eqref{yc-char}, using the fact that $\arccos \frac{1-3y_c}{2y_c^{3/2}}
=3\arctan(\sqrt{4y_{c}-1})$.

% ============================================================
\subsection{Solution for $\bm{\om=1}$: proof of Theorem~\ref{thm:quartic}}
\label{sec:om1}
% ============================================================
We now address the case $\om=1$, again with the hope that in this case $\overline\Znn(x,y)=1$.
Since the first factor of~\eqref{symmetriceq} is now symmetric in $x$ and $y$, this must be true of the product of the two factors involving $\Mnn(y)$.
This means that
\beq\label{omis1eqforH}
\left(y\Mnn(y)-t(v-1)\right)\left(1- y-\Mnn(y)\right)= \Rgf_1
\eeq
for  some series $\Rgf_1$ in $\GK[[t]]$.
% (and even $\qs[v][[t]]$).
Conversely, if this holds then~\eqref{symmetriceq} is $x/y$-symmetric. Using the first terms in the expansion of $\Mnn(x)$ (see~\eqref{M-expansion} and our {\sc Maple} session), we find that the series~$\Rgf_1$ should start as follows:
\[ %\beq\label{Rgf1-ser}
  \Rgf_1= 
  t -\left(1+2 v\right) t^{2}-2\left( v^{2}+4v +1\right) t^{3}+\LandauO(t^4).
\]% \eeq
Again, these first terms do not depend on $y$, which is encouraging.

The rest of the argument is very close to the case $\om=0$.  Let $R$ be \emm any, series in $ t\GK[[t]]$. The following quadratic equation in $M(y)$, derived from~\eqref{omis1eqforH},
\beq\label{eq-M1}
\left(yM (y)-t(v-1)\right)\left(1- y-M(y)\right)= R,
\eeq
admits two solutions. Both are \fps\ in $t$, but one of them has a non-zero constant term (more precisely, $1-y$), and
we are not interested in it. The other one has constant term $0$ and can be written as a series in $R$, $y$ and $t(v-1)$ as follows:
\begin{align}
  \label{My-exp1}
  M(y)%&=\frac{y(1-y)+t(v-1)}{2y} \left(1 - \sqrt{1-4y \frac{t(v-1)(1-y)+R}{(y(1-y)+t(v-1))^2}}\right)  \\
  &= \frac{y(1-y)+t(v-1)}{2y} -\frac{y(1-y)-t(v-1)}{2y}\sqrt{1- \frac{4yR}{(y(1-y)-t(v-1))^2}}\\
      &= \frac{tv-t}{y}+\sum_{n,k,j\geq0}\frac{1}{n+1}{2n\choose n}{2n+k\choose k}{2n+k+j\choose j}t^k(v-1)^{k}R^{n+1}y^{j-n-k-1}. \nonumber
\end{align}
Since we have taken $R$ with no constant term,  this is a series of $t\GK((y))[[t]]$. 

Now, %following the conditions of Proposition~\ref{prop:M},
we would like the coefficient of $y^{-1}$ in $M(y)$ to be $tv$. Equivalently,
\beq\label{t-R1}
t=\sum_{k,n\geq0}\frac{1}{n+1}{2n\choose n}{2n+k\choose k}{3n+2k\choose n+k}R^{n+1}(tv-t)^{k}.
\eeq
Again, this condition defines a \emm unique, series $R$ with constant term $0$. We finally conclude, as in the case $\om=0$, that this unique solution is indeed the series $\Rgf_1$ of~\eqref{omis1eqforH}, and that the corresponding value of $M(y)$ is $\Mnn(y)$.

According to Proposition~\ref{prop:M},  the series $\Qgf_1$ that counts quartic Eulerian orientations by vertices  ($t$) and clockwise faces ($v$) is given by
\[
  \Qgf_1=-v+\frac{1}{t^2}\sum_{k,n\geq0,~n+k>0}\frac{1}{n+1}{2n\choose n}{2n+k\choose k}{3n+2k-1\choose n+k-1}t^k(v-1)^{k}\Rgf_1^{n+1}.
\]
This gives the expression  of Theorem~\ref{thm:quartic}. By the Ambj\o rn-Budd bijection (Proposition~\ref{prop:bij}) and duality, this is also the \gf\ of Eulerian partial orientations counted by edges and vertices. For instance, $\Qgf_1=v(v+3)t+ \LandauO(t^2)$, and there are indeed $4$ Eulerian partial orientations with one edge: the loop map, with one vertex, admits $3$ partial orientations, while the link map, with two vertices, has a single one. We refer to Section~\ref{sec:DA-om1} for the statements dealing with D-algebraicity.
\qed

\medskip

   \begin{Remark}\label{rem:om1}
      In this case ($\om=1$), the series $\Fnn(y)$ defined by~\eqref{FM} is simply $\Rgf_1$ (see~\eqref{omis1eqforH}), hence is independent of $y$. In particular, the expression \eqref{Q-alt} specialises to~\eqref{Q1-alt} at $\omega=1$.  Along with Remark~\ref{rem:om0}, this means that $\Fnn(y)$ is a rational function of $y$  and $\Rgf:=\Fnn(0)$ for $\omega=0,1$ and any $v$. In Section~\ref{sec:Tutte}, we prove that $\Fnn(y)$ is an \emm algebraic, function of $y$ and $\Rgf$ when $v=1$ and $\om \neq -2,2$ is of the form $-2\cos(k\pi/m)$ for $k,m\in\mathbb{Z}$.
   \end{Remark}

   To finish, let us give our prediction for the dominant singularity of $\Qgf_1$, as a function of $t$, for fixed $v$.

   \begin{Prediction}[{\bf Critical point}]
  Fix $v\in(0,\infty)$. Let $y_{c}\in(1/6,1/2)$ be the unique value satisfying
\beq\label{yc-char1}
\frac{\pi v}{v-1}=2\, 
\arccos \left(\frac{1-4y_c}{2y_c}\right)
% \arctan\left(\sqrt{\frac{6y_c-1}{1-2y_c}}\right)
+\frac{\sqrt{(6y_c-1)(1-2y_c)}}{4y_{c}(3y_{c}-1)}.
\eeq
Then the dominant singularity of $\Qgf_1(t,v)$, as a function of $t$, is
$t_{c}=\frac{y_{c}(3y_{c}-1)}{v-1}$. At this point, $\Rgf=y_{c}(1-y_{c})^2:=\Rgf_c$.    
\end{Prediction}

  Again, the agreement with estimates derived from the first $38$ coefficients of $\Qgf_1$ is excellent, and at $v=1$ our prediction coincides with the known value $1/(4\sqrt 3 \pi)$.

  As in the case $\om=0$, our prediction relies on the assumption  that the discriminant occurring in~\eqref{My-exp1} has a double root $y_c$ when $t=t_c$ (for small $t>0$, it has $4$ positive roots $Y_-<Y_+ < Y_1 <1 <Y_2$, and the series $M(y)$ converges absolutely in the annulus $\{z: Y_+<|z|< Y_1\}$). This gives at once the above expressions of $t_c$ and $\Rgf_c$.
  Then, we determine $t_c$ (and $y_c$) in terms of $v$ using the fact that $tv=[y^{-1}]M(y)$. We do not give all details, but here are some counterparts of the explanations given in the case $\om=0$.  For $|y|=y_c$, the expression~\eqref{My-exp1} can be rewritten as
  \[
    M(y)= \frac{y(1-y)+y_c(3y_c-1)}{2y} + \frac{y-y_c}{2y} S(y),
  \]
  with
  \[
    S(y)=(y+y_c-1)\sqrt{ \frac{y^2+2y(y_c-1)+(3y_c-1)^2}{(y+y_c-1)^2}}.
  \]
  The right-hand side is meromorphic   in $\cs\setminus[1-y_c -\sqrt{y_c(1-2y_c)}, 1-y_c +\sqrt{y_c(1-2y_c)}]$.
The above expression for $M(y)$ has an explicit  antiderivative:
\[
  \widehat M(y):=
  y_{c}(3y_{c}-1)\log\left(1-y-3y_{c}-S(y)\right)-\frac{1}{4}(1+y_{c}-y)S(y)+\frac{2y-y^{2}}{4}.
\]
This time we take the principal value of the logarithm, which makes  the above function analytic on  $\cs\setminus[0, +\infty)$. Integrating $M(y)$ on the circle of radius $y_c$ centered at $0$ gives
\[
  t_cv=y_c(3y_c-1) \frac v{v-1}= \lim_{\epsilon\to 0^+}\frac{1}{2\pi i}\left(\widehat{M}(y_{c}-i\epsilon)-\widehat{M}(y_{c}+i\epsilon)\right),
\]
and a careful analysis of $\widehat M$ near $y_c$ yields~\eqref{yc-char1}.

%%%%%%%%%%%%%%%%%%%%%%%%%%%%%%%%%%%%%%%%%%%%%%%%%%%%%%%%%%%%%%%%%%%%%%%%%% 
   \section{Solution for $\bm{v=1}$:  the six vertex model}
\label{sec:six}
%%%%%%%%%%%%%%%%%%%%%%%%%%%%%%%%%%%%%%%%%%%%%%%%%%%%%%%%%%%%%%%%%%%%%%%%%% 
Our starting point is Proposition~\ref{prop:F}. We are looking for the unique series $F(x)$ in $t\GK[[x,t]]$ (where now $\GK=\qs(\om)$) such that  the series $M(x)$ that solves
\beq\label{FM-v=1}
F(x):=xM(x)(1-\om x-M(x))
\eeq
and is a multiple of $t$, namely,
\beq\label{MF}
M(x):= \frac{1-\omega x} 2\left( 1 - \sqrt {1- \frac{4F(x)}{x(1-\omega x)^2}}\right),
\eeq
also satisfies
\beq\label{F-prop-v=1}
F(M(x))= xM(x)(1-x-\omega M(x))
% \left(1-\frac{\Fnn(y)}{xy(1-x-\omega y)}\right)
% /\left(1-\frac{\Fnn(x)}{xy(1-y-\omega x)}\right),
\eeq
% expanded as a series in $t$, has coefficients in $\GK[[x,y]]$.
and $[x^{-1}]M(x)=t$. As already noted, this forces  $M(x)$ to be a series in $t$ with coefficients in $\GK((x))$, and more precisely a series of $t/x\GK[[t/x,x]]$.

The solution of this case is significantly less direct than in the cases $\om=0$ and $\om=1$.
We will first construct a collection of series $F(x)$ satisfying~\eqref{FM-v=1} and~\eqref{F-prop-v=1}, indexed by an arbitrary series $q\equiv q(t)\in t\GK[[t]]$. This is done in Section~\ref{sec:ansatz-rigorous}, while Section~\ref{sec:ansatz} presents our Ansatz for reaching this solution. Finally, in Section~\ref{sec:qt} we determine the series $q(t)$ that  guarantees that $[x^{-1}]M(x)=t$.
  
In Appendix~\ref{app:complex}, we give an alternative argument to deduce the same result using complex analysis, somewhat along the lines of \cite{kostov,elvey-zinn}. In particular, this argument justifies our Ansatz $\Mnn(\chi(z))=\chi(z-\gamma)$ (see~\eqref{ansatz} below), even beyond the  case $v=1$.

% ============================================================
\subsection{An Ansatz} \label{sec:ansatz}
% ===============================================================
As explained in the introduction, it is convenient to write $\om=-2\cos(2\alpha)$.
We will make an \emm Ansatz, on the form of $M(x)$, and hence of $F(x)$, by assuming that, up to a change of variables,~$M$ acts as a shift. That is,
\beq\label{ansatz}
  M(\chi(z)) =\chi(z-\gamma)
\eeq 
for some series $\chi(z)$, depending on $z$ and $t$, while $\gamma$ would only depend on $t$. Then by~\eqref{FM-v=1}, we would have
\beq\label{Fchi}
F(\chi(z))=\chi(z)\chi(z-\gamma) (1-\om \chi(z) -\chi(z-\gamma)),
\eeq
so that, upon replacing $z$ by $z-\gamma$,
\[
  F(M(\chi(z)))=   F(\chi(z-\gamma))=\chi(z-\gamma)\chi(z-2\gamma) (1-\om \chi(z-\gamma) -\chi(z-2\gamma)).
\]
On the other hand, by~\eqref{F-prop-v=1} we should have
\[
  F(M(\chi(z))) =  \chi(z)\chi(z-\gamma)(1-\chi(z)-\omega \chi(z-\gamma)).
\]
Comparing the last two identities shows that the change of variables $\chi$ should satisfy either $\chi(z-2\gamma)=\chi(z)$ or
\[
  \chi(z)+\omega \chi(z-\gamma) + \chi(z-2\gamma)=1,
\]
or equivalently
\beq\label{diff}
\chi(z+2\gamma)+\omega \chi(z+\gamma) + \chi(z)=1.
\eeq
We will proceed with the second choice (the first one did not appear to lead to a solution to the original equation). Combining~\eqref{Fchi} and~\eqref{diff} yields}
\beq\label{F-Lambda}
F(\chi(z))= \chi(z-\gamma) \chi(z)\chi(z+\gamma).
\eeq

\begin{Remark}\label{remark:v_not_1}
  In the general case where $v$ is not necessarily equal to $1$, we prove in Appendix~\ref{app:complex} (Proposition~\ref{prop:Xz_characterisation}) that, at an analytic level, the connection~\eqref{ansatz} between $\Mnn$ and a meromorphic function $\chi$ indeed holds true, but that the relation~\eqref{diff} should be replaced by
  \[
 \Xz(\gamma+z)+\omega\Xz(z)+\Xz(z-\gamma)=1+\frac{t(v-1)}{\Xz(z)}   .
\]
It would be extremely interesting to solve this equation, as we solve the $v=1$ case below.
   \end{Remark}

Let us return to the case $v=1$. Our first objective here is to find a function $\chi(z)$ that not only satisfies the difference equation~\eqref{diff}, but has also interesting properties as a series in $z$, so that~\eqref{F-Lambda} can be used to characterise $F(x)$.

One obvious solution of the difference equation~\eqref{diff} is the constant function
\[
  \chi_0(z):= \frac 1 {2+\omega}= \frac1{4\sin^2\!\alpha}.
\]
Hence we can focus on the homogeneous equation
\[
  \psi(z+2\gamma)+\omega \psi(z+\gamma) + \psi(z)=0,
\]
which can be written as
\[
  \left( \psi(z+2\gamma)- e^{2i\alpha}  \psi(z+\gamma)   \right)
  - e^{-2i\alpha} \left( \psi(z+\gamma)- e^{2i\alpha}  \psi(z)   \right)=0,
\]
or as
\[
  \left( \psi(z+2\gamma)- e^{-2i\alpha}  \psi(z+\gamma)   \right)
  - e^{2i\alpha} \left( \psi(z+\gamma)- e^{-2i\alpha}  \psi(z)   \right)=0.
\]
These two forms allow us to reduce the order to $1$, and to consider the equations
\beq\label{order1}
\psi_+(z+\gamma)= e^{2i\alpha}  \psi_+(z)   \qquad \text{and}\qquad  \psi_-(z+\gamma)= e^{-2i\alpha}  \psi_-(z) .
\eeq
Solutions to such equations are known in terms of Jacobi's theta functions. More precisely, let us introduce the following function, closely related to Jacobi's  function $\vartheta_1(z,q^{1/2})$ (see~\cite[Ch.~16]{AS}):
\begin{align}
  \theta(z)&:= \sum_{n\ge 0} (-1)^n q^{n(n+1)/2} \sin ((2n+1)z) \label{theta-def}
  \\
           &= \sin z \prod_{n\ge 1} (1-q^n) \left(1-q^n e^{-2iz}\right) \left(1-q^n e^{2iz}\right).  \label{theta-def-prod}
\end{align}
So far we consider $\theta(z)$, classically, as an analytic function of two complex variables $q$ and $z$, with $|q|<1$.   Then $\theta(-z)=-\theta(z)$ and, if we write $q=e^{2i\gamma}$ (with $\Im \gamma >0$), then it follows from~\eqref{theta-def-prod} that
\[
  \theta(z+\gamma)= - e^{-2iz-i\gamma} \theta(z).
\]
 These two identities imply that the series
\[
  \psi(z):= \frac{\theta(z+\alpha)}{\theta(z)}
\]
satisfies the second equation in~\eqref{order1}, while
\[
  \psi(-z)=\frac{\theta(z-\alpha)}{\theta(z)}
\]
satisfies the first one. We will thus consider solutions $\chi(z)$ of the second order equation~\eqref{diff} of the form
\[
  \chi(z)= c_+ \psi(z)+ c_-\psi(-z) + \frac1{4\sin^2\!\alpha}.
\]
Note that $\psi(z)$ has a pole of order 1 at $z=0$. However, in order to invert~\eqref{F-Lambda}, we would rather handle a function $\chi$ that vanishes at $z=0$. This leads us to choose $c_-=c_+:=c$ (so as to cancel the pole), and then to choose $c$ such that $\chi(0)=0$. After a few calculations based on the definition of $\psi$ in terms of $\theta$, we are led to
\beq\label{c-def}
c= -\theta'(0)/(8\sin^2\!\alpha \, \theta'(\alpha)),
\eeq
so that
\begin{equation}\label{eq:chi_formula}
  \chi(z)=  \frac1{4\sin^2\!\alpha}\left( 1 -\frac {\theta'(0)}{2 \theta'(\alpha)}\left( \psi(z)+\psi(-z)\right)\right).
\end{equation}

%{Below I  stated that coefficients are integers, as you had proved that, which is nice, but not relevant for our purpose. Moving to $\qs$ instead of $\zs$ would make things a little bit lighter. Tell me if you prefer to keep this or drop it. What is important though is to control the degrees of the polynomials $\Pol_n$. Otherwise we cannot say that $\chi$ is a series in $q$ and $u=4q \sin ^2 z$ I think. I have added something to that effect.}

\begin{Lemma}\label{lem:chi}
  The above series $\chi$ is a formal power series in $q$ with coefficients that are polynomials in $\omega=-2\cos(2\alpha)$ and $\sin ^2 \! z$. More precisely, it expands as
  \beq\label{chi_expansion}
  \chi(z)= 4 q \sin^2\! z \left( 1 + \sum_{n\ge 1}q^n\Pol_{n}(4\sin^2\! z, \omega)\right)
  \eeq
  where each $\Pol_n$ is a bivariate polynomial with integer coefficients, of degree at most $n$ in its first variable. Hence $\chi$ can also be seen as a \fps\ in $q$ and $u:= 4 q \sin^2\! z$ with coefficients in~$\zs[\omega]$.

  Finally,
  \beq\label{order1a}
  \lim_{z\rightarrow 0} \frac {\chi   (z)}{\sin^2\! z}=
  % \frac{\theta'(a)\theta^{(3)}(0)-\theta'(0)\theta^{(3)}(a)}{24 \theta'(0)\theta'(a) \sin ^2\! \alpha}      =
  \frac{1}{24\sin ^2\! \alpha} \left( \frac{\theta^{(3)}(0)}{\theta'(0)}-\frac{\theta^{(3)}(\alpha)}{\theta'(\alpha)}\right)= 4q+\LandauO(q^2).
  \eeq
\end{Lemma}

\begin{proof}
  Let us first observe  that for all $n$, $\frac{\cos((2n+1)\alpha)}{\cos\alpha}$ is a polynomial in $\omega=-2\cos(2\alpha)$,
  and $\frac{\sin((2n+1)z)}{\sin z}$ is a polynomial in $4\sin^{2}\!z$ of degree $n$, both with integer coefficients.

  We start by showing that $(\omega+2)\chi$ is a formal power series in $\mathbb{Z}[4\sin^{2}\!z,\omega][[q]]$. 
  From the definition~\eqref{theta-def}, we have
  \begin{align}
    \frac{\theta(z)}{\sin z}&=\sum_{n\ge 0} (-1)^n q^{n(n+1)/2} \,\frac{\sin ((2n+1)z)}{\sin z}&&\in 1+q\mathbb{Z}[4\sin^2\! z][[q]], \nonumber\\
    \frac{\theta(z+\alpha)+\theta(z-\alpha)}{2\sin z\cos \alpha}&=\sum_{n\ge 0} (-1)^n q^{n(n+1)/2} \,\frac{\sin ((2n+1)z)\cos((2n+1)\alpha)}{\sin z \cos \alpha}&&\in 1+q\mathbb{Z}[4\sin^2\! z,\omega][[q]], \nonumber\\
    \frac{\theta'(\alpha)}{\cos \alpha}&=\sum_{n\ge 0} (-1)^n q^{n(n+1)/2} (2n+1)\frac{\cos((2n+1)\alpha)}{\cos \alpha}&&\in 1+q\mathbb{Z}[\omega][[q]],  \label{exp3}\\
    \theta'(0)&=\sum_{n\ge 0} (-1)^n q^{n(n+1)/2}(2n+1)&&\in 1+q\mathbb{Z}[[q]]. \label{exp4}
  \end{align}
  Moreover, in these series the exponent of $\sin^2\!z$ never exceeds the exponent of $q$.  It follows that the series
  \[
    (\omega+2)\chi(z):=1-\frac{\theta'(0)}{2\theta'(\alpha)}\cdot\frac{\theta(z+\alpha)+\theta(z-\alpha)}{\theta(z)}
    =
    1 - \theta'(0)\cdot \frac{\cos \alpha}{\theta'(\alpha)} \cdot \frac{\theta(z+\alpha)+\theta(z-\alpha)}{2\sin z\cos \alpha} \cdot \frac{\sin z}{\theta(z)}
  \]
  belongs to  $q\mathbb{Z}[4\sin^2\! z,\omega][[q]]$, with again the same bound on the exponents of $\sin^2\!z$.

  We note additionally that the expression above
  converges to $0$ when either $\alpha\to 0$ (that is $\omega\to -2$) or $z\to 0$, so that the coefficient of $q^n$ in this series must be a multiple of  $(\om+2)4 \sin ^2\!z$ (again with integer coefficients).  That is, $\chi(z)\in 4q\sin^2\! z\mathbb{Z}[4\sin^2\! z,\omega][[q]]$, and again the exponent of $\sin^2\!z$ never exceeds the exponent of~$q$. Simply computing the first term of this series explicitly yields~\eqref{chi_expansion}, while expanding $\chi(z)$ as a series in $z$ using $\theta''(0)=\theta(0)=0$ yields~\eqref{order1a}.
\end{proof}

% =============================================================
\subsection{Construction of  $\bm {F(x)}$}
\label{sec:ansatz-rigorous}
% ======================================================
Guided by the above considerations, we will now construct from scratch  a collection of series $F(x)$ (and $M(x)$) satisfying
Conditions~\eqref{FM-v=1} and~\eqref{F-prop-v=1}. They depend on a variable $q$, which will be  specialised later to a well-chosen series in $t$, so as to satisfy $[x^{-1}]M(x)=t$ (Section~\ref{sec:qt}).  Underlying this construction
are three levels of Laurent series in $q$, whose coefficients involve respectively:
\begin{itemize}
\item  the variable $z$ of the previous subsection,
\item a variable $u$ that will be often instantiated to $4q\sin^2\!z$ as in  Lemma~\ref{lem:chi},
\item the original variable $x$.
\end{itemize}
The series involving $z$ will be denoted with greek letters (as above), those involving $u$ will be denoted in roman capital  letters (not italicised) and those involving $x$ in the usual font.
In the first part of this subsection, we define various \fps\ in $q$, with coefficients in $\GK(u)$, that are formal counterparts of the series of the previous subsection, in the sense that they specialise to these series when $u=4q\sin^2\!z$. Then in Proposition~\ref{prop:F-construct}, we perform the final step that takes us to series in $q$ with coefficients in $\GK(x)$.

\smallskip
\noindent{\bf Analogue of $\bm{e^{2iz}}$.}   First, we introduce a series of $\qs(u)[[q]]$  that corresponds to $e^{2iz}$ when $u=4q\sin^2\!z$:
\[
  \Erm( u)=1-\frac u {2q} \left({1-\sqrt{1-4q/u}}\right)
  = - \sum_{n\ge 1} \frac 1 {n+1}\binom{2n}{n} \frac{q^n}{u^n} .
\]
The notation $\Erm(u)$ should be reminiscent of the exponential $e^{2iz}$.
Note that
\beq\label{algE}
u=q\left(2-\Erm( u)-\frac 1{\Erm( u)}\right).
\eeq
Consequently, any symmetric polynomial in $\Erm(u)$ and $1/\Erm(u)$  (say, with coefficients in $\qs$ and degree $d$), is a polynomial in $u/q$ of degree $d$.

\smallskip
\noindent{\bf Analogue of the shift.}    The  shift transformation $z\mapsto z-\gamma$ of the previous subsection gives $e^{2iz} \mapsto q^{-1}e^{2iz}$ (since we had taken $q=e^{2i\gamma}$). It should thus correspond to $\Erm (u)\mapsto q^{-1} \Erm(u)$, hence, by the above identity, to $u\mapsto \Srm(u)$ with 
\begin{align}
  \Srm (u)&:=2q -  \Erm( u)- \frac {q^2}{\Erm( u)}  \label{SE}\\
          &=\frac{(1+u)^2}u q + \LandauO(q^2).\nonumber
\end{align}

We now examine the effect of the shift $u\mapsto \Srm(u)$ in certain series involving $\Erm(u)$.

\begin{Lemma}\label{lem:shift}
  Given a series $\Grm(u)\in \GK[[u,q]]$,
  the substitution $u\mapsto \Srm(u)$ in $\Grm(u)$ is well defined, as a series of $\GK(u)[[q]]$.

  Let $\Hrm(v)$ be a series of $\GK[v+1/v][[q]]$, say $\Hrm(v)=\sum_{m\ge 0} \Hrm_m(v+1/v) q^m$ for some polynomials~$\Hrm_m$, and assume that
  \[
    \forall m\ge 0,\    \deg \Hrm_m \le m  \qquad \text{and} \qquad m-\deg \Hrm_m\rightarrow \infty.
  \]
  Then  $\Grm(u):=\Hrm(\Erm( u))$ is a well-defined series in $\GK[u][[q]]$. Moreover, $\Grm(\Srm(u))=\Hrm(q^{-1}\Erm( u))$.
\end{Lemma}
\begin{proof}  
  The first point follows from the fact that $\Srm(u)$ is a series in $q$ (with coefficients in $\qs(u)$) that has no constant term.

  Now let $\Hrm(v)$ satisfy the assumptions of the lemma. As noticed above, it follows from~\eqref{algE} that $\Hrm_m( \Erm+1/\Erm)=\Hrm_m(2-u/q)$    is a polynomial in $u/q$ of degree $\deg \Hrm_m$. The assumptions on the degree of~$\Hrm_m$ then imply that $q^m \Hrm_m( \Erm+1/\Erm)$ is a polynomial in $q$ and $u$, and that its valuation in $q$ tends to infinity as $m$ increases. Hence
    \[
      \Grm(u)= \sum_{m\ge 0} q^m \Hrm_m(2- u/q)
    \]
    is well defined in $\GK[u][[q]]$.
  In particular, $\Grm(\Srm(u))$ is well defined by the first statement of the lemma.  We have
  \[
    \Grm(\Srm(u))= \sum_{m\ge 0} q^m \Hrm_m(2- \Srm(u)/q) =\sum_{m\ge 0} q^m \Hrm_m(\Erm(u) /q+ q/\Erm(u))
  \]
  by~\eqref{SE}, and this proves the final statement.
\end{proof}

\smallskip
\noindent{\bf Analogue of the series $\bm \theta$.} We now introduce  a series $\Trm(u)$ that is  the counterpart of $\theta(z)$ (up to a factor $\sin z$):
\[
  \Trm(u)=\sum_{n\ge 0}(-1)^n q^{n(n+1)/2}\, \frac{\Erm( u)^{n+1}-\Erm( u)^{-n}} {\Erm( u)-1} .
\]
Note that
\[
  \frac{\Erm( u)^{n+1}-\Erm( u)^{-n}} {\Erm( u)-1}= \frac 1 {\Erm(u)^n} + \cdots + \Erm(u)^n
\]
is a symmetric polynomial of degree $n$ in $\Erm(u)$ and $1/\Erm(u)$, hence a polynomial in $u/q$ of degree~$n$.  By Lemma~\ref{lem:shift},  $\Trm(u)$ is a series of $\qs[u][[q]]$.
It starts as follows:
\[
  \Trm(u)=1+u +(u^2-3) q+ \LandauO(q^2).
\]
By construction,
  \beq\label{T-theta}
  \sin z \, \Trm(4q\sin^2\!z)= \theta(z),
  \eeq
where $\theta(z)$ is defined by~\eqref{theta-def}.
We also define the following counterparts of $\theta(z+\alpha)+\theta(z-\alpha)$ and $\theta(z+\alpha)-\theta(z-\alpha)$:
\[
  \Trm^{+}(u):=2\sum_{n\ge 0}(-1)^n q^{n(n+1)/2}\cos((2n+1)\alpha)\,
  \frac{\Erm( u)^{n+1}-\Erm( u)^{-n}} {\Erm( u)-1},
\]
\[
  \Trm^{-}(u):=-2 \sum_{n\ge 0}(-1)^n q^{n(n+1)/2}\sin((2n+1)\alpha)
  % (e^{(2n+1)i\alpha}-e^{-(2n+1)i\alpha})
  \frac{\Erm( u)^{n+1}+\Erm( u)^{-n}}{\Erm( u)+1}.
\]
Lemma~\ref{lem:shift} applies again, and both series (in $q$) have polynomial coefficients in $u$. More precisely, given that  $\omega=-2\cos(2\alpha)=2(1-2 \cos^2\! \alpha)$, the former series has coefficients  in $\cos\alpha \qs[\cos^2\alpha, u]=\cos\alpha \qs[\omega, u]$
and the latter one in $\sin\alpha \qs[\cos^2\alpha, u]=\sin\alpha \qs[\omega, u]$. The first terms read:
  \[
    \frac{ \Trm^+(u)}{\cos \alpha}=2  (1-u(\om+1))+ \LandauO(q), \qquad
    \frac{\Trm^{-}(u)}{\sin \alpha }= 2(-1+u(\om-1)) + \LandauO(q).
  \]
  By construction,
  \beq\label{Tpm-theta}
  \sin z\  \Trm^+(4q\sin^2\!z)=\theta(z+\alpha)+\theta(z-\alpha) \quad \text{and } \quad
  -\cos z\ \Trm^-(4q\sin^2\!z)=\theta(z+\alpha)-\theta(z-\alpha)  .
  \eeq 
Lemma~\ref{lem:shift} also allows us to determine the values of the three $\Trm$-series at $\Srm(u)$.

\begin{Lemma}\label{lem:shiftT}
  We have the following identities:
  \begin{align*}
    (q-\Erm)  \Trm(\Srm(u))&={\Erm (\Erm-1)} \Trm(u), \\
    % \cos(2\alpha)     (q-\Erm) \Trm^+(\Srm(u)) + \sin(2\alpha) (q+\Erm) \Trm^-(\Srm(u))=
    % \Erm(\Erm-1) \Trm^+(u),
    (q-\Erm) \Trm^+(\Srm(u))&=   \cos(2\alpha)     \Erm (\Erm-1)\Trm^+(u)+  \sin (2\alpha)\Erm(\Erm+1)\Trm^-(u),     
    \\
    (q+\Erm) \Trm^-(\Srm(u))&=   \sin(2\alpha)     \Erm (\Erm-1)\Trm^+(u)-  \cos(2\alpha)\Erm(\Erm+1)\Trm^-(u).
  \end{align*}
\end{Lemma}
\begin{proof}
  Use Lemma~\ref{lem:shift} and the expressions of $\Trm$, $\Trm^+$ and $\Trm^-$ in terms of $\Erm$.
\end{proof}

\smallskip
\noindent {\bf Analogue of the series $\bm \chi$.} We  now introduce the following series in $q$:
\beq\label{Xdef}
\Xrm(u):=c\frac{\Trm^{+}(u)}{\Trm(u)}+ \frac 1{4\sin^2\!\al}
=\frac{1}{4\sin^{2}\alpha}\left(1-\frac{\Trm(0)}{\Trm^{+}(0)}\cdot \frac{\Trm^{+}(u)}{\Trm(u)}\right),
\eeq
where $c$ is given by~\eqref{c-def}.
From the second expression, this series vanishes at $u=0$.
The first term reads
\[
  \Xrm(u)=\frac u {1+u} + u\LandauO(q),
\]
Let us prove that this series has coefficients in $\qs[\omega, u, 1/(1+u)]\subset \qs[\omega][[u]]$. First, it follows from the above properties of $\Trm$ and $\Trm^+$,
that  $  4\sin^2\!\alpha \Xrm(u)$ has coefficients in $\qs[\om, u, 1/(1+u)]$. Then, if we specialise at $\al=0$, that is, $\om=-2$, the definition of this series, we find that it vanishes for this value of $\alpha$ (because $\Trm^+=2\Trm$ when $\alpha=0$). Given that $\om+2=4\sin^2\! \al$, this proves our claim.  Also, observe that  $\Xrm(4q\sin^2\!z)$ is  $\chi(z)$ (cf. Lemma~\ref{lem:chi}).
We finally denote
\[
  \Xrm^+(u)= c \,\left( \cos 2\alpha\, \frac{ \Trm^{+}(u)}{\Trm(u)} +{\sin 2\alpha}{\sqrt{1-4q/u}}\, \frac{\Trm^-(u)}{\Trm(u)}\right)+  \frac 1{4\sin^2\!\al},
\]
and
\beq\label{X-minus}
\Xrm^-(u)= c \,\left( \cos 2\alpha\, \frac{ \Trm^{+}(u)}{\Trm(u)} -{\sin 2\alpha}{\sqrt{1-4q/u}}\, \frac{\Trm^-(u)}{\Trm(u)}\right)+  \frac 1{4\sin^2\!\al}.
\eeq
Both series have coefficients in $\GK(u)$. More precisely, we can argue as above to prove that they have coefficients in $\qs[\om, u, 1/u, 1/(1+u)]$ (and in fact the coefficients of $\Xrm^-(u)$ have no pole at $u=-1$). We have
\beq\label{Xpm-exp}
\Xrm^+(u)=\frac{1-u+4u\cos^2\!\alpha }{1+u}+ \LandauO(q),
\quad\text{and}\quad
\Xrm^-(u)=\frac{(1+u)^2}u q+\LandauO(q^2).
\eeq
The specializations of these series at $u=4q\sin^2\!z$ are \emm not, well defined, but they should be thought of as counterparts of $\chi(z+\gamma)$ and $\chi(z-\gamma)$, respectively. In particular, observe that
\beq\label{eqfunc-X}
\Xrm^+(u)+\omega \Xrm(u)+ \Xrm^-(u)=1,
\eeq
which is the counterpart of $\chi(z+\gamma)+\om \chi(z)+\chi(z-\gamma)=1$.

\smallskip
Let us now examine the effect of the shift $u\mapsto \Srm(u)$ on the series $\Xrm(u)$ and $(\Xrm^-\Xrm^+)(u)$.
\begin{Lemma}\label{lem:shiftX} 
  We have 
  \[
    \Xrm(\Srm(u))= \Xrm^-(u) \qquad \text{and} \qquad (\Xrm^-\Xrm^+)(\Srm(u))= \Xrm(u)(1-\Xrm(u)-\omega \Xrm^-(u)).
  \]
\end{Lemma}
\begin{proof}
  First, $\Xrm(\Srm(u))$ is well defined since $\Xrm(u)$ is a series in $q$ and $u$ (see Lemma~\ref{lem:shift}). Moreover,
  \[
    \Xrm(\Srm(u))= c \frac{\Trm^+(\Srm(u))}{\Trm(\Srm(u))} + \frac 1 {4\sin ^2\alpha}.
  \]
  We then use the expressions of $ \Trm^+(\Srm(u))$ and $\Trm(\Srm(u))$ given in Lemma~\ref{lem:shiftT}, and conclude using
  \[
    \frac{1+\Erm}{1-\Erm}=   \sqrt{1-4q/u}.
  \]

  We proceed similarly for the series
  \beq\label{Xpm}
  u \cdot (\Xrm^-\Xrm^+)(u)= u   \,\left(c \cos 2\alpha \frac{ \Trm^{+}(u)}{\Trm(u)}+  \frac 1{4\sin^2\!\al}\right)^2
  -c^2\sin^2 (2\alpha)(u-4q) \left(\frac{\Trm^-(u)}{\Trm(u)}\right)^2,
  \eeq
  which is indeed a series in $u$ and $q$.  We  replace $u$ by $\Srm(u)$ in this identity, inject the expressions of $ \Trm(\Srm(u))$, $\Trm^+(\Srm(u))$  and $\Trm^-(\Srm(u))$ given in Lemma~\ref{lem:shiftT}, and finally  the expression~\eqref{SE} of $\Srm(u)$ in terms $\Erm$. This gives the desired identity.
\end{proof}

We can now explicitly describe a series $F(x)$ satisfying~\eqref{FM-v=1} and~\eqref{F-prop-v=1}.

\begin{Proposition}\label{prop:F-construct}
  There exists a unique series $F(x)$ in $\GK[[x,q]]$ such that
  \beq\label{F-prod}
  F(\Xrm(u))=  \Xrm^+(u)\Xrm(u) \Xrm^-(u).
  \eeq
  This series is in fact a multiple of $q$. The unique series $M(x)$ of $q\GK((x))[[q]]$ that solves
  \beq\label{algFM}
  F(x)=x M(x) (1-\omega x-M(x))
  \eeq
  satisfies
  \beq\label{MXu}
  M(\Xrm(u))= \Xrm^-(u).
  \eeq
  It also satisfies, as desired
  \[
    F(M(x))=x M(x) (1-\omega M(x)-x).
  \]
\end{Proposition}
\begin{proof}
  As argued below~\eqref{Xdef}, the series $\Xrm(u)$ can be seen as a series in $u$, with constant term~$0$ and coefficients in $\qs[\om][[q]]$. Recall that $\Xrm(4q \sin^2\!z)=\chi(z)$. The coefficient of $u$ in $\Xrm(u)$ is thus
  \beq\label{X1}
  \Xrm_1:=   \lim_{u\rightarrow 0} \frac{\Xrm(u)}u =   \lim_{z\rightarrow 0} \frac{\Xrm(4q\sin^2\!z)}{4q\sin^2\!z}=   \lim_{z\rightarrow 0} \frac {\chi   (z)}{4q\sin^2\! z}=1+ \LandauO(q),
  \eeq
  by the last result of Lemma~\ref{lem:chi}.
  Hence    $\Xrm(u)$ has a compositional inverse $U(x)$. By the Lagrange inversion formula, its coefficients are polynomials in the coefficients $\Xrm_i$ of $\Xrm(u)$, divided by powers of $\Xrm_1$: hence they are \fps\ in $q$. Moreover, 
  \[
    % \Xrm^{<-1>}
    U(x)= \frac{x}{\Xrm_1}  + \LandauO(x^2).
  \]

  Let us now consider the right-hand side of~\eqref{F-prod}, namely $ \Xrm^+(u)\Xrm(u) \Xrm^-(u)$. When we expand the product  $\Xrm^+(u) \Xrm^-(u)$, the square roots disappear, and, as already observed in the proof of Lemma~\ref{lem:shiftX},  $u\Xrm^+(u) \Xrm^-(u)$ is a \fps\ in $u$ and $q$. Moreover, $\Xrm(u)$ is a \fps\ in $u$ and $q$, and a multiple of $u$ as already discussed. Hence the right-hand side of~\eqref{F-prod} is a \fps\ in $q$ and $u$, and the substitution $u\mapsto  U(x)$ in~\eqref{F-prod} defines $F(x)=(\Xrm^-\Xrm\Xrm^+)(U(x))$ as a \fps\ in $x$ and $q$. The fact that~$\Xrm^-(u)$ is a multiple of $q$ (see~\eqref{Xpm-exp}) implies that the same holds for $F(x)$.

  Let us now define $M(x)$ in $q\GK((x))[[q]]$ by~\eqref{algFM}, or equivalently~\eqref{MF}. The coefficient of $q^n$ in this series is a Laurent series in $x$ with valuation at least $-n$. In other words, $M(x)$ lies in $q/x\GK[[x,q/x]]$. Now  $\Xrm(u)= u/(1+u)+u\LandauO(q)$ belongs to $u\GK[[u,q]]$, and $q/\Xrm(u)$ can be seen to be in $q/u\GK[[u,q]]$. Hence  the substitution $M(\Xrm(u))$ is well-defined as a series in $u$ and $q/u$, and by definition of $M(x)$,
  \begin{align*}
    F(\Xrm(u))&= \Xrm(u) M(\Xrm(u)) (1- \omega \Xrm(u) - M(\Xrm(u)))\\
              & = \Xrm(u) \Xrm^+(u) \Xrm^-(u)  &&\text{by definition of } F,\\
              &= \Xrm(u) \Xrm^-(u) (1-\omega \Xrm(u) - \Xrm^-(u)) &&\text{by~\eqref{eqfunc-X}.}
  \end{align*}
  Comparing the first and third lines shows that $M(\Xrm(u))$ is either $\Xrm^-(u)$ or $1-\omega \Xrm(u) - \Xrm^-(u)= \Xrm^+(u)$. Since $M(\Xrm(u))$ is a multiple of $q$, while $\Xrm^+(u)$ is not, we conclude that
  $    M(\Xrm(u))= \Xrm^-(u)$, as stated in~\eqref{MXu}.

  Let us now prove the final statement of the proposition. It suffices to prove it for $x$ of the form $\Xrm(u)$. We have
  \begin{align*}
    F(M(\Xrm(u)))&=  F(\Xrm^-(u))  &&\text{by~\eqref{MXu}},\\
                & = F(\Xrm(\Srm(u))) &&\text{by Lemma~\ref{lem:shiftX}}, \\
                &=\Xrm(\Srm(u))  \left(\Xrm^- \Xrm^+\right)(\Srm    (u))  &&\text{by definition of } F,\\
                &=\Xrm^-(u) \Xrm(u) (1-\Xrm(u)-\omega \Xrm^-(u)) && \text{by Lemma~\ref{lem:shiftX} again},\\
                &=\Xrm(u)M(\Xrm(u)) (1-\Xrm(u)-\omega M(\Xrm(u))), &&\text{by~\eqref{MXu} again},
  \end{align*}
  as desired.
\end{proof}

% ======================================================
\subsection{Determination of $\bm{q\equiv q(t)}$. End of the proof of Theorem~\ref{thm:allomega}}
\label{sec:qt}
% ======================================================
We are nearly there! Now let $q$ be  \emm any, \fps\ in $t$ with no constant term, having coefficients in $\GK=\qs(\omega)$. Then the series $F(x)$ becomes a series in $t\GK[[x,t]]$ that satisfies all conditions of Proposition~\ref{prop:F} (for $v=1$), except for the last one: we still need $[x^{-1}] M(x)=t$. This will force the choice of the series $q\equiv q(t)$.

\begin{Lemma}\label{ct-H}
  Let $F(x)$ and $M(x)$ be the series in $q$ defined in Proposition~\ref{prop:F-construct}.   Then 
  \[
    [x^{-1}] M(x)= \frac{\cos \al}{64 \sin^3œ\!\al} \left(\frac{\theta''(\al)}{\theta'(\al)} - \frac{\theta(\al)\theta^{(3)}(\al)}{\theta'(\al)^2}\right),
  \]
  where $\theta(z)$ is defined by~\eqref{theta-def}.
\end{Lemma}
The proof is given in Appendix~\ref{app:xm1}.
Now the right-hand side of the above expression is a series in $q$ with coefficients in $\qs[\om]$ and no constant term, starting with $q+\LandauO(q^2)$. By inversion, there exists a unique \fps\ in $t$, say $q\equiv q(t)$, having coefficients in $\qs[\om]$ and no constant term,  such that the above expression equals $t$. This series starts as follows:
\[
  q=t+\left(6+6 \omega \right) t^{2} +\left(45 \omega^{2}+84 \omega +48\right) t^{3}+\left(378 \omega^{3}+998 \omega^{2}+1076 \omega +436\right) t^{4}+\LandauO(t^5).
\]
The series $F(x)$ and $M (x)$ obtained for this value of $q$ satisfy all conditions of Proposition~\ref{prop:F} (for $v=1$), and hence must coincide with $\Fnn(x)$ and $\Mnn(x)$, respectively.

The next step in the proof of Theorem~\ref{thm:allomega} is to express the \gf\ $\Qgf$ in terms of this series $q(t)$. We  start from the alternative expression~\eqref{Q-alt}, which, since $v=1$, specialises to:
\[
  \Qgf=\frac{t-\Fnn(0)}{(\omega+2)t^2}-1.
\]
The announced expression of $\Qgf$ (with $\widetilde\Rgf=\Fnn(0)$) is then a direct consequence of the following lemma.
\begin{Lemma}\label{lem:v1_final}
  We have
  \[
    \Fnn(0)=\frac{\cos^{2}\!\alpha}{96\sin^{4}\!\alpha}\cdot \frac{\theta(\alpha)^{2}}{\theta'(\alpha)^{2}}
    \left(\frac{\theta^{(3)}(0)}{\theta'(0)}-\frac{\theta^{(3)}(\alpha)}{\theta'(\alpha)}\right).
  \]
\end{Lemma}
\begin{proof}
  We will specialise the definition~\eqref{F-prod} of $F\equiv \Fnn$ at $u=0$.  We have seen that $\Xrm(u)$ is a series of $u\qs[\om][[u,q]]$. Moreover, it follows from~\eqref{X1} and~\eqref{order1a} that the coefficient of $u^1$ in $\Xrm(u)$  is
    %the series~\eqref{order1a} divided by $4q$ (see~\eqref{X1}):
    \[
      \Xrm_{1}:=[u^{1}]\Xrm(u)=\frac{1}{96q\sin ^2\! \alpha} \left( \frac{\theta^{(3)}(0)}{\theta'(0)}-\frac{\theta^{(3)}(\alpha)}{\theta'(\alpha)}\right).
    \]
         On the other hand, it follows from~\eqref{Xpm} that the product $\Xrm^+(u) \Xrm^-(u)$ has a simple pole at $u=0$, with residue
  \beq\label{residue}
  4q c^2\sin^2 (2\alpha) \left( \frac{\Trm^-(0)}{\Trm(0)}\right)^2,
  \eeq
  where $c$ is given by~\eqref{c-def}. Since $\Trm(0)=\theta'(0)$ and $\Trm^{-}(0)= -2 \theta(\al)$, this reads
    \[
      q\,\frac{\cos^2\!\al}{\sin^2\!\al}\cdot \frac{\theta(\al)^2}{\theta'(\al)^2}.
    \]
    We now write
    \[
      \Fnn(0) = \left. \frac{\Xrm(u)}{u} \cdot \big(u \Xrm^+(u) \Xrm^-(u)\big)\right|_{u=0 } = \Xrm_1 q\,\frac{\cos^2\!\al}{\sin^2\!\al}\cdot \frac{\theta(\al)^2}{\theta'(\al)^2} ,
    \]
  and the result follows.
\end{proof}

To conclude the proof of Theorem~\ref{thm:allomega}, we need to prove that $\Qgf$ is D-algebraic, which is done in Section~\ref{sec:DA}, and that $t$ has a D-finite expression in $\widetilde\Rgf$ in some cases, which is done in Section~\ref{sec:Tutte}.

\begin{Remark}
  \noindent
     Observe that the definition of $\widetilde\Rgf$ as $\Fnn(0)$ in this section is consistent with the cases $\om=0$ and $\om=1$ solved in Section~\ref{sec:01}, where we had $\Fnn(y)=\Rgf_0/(1-y)$ in the former case, and $\Fnn(y)=\Rgf_1$ in the latter.
\end{Remark}

%===================================================
\subsection{Predicted singular   behaviour}
% =====================================================

      In this section we discuss the dominant singularity and singular behaviour
      % asymptotic form
      of $\Qgf$ (still taken at $v=1$)
      % $(t, \om,1)\equiv \widetilde \Qgf(t,\om)$
      as a function of $t$, for a fixed
       value of $\omega\in\mathbb{R}$.
      Clearly, a key point is the implicit definition of $q$ by~\eqref{t-q}. We denote by $t(q)$ the right-hand side of this expression. This is a well defined meromorphic function of $q$ in the open unit disk, satisfying $t(0)=0$ and having poles at points~$q$ where $\theta'(\alpha, q)=0$ (note that we use here the complete notation $\theta(\alpha,q)$ rather than $\theta(\alpha)$ as above, but the prime always indicates a derivative in the first variable). We will consider two candidates for the dominant singularity of $\Qgf$ and determine the singular behaviour of $\Qgf$ %$(t,\om)$
      in each case.

      The first natural candidate       we consider is
            $t_1:=t(1)$. We analyse this possibility using the same method as in~\cite[Sec.~4]{kostov}; see also~\cite[Sec.~3]{elvey-ness-raschel} where this method is applied to walks in the quarter plane. The first step is to use the Jacobi identity: 
      \[
        \theta(z,q)= i\, \frac{\hat{q}^{\frac{1}{8}}}{q^{\frac{1}{8}}}\sqrt{\frac{-\log(\hat{q})}{2\pi}}\exp\left(\frac{\log(\hat{q})z^{2}}{2\pi^{2}}\right)\theta\left(\frac{i}{2\pi}\log(\hat{q})z,\hat{q}\right),
      \]
where $q$ and $\hat{q}$ are related by $\log(q)\log(\hat{q})=4\pi^{2}$, so $\hat{q}\to 0$ as $q\to 1$. Substituting this into the expressions of~$t$ and $\Rgf$ given in  Theorem~\ref{thm:allomega} yields the following, where $s=\frac{i   \log\left(\hat{q}\right)\alpha}{2  \pi}$:
  \begin{align*}
    \frac{64\sin^{3}\! \alpha}{\cos \alpha}\,t&=-\log( \hat q )  \frac{4\alpha\theta'(s, \hat q)^{2}-
\theta''(s, \hat q)\left(
4\alpha\theta(s, \hat q) 
- \pi i\theta'(s, \hat q)\right)-\pi i\theta^{(3)}(s, \hat q)\theta(s, \hat q)
                                                }{2 \, {\left(2  \alpha \theta(s, \hat q)+ \pi i \theta'(s, \hat q)\right)^{2}}}
    \\
&~~~-\frac{4\theta(s, \hat q)
         }{ \, {2  \alpha \theta(s, \hat q)+ \pi i \theta'(s, \hat q)}},
    \\
\frac{96\sin^{4}\! \alpha}{\cos^{2}\! \alpha}\,\Rgf&=-\frac{{\left(
8 \, \alpha^{3} \theta(s, \hat q)
+ 12 \pi i \alpha^{2} \theta'(s, \hat q)
- 6 \, \pi^{2} \alpha  \theta''(s, \hat q)
- \pi^{3} i \theta^{(3)}(s, \hat q)\right)} \theta(s, \hat q)^{2}
}{{\left(2  \alpha \theta(s, \hat q)
+ \pi i \theta'(s, \hat q)\right)^{3}}}\\
                                                &~~~ -\frac{\theta^{(3)}(0, \hat q)}{\theta'(0, \hat q)}\cdot \frac{\pi^{2}\theta(s, \hat q)^{2}}{{\left(2  \alpha \theta(s, \hat q)+ \pi i \theta'(s, \hat q)\right)^{2}} }.
 \end{align*}
We now assume $\om\in(-2,2)$, so that $\alpha\in (0, \pi/2)$. Then $x:=e^{-is}=\hat{q}^{\frac{\alpha}{2\pi}}$ tends to $0$ as $q$ tends to~$1$, and moreover $y:=\hat q e^{4is}$ tends to $0$ as well.  Using \eqref{theta-def-init}, we can write
   \[
     \theta(s, \hat q)= -\frac i {2x} +\frac i 2 (y+1)x+ x^2 \beta(x,y),
   \]
      where $\beta(x,y) \in \qs[x][[y]]$.
   Replacing $\theta(s, \hat q)$ by this expression in the above expression of $\Rgf$ gives~$\Rgf$ as a \fps\ in $x$ and $y$ (see our {\sc Maple} session~\cite{bmep-ref-arxiv} for details),
 %and more precisely,}
 % these expand as series in $\hat{q}$ and $e^{-is}=\hat{q}^{\frac{\alpha}{2\pi}}$
 with the following initial terms around $\hat q=0$:

  \beq\label{R-exp}
    \Rgf=\frac{\cos^{2}\! \alpha}{24\sin^{4}\! \alpha} \frac{\alpha}{(\pi-2\alpha)^{2}} \left(\pi-\alpha-\frac {6\pi^2}{\pi-2\alpha} \hat{q}^{\frac{\alpha}{\pi}}\right)+o(\hat{q}^{\frac{\alpha}{\pi}}).
  \eeq
%  \[
%    \Rgf=\frac{\cos^{2}\! \alpha}{96\sin^{4}\! \alpha}\left(\frac{4\alpha(\pi-\alpha)}{(\pi-2\alpha)^{2}}-\frac{24\pi^{2}\alpha}{(\pi-2\alpha)^{3}}\hat{q}^{\frac{\alpha}{\pi}}\right)+o(\hat{q}^{\frac{\alpha}{\pi}})\]}
Analogously, one obtains
  \beq\label{t-exp}
    t=\frac{\cos \alpha}{16\sin^{3}\! \alpha}\frac{1}{\pi-2\alpha}
    \left(1+  \frac {2\hat{q}^{\frac{\alpha}{\pi}}} {\pi-2\alpha} (\alpha \log (\hat q) -\pi)\right) +o(\hat{q}^{\frac{\alpha}{\pi}}).
  \eeq
%{\[t=\frac{\cos \alpha}{16\sin^{3}\! \alpha}\frac{1}{\pi-2\alpha}\left(1+\frac{2\alpha \log(\hat{q})-2\pi}{\pi-2\alpha}\hat{q}^{\frac{\alpha}{\pi}}+o(\hat{q}^{\frac{\alpha}{\pi}})\right)\]}
Setting $\hat q=0$ in the latter expression yields the value of the singularity $t_1$:
\beq\label{t1-val}
  t_1=\frac{\cos \alpha}{16\sin^{3}\! \alpha}\frac{1}{\pi-2\alpha}=\frac{1}{4\arccos(\om/2)}\frac{\sqrt{2-\om}}{(2+\om)^{3/2}}.
\eeq
From~\eqref{t-exp} we also derive
  \[
    \hat q^{\frac{\alpha}{\pi}}\sim- \frac{\pi-2\alpha}{2\pi} \frac{1-t/t_1}{\log(1-t/t_1)}
  \]
  (see~\cite[Sec.~7]{mbm-courtiel} for a rigorous singular study of an equation of the same type as~\eqref{t-exp}). This identity allows us to rewrite the expansion~\eqref{R-exp} of $\Rgf$ in terms of $1-t/t_1$ rather than $\hat q$. Combining this series for $\Rgf$ with  the expression of $\Qgf$ given in Theorem~\ref{thm:allomega}, we obtain
% , then comparing the two series yields
the singular behaviour of $\Qgf$ predicted by Kostov~\cite{kostov} and Zinn-Justin~\cite{zinn-justin-6V-random}: 
\begin{equation}\label{eq:log-asymptotics}
  \Qgf -P\sim -{8\pi\alpha}\frac{1-t/t_1}{\log(1-t/t_1)},
\end{equation} 
for some polynomial $P$ in $t$. This behaviour has been proven to be correct
in~\cite{mbm-aep1} for $\om=0$ and $\om=1$.

The other natural candidate  we consider is
a point $t_0$ corresponding to some $q_{0}$, of modulus less than $1$, satisfying $t'(q_{0})=0$. There the implicit function theorem does not apply to define~$q$ in terms of $t$. In this case we use the following equations, which were proven in~\cite{elvey-zinn} (the first is~(22) while the second combines~(22) and~(23) and is equivalent to the first equation in the proof of Proposition 5.2 in~\cite{elvey-zinn}):
\beq \label{t-dR}
% \begin{align*}
\begin{cases}
t&=\displaystyle -\frac{\cos\alpha}{8\sin^{3}\!\alpha}q\frac{d}{dq}\frac{\theta(\alpha,q)}{\theta'(\alpha,q)},\\
\displaystyle  \frac{d \Rgf}{dt}&\displaystyle =\frac{\cos\alpha}{\sin\alpha} \cdot \frac{\theta(\alpha,q)}{\theta'(\alpha,q)}.
\end{cases}
  \eeq
  Hence $t$ is, up to an explicit factor,~$q$ times the $q$-derivative of $d\Rgf/dt$. We can then write $t$ and $d \Rgf/dt$ (which are both meromorphic in~$q$ near $q_0$) as series around $q=q_{0}$ as follows: 
\begin{align*}
t&=t_0+c_{2}(q_{0}-q)^2+\LandauO\!\left((q_{0}-q)^3\right),\\
  \frac{d \Rgf}{dt}
  % \Rgf'(t(q))
  % =\frac{\cos\alpha}{\sin\alpha}\cdot \frac{\theta(\alpha,q)}{\theta'(\alpha,q)}
  &=b_{0}+b_{1}(q_{0}-q)+\LandauO\!\left((q_{0}-q)^2\right),
\end{align*}
for constants $c_{2},b_{0},b_{1}$. It then follows from~\eqref{t-dR} that $b_{1}=2t_0(\omega+2)/q_{0}$.  Combining these yields
  \[
    \frac{d \Rgf}{dt}=b_{0}+\frac{2t_0(\omega+2)}{q_{0}}\sqrt{\frac{t_0}{-c_{2}}}(1-t/t_0)^{1/2}+\LandauO\!\left(1-t/t_0\right),
  \]
which allows us to write $\Qgf=-1+\frac{t-\Rgf}{(\omega+2)t^{2}}$ as a the following series for some polynomial $P$ in $t$: 
%\[\Qgf(t)=P(t)-\frac{2b_{1}}{3(\omega+2)t_0^{2}c_{2}^{1/2}}(t-t_0)^{3/2},\]
\begin{equation}\label{eq:map-like-asymptotics}
  \Qgf=P+\frac{4}{3q_{0}}\sqrt{-\frac{t_0}{c_{2}}}(1-t/t_0)^{3/2}+\LandauO\!\left((1-t/t_0)^{2}\right),
\end{equation}
which has a map-like singularity.

It remains to determine which (if any) of these two cases holds for each value of $\omega$.

\begin{Prediction}
  There exists a value $\om_c\approx -0.764$ such that the dominant singularity of the series $\Qgf$ is
  \begin{itemize}
  \item of the logarithmic type~\eqref{eq:log-asymptotics} for $\om\in (\om_c,2)$, with $t_1$ given by~\eqref{t1-val},
    \item of the map type~\eqref{eq:map-like-asymptotics} for $\om<\om_c$ (with $t_0<0$) and for $\om>2$ (with $t_0>0$).
    \end{itemize}
    In terms of the coefficients  of $\Qgf$, these behaviours correspond respectively to 
     \beq\label{qn-asympt}
  q_n:=  [t^n] \Qgf \sim \frac \kappa {t_1^n n^2 (\log n)^2}  \qquad \text{and} \qquad  q_n\sim \frac \kappa {t_0^n n^{5/2}}.
\eeq
\end{Prediction}

\begin{proof}[Some justification]
  As discussed above, the main question is the behaviour of $t(q)$ as  a function of $q$, the location of its poles, and a comparison of the values $t_1=t(1)$ (which we have determined in~\eqref{t1-val} for $\om\in(-2,2)$) and $t_0=t(q_0)$ such that $t'(q_0)=0$ (if such a value exists).
    %     We discuss our (non-rigourous) observations below.
    Our predictions mainly rely on an estimate  of $\theta(\alpha,q)$ obtained from the first $16$ terms of the expansion~\eqref{theta-def-init}, which gives a rational approximation of $t(q)$. We complete this with the explicit value of the first~40 coefficients of the series $\Qgf$.

Let us begin with non-negative values of $\om$. By Pringsheim's theorem, we know that the radius of convergence of $\Qgf$ is one of the dominant singularities, so we only explore positive values of~$t$.  Looking for poles of $t(q)$,  we find  that there is a root $q_{c}\in(-1,0)$ of $\theta'(\alpha,q)$ such that the function $q\mapsto t(q)$ increases for $q\in(q_{c},0]$, with $t(q_{c})=-\infty$ and $t(0)=0$. Then the behaviour for positive values of $q$ depends on whether $\om<2$ or $\om>2$ (Figure~\ref{fig:tqplots-positive}). If $\om<2$, then $t(q)$ also increases from $t(0)=0$ to $t(1)=t_1$. Hence the dominant singularity corresponds to $q=1$, and we have the behaviour  described in~\eqref{eq:log-asymptotics}, predicted by Kostov and Zinn-Justin and proved rigorously in~\cite{mbm-aep1} for $\om=0$ and $\om=1$.
  If $\omega>2$,
  % we have $\alpha\in \frac{\pi}{2}+i\mathbb{R}$ and
  we observe that there is some $q_{0}\in(0,1)$ for which $t'(q_0)=0$ and $t(q)$ is increasing for $q\in[0,q_{0}]$. Then the dominant singularity is $t_0=t(q_0)$ and this yields the map-like singular behaviour~\eqref{eq:map-like-asymptotics}.

\begin{figure}[h!]
  \centering
  \includegraphics[width=55mm]{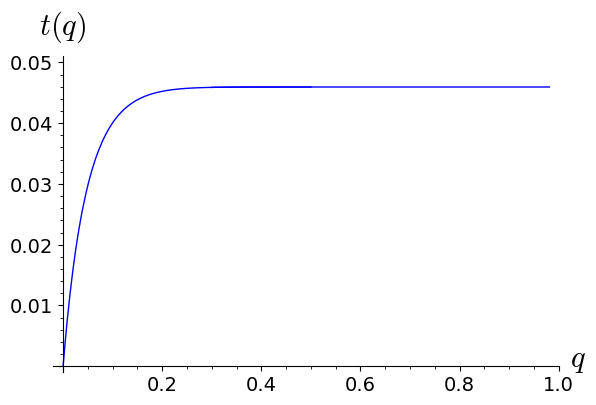}
  \includegraphics[width=55mm]{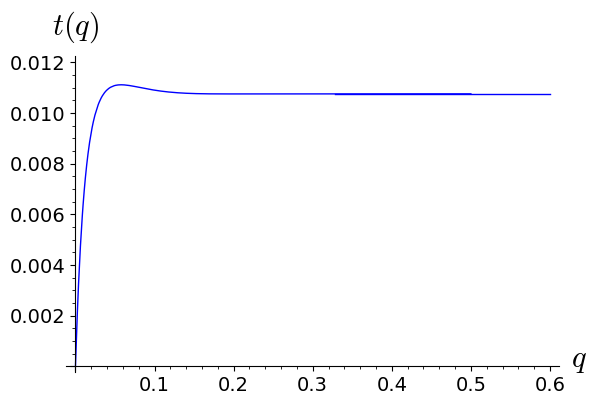}
  \caption{Plots of our approximation of the function $t(q)$ for $\om \in [0,2)$ and $\om >2$ (left $\om=1$, right $\om=7$).}
  \label{fig:tqplots-positive}
\end{figure}

The negative values of $\om$ are non-physical, but we still discuss them because the whole interval $[-2,2]$ was considered in~\cite{kostov,zinn-justin-6V-random}. The first difficulty is that the coefficients of the series $\Qgf$ may have negative signs. However, looking at the $40$ coefficients that we have determined suggests that a transition occurs somewhere between $-1$ and $-1/2$: below $-1$, the signs seem to be alternating, so that a dominant singularity should be real and negative, while above $-1/2$, the signs seem to be positive, so that one of the dominant singularities should be real and positive. This will be explained by the competition between a positive and a negative singularity which both exist on $[-2,0)$.  For $\om<-2$ first, we find a root $q_0<0$ of $t'(q)$  and a pole $q_c>0$ of $t(q)$  such that $t(q)$ increases from $t_0=t(q_0)<0$ to $t(q_c)=+\infty$ on the interval $[q_0,q_c]$ (Figure~\ref{fig:tqplots-negative}, left). We thus expect the map-like singular behaviour~\eqref{eq:map-like-asymptotics}. For $\om\in (-2,0)$, there is no pole $q_c>0$ but the root of $t'(q)$ at some $q_0<0$ remains. The map $t(q)$ increases from $t_0=t(q_0)<0$ to $t_1=t(1)$. The question is thus whether $|t_0|<t_1$ or not. We observe that there is some $\om_c\approx -0.764$ such that $|t_0|<t_1$  for $\om\in (-2,\om_c)$, yielding a singularity  of the map type~\eqref{eq:map-like-asymptotics}, while $|t_0|>t_1$  for  $\om\in (\om_c, 0)$, yielding the logarithmic singularity~\eqref{eq:log-asymptotics} as we already found for $\om\in(0,2)$ (see the rightmost two plots in Figure~\ref{fig:tqplots-negative}). It appears that Kostov and Zinn-Justin didn't consider the possibility of a transition in $[-2,0]$, as they predicted  a uniform log-type behaviour in this interval.

\begin{figure}[h!]
  \centering
 \includegraphics[width=45mm]{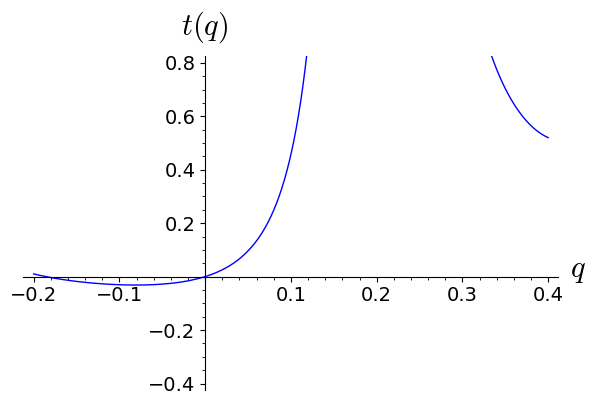}
 \includegraphics[width=45mm]{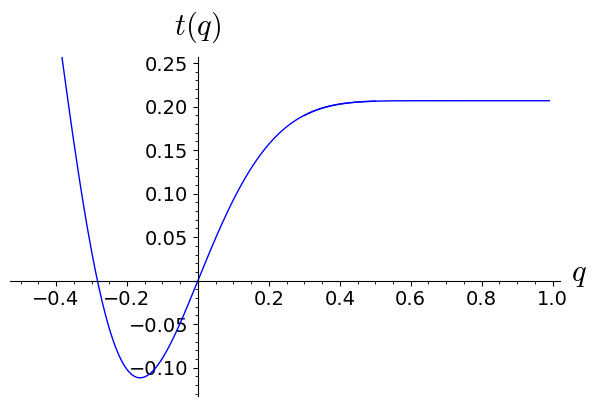}
  \includegraphics[width=45mm]{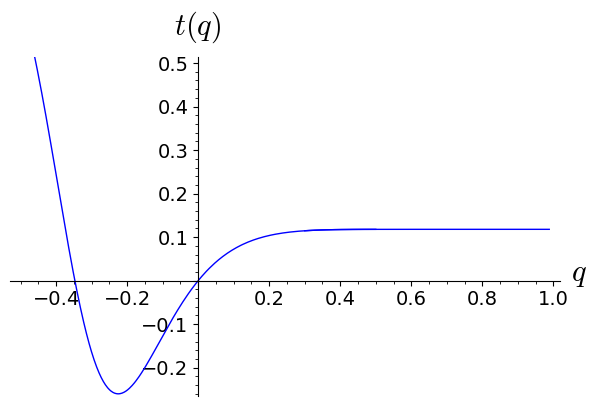}
  \caption{Plots of our approximation of the function $t(q)$ for $\om<-2 $, $\om \in (-2, \om_c)$ and $\om\in (\om_c,0)$ (left $\om=-3$, center $\om=-1$, right $\om=-1/2$).}
  \label{fig:tqplots-negative}
\end{figure}

We can confirm our predictions using the $40$ known coefficients of $\Qgf$. On the left of Figure~\ref{fig:ratio} we show the ratio of two consecutive coefficients, as a function of $\om$, avoiding the critical zone around~$\om_c$. The convergence to the value $t_1$ (red curve) is clear for $\om \in[-1/2,2]$. It even seems that this could persist beyond $\om=2$, but we do not believe this to be the case.
% We note that the function $t_1$ in~\eqref{eq:log-asymptotics} has an analytic continuation {on} $\cs\setminus(-\infty, -2]$, and {applying the} ratio test {to} the coefficients of $\Qgf$ suggests that this value may be the radius {of convergence} of $\Qgf$ for all positive values of $\om$; see Figure~\ref{fig:ratio}. {(Could it be that $t_0=t_1$?)} {(non, je crois pas, pour $\omega=5$ par example je trouve $t_0\approx0.015090228$, mais $t_1\approx 0.014922453$.)} Then, in order to distinguish between the behaviours~\eqref{eq:log-asymptotics} and~\eqref{eq:map-like-asymptotics}, which correspond to coefficients of the asymptotic form
On the right-hand side of Figure~\ref{fig:ratio}  we plot estimates of the exponent of $n$ in the asymptotic behaviour~\eqref{qn-asympt} of $q_n$, namely $n^2(1- q_{n+1} q_{n-1}/q_n^2)$. They are in agreement with a double transition, at $\om_c$ and at $2$.

\begin{figure}[htb]
  \centering
  \includegraphics[width=55mm]{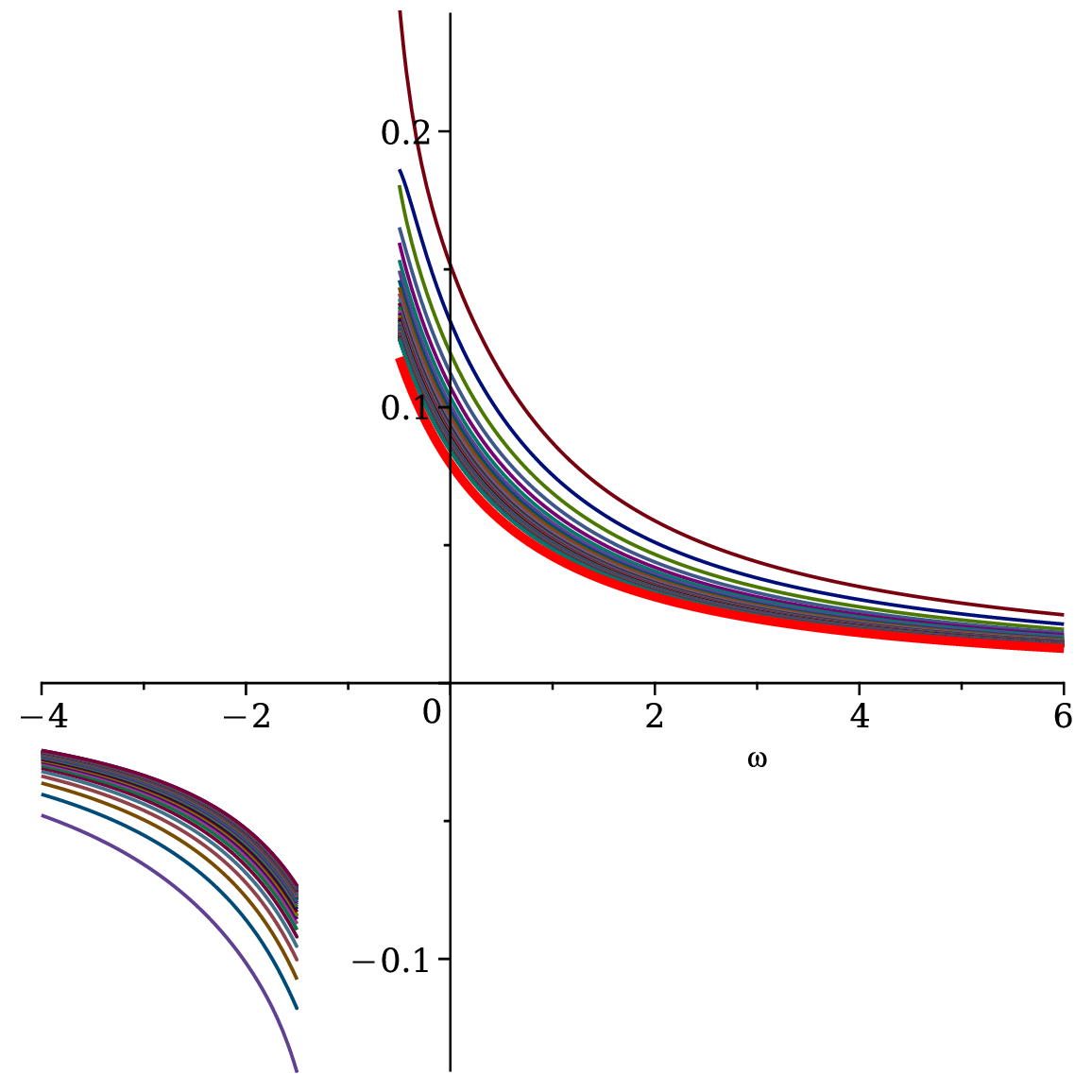} \hskip 10mm
  \includegraphics[width=55mm]{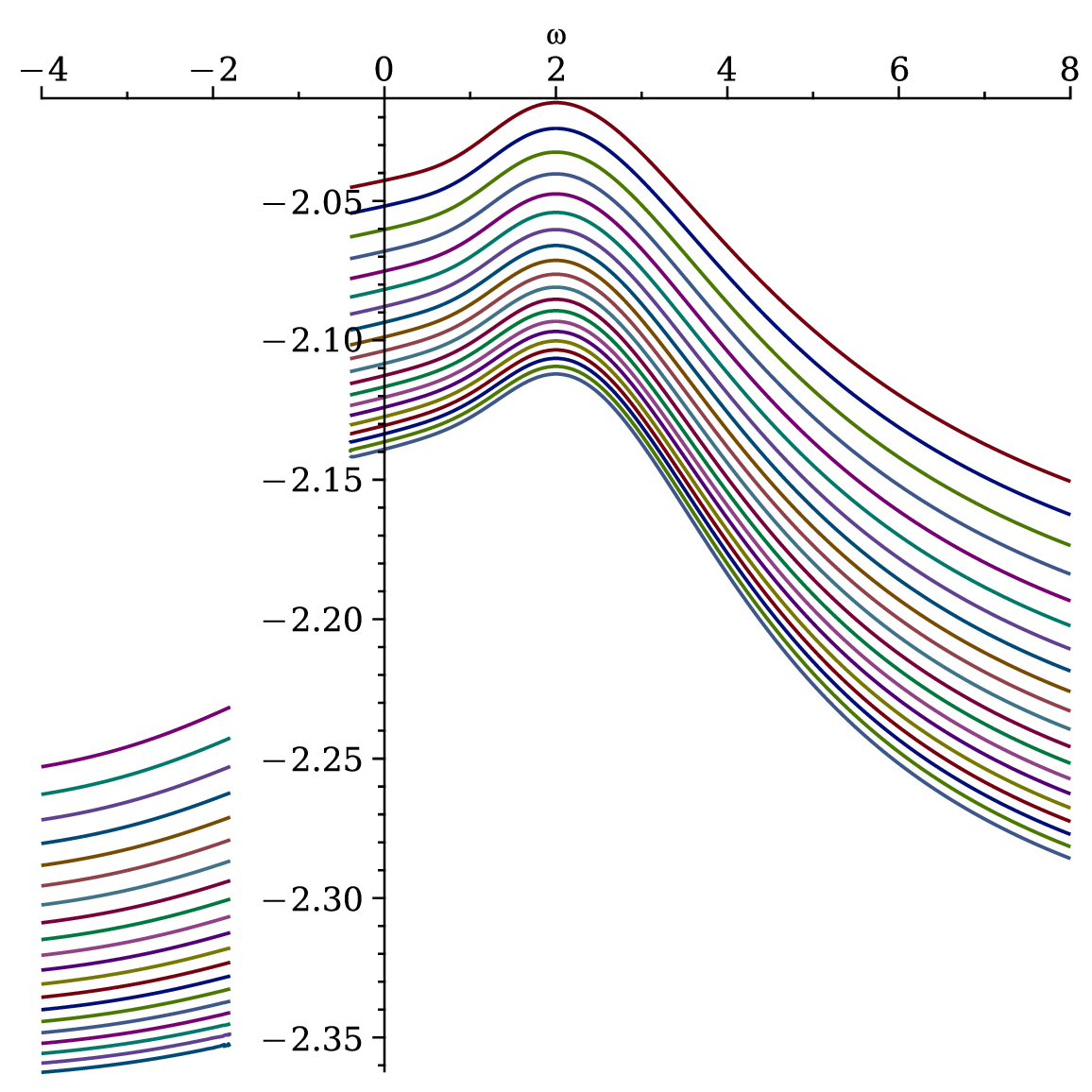}
  \caption{\emm Left:, The ratio of the coefficients of $t^n$ and $t^{n+1}$ in the series $\Qgf$, as functions of $\om$, for $n<40$. The red bottom curve  is the value of $t_1$ given in~\eqref{t1-val}. \emm Right:, Estimates of the critical exponent. Both figures show that a transition occurs at some $\om_c\in (-2,-1)$, and the figure on the right also clearly shows the transition at $\om=2$.}
  %suggest a transition at $\om=2$ between the logarithmic regime~\eqref{eq:log-asymptotics} and the map-like behaviour~\eqref{eq:map-like-asymptotics}.}
  \label{fig:ratio}
\end{figure}

\end{proof}

%%%%%%%%%%%%%%%%%%%%%%%%%%%%%%%%%%%%%%%%%%%%%%%%%%%%%%%%%%%%%%%%%%%%%
\section{Solution for $\bm{v=1}$ and special values of $\bm{\om}$}
\label{sec:Tutte}
%%%%%%%%%%%%%%%%%%%%%%%%%%%%%%%%%%%%%%%%%%%%%%%%%%%%%%%%%%%%%%%%%%%%%
At this stage, we have solved the three cases $\om=0$, $\om=1$ and $v=1$. It is natural to look at the doubly specialised cases: $v=1$ \emm and, $\om\in\{0,1\}$. As already observed in~\cite{elvey-zinn}, say when $v=1$ and $\om=0$, we have in Theorems~\ref{thm:general} and~\ref{thm:allomega} two very different characterisations of the same series $\Rgf:= \Rgf_0= \widetilde \Rgf=\Fnn(0)$. The former one looks simpler: $t$ has an explicit D-finite expression in~$\Rgf$.  Also the series involving the catalytic variable $x$ are much simpler when $\om=0$:  we have seen that $\Fnn(x)$ is simply $\Rgf/(1-x)$ (see Remark~\ref{rem:om0}). Similar statements apply to the case $v=\om=1$, where $\Fnn(x)=\Rgf$ is even independent of $x$; see Remark~\ref{rem:om1}.

In this section we present an alternative way to derive $\Mnn(x)$ and $\Fnn(x)$ from the characterisation of Section~\ref{sec:charac}, which works for the doubly specialised cases where $v=1$ and $\om \neq  -2,2$ is of the form $-2\cos(k\pi/m)$, for integers $k$ and $m$. This approach implies  that, in these cases, the generating function $\Fnn(x)$ is algebraic in~$x$ and $\Rgf:=\Fnn(0)$, with base field $\qs$. By~\eqref{MF}, this means that the same holds for $\Mnn(x)$, as $\om$ is also algebraic over $\qs$ in these cases. Since $t=[x^{-1}] \Mnn(x)$, this implies that $t$ is a D-finite function of $\Rgf$.
  The latter property was already shown in~\cite[Thm.~6.1]{elvey-zinn}, using the explicit solution of Theorem~\ref{thm:allomega} and its modular properties.

\begin{Proposition}\label{prop:alg}
  Take $v=1$ and  $\om=-2\cos(k\pi/m)$, excluding the two extreme values $\om=-2,2$. Then the series $\Fnn(x)$ defined by~\eqref{FM}, and the associated series $\Mnn(x)$, are algebraic over $\qs(x,\Rgf)$, where $\Rgf:=\Fnn(0)$.

  Consequently, $t$  is a D-finite function of $\Rgf$.
\end{Proposition}

  \begin{Remark}
    We find this result to be reminiscent of an old result of Tutte on properly coloured planar triangulations: when the number of colours is of the form $2 + 2 \cos (k\pi/m)$, excluding the extreme values $0$ and $4$, the associated \gf\ is algebraic~\cite{tutte-chromatic-revisited}. The same phenomenon was later proved to hold as well for the more general Potts model, both on triangulations and general planar maps~\cite{bernardi-mbm-alg}. Our proof of Proposition~\ref{prop:alg}  is in spirit rather close to the one of Tutte: it uses  a notion of \emm invariants,, which lead to  a  polynomial equation in one catalytic variable~$x$ satisfied by $\Fnn(x)$. Algebraicity then follows from a general result on such equations~\cite{mbm-jehanne,popescu}. 

    One could alternatively start from our theta-function expressions to show that $\Fnn(x)$ is algebraic in $x$ and $\Rgf$. Along the lines of \cite{elvey-zinn}, we would use properties of elliptic functions to show algebraicity in $x$ and modular properties to show algebraicity in $t$.
  \end{Remark}

\medskip
From now on, let us fix $v=1$.
It follows from the characterisation of $\Mnn(x)$, and more precisely from~\eqref{FM} and~\eqref{kernel-F}, that
\beq\label{A-inv}
A(\Mnn(x))=A(x), \qquad \text{where} \qquad  A(x):= x^{2}-x-\frac{\Fnn(x)}{x}.
\eeq
Note that this holds regardless of the value of $\om$. By analogy with Tutte's work, we will call \emm invariant, any series $I(x)$ in $x^{-\ell}t^{-m}\cs[[x,t]]$, for some $\ell, m \in \ns$, such that $I(\Mnn(x))=I(x)$. Hence our characterisation of $\Mnn(x)$, in the case $v=1$, just tells us that $A(x)$ is an invariant.

For the special value $\om=-2\cos({2\alpha})$ with $\alpha \in \pi\qs$, we will construct a second invariant $B(x)$, and derive from the following \emm invariant lemma,  that $A(x)$ and $B(x)$ are polynomially related.

\begin{Lemma}\label{lem:inv} 
  Let $I(x)$ be an invariant that belongs to $x\cs[[x,t]]$. Then $I(x)=0$.
\end{Lemma}
\begin{proof}
   Assume that $I(x) \neq 0$, and let $N$ be minimal such that there exists $n$ and $m$ such that $m+n=N$ and $[x^{n}t^m]I(x)\neq0$. Let $n$ be maximal for this property. Note that $n\ge 1$.  Recall that $\Mnn(x)$ lies in $t/x\cs[[t/x,x]]$, and reads $t/x (1+\LandauO(t/x) + \LandauO(x))$ (recall that $v=1$). This implies that $I(\Mnn(x))$ has valuation~$N$ in $t$, and that the coefficient of~$t^N$ has valuation~$-n$ in $x$. But, since
     $I(\Mnn(x))=I(x)$, this gives a term with a negative power of $x$ in $I(x)$, a contradiction.
\end{proof}
We now embark on the construction of a second invariant. Let $u=e^{2i\alpha}$,  let $n$ be the smallest positive integer satisfying $u^{n}=1$ and define
\[
  c_{j}:=\frac{u^j-u-1+u^{1-j}}{u-1},
\]
so that in particular $c_0=c_1=0$ and $c_j=c_{1-j}$, and define
 \begin{align*}
    B_{j}(x)&:=(u^j-u^{-j})\Mnn(x)-(u^{j-1}-u^{1-j})x+c_{j},\\
  B(x)&:=\prod_{j=0}^{n-1}B_{j}(x).
\end{align*}
Recall from Proposition~\ref{prop:M-charac} that $\Mnn(\Mnn(x))=x$. It follows  that $B_{j}(\Mnn(x))=B_{1-j}(x)$, and from $u^{n}=1$, we also have $B_{j+n}(x)=B_{j}(x)$. Together these imply that $B(\Mnn(x))=B(x)$, so that $B(x)$ is an invariant.

Moreover, returning to the definition~\eqref{FM} of $\Fnn(x)$, and using $\omega=-u-u^{-1}$ as well as $c_{-j}=c_{j+1}$, we can write
\[
  B_{j}(x)B_{-j}(x)=(u^j-u^{-j})^{2}\,\frac{\Fnn(x)}{x}-\left((u^{j-1}-u^{1-j})x-c_{j}\right)\left((u^{j+1}-u^{-j-1})x+c_{j+1}\right).
\]
This allows us to express the product $B(x)$ in terms of $\Fnn(x)$ as well. If $n$ is even, say  $n=2k$, we have $u^k=-1$ and 
\begin{multline*}
  B(x)=B_0(x) B_k(x) \prod_{j=1}^{k-1} B_j(x) B_{-j}(x)\\
  =-x(1+u)^2((1-u^{-1})^2x+2u^{-1})\prod_{j=1}^{k-1}\left((u^j-u^{-j})^{2}\frac{\Fnn(x)}{x}  \right.
  \\
  \left. -\left((u^{j-1}-u^{1-j})x-c_{j}\right)\left((u^{j+1}-u^{-j-1})x+c_{j+1}\right)\right),
\end{multline*}
while if  $n=2k-1$, we have
\begin{multline*}
  B(x)=B_0(x) \prod_{j=1}^{k-1} B_j(x) B_{-j}(x) \\
  =x(u-u^{-1})\prod_{j=1}^{k-1}\left((u^j-u^{-j})^{2}\frac{\Fnn(x)}{x}-\left((u^{j-1}-u^{1-j})x-c_{j}\right)\left((u^{j+1}-u^{-j-1})x+c_{j+1}\right)\right).
 \end{multline*}
 In both cases $x^{k-2}B(x)$ is a polynomial in $x$ and $\Fnn(x)$
with coefficients in $\qs(u)$, so $B(x)\in x^{2-k}\overline\qs
   % \mathbb{C}
   [[x,t]]$ since $u^n=1$.

 \medskip

 \noindent{\bf Examples.}
For $\om \in\{-2, 2\}$, that is, $u\in\{-1,1\}$, we have $n\in\{1,2\}$, $k=1$,
   $B_0(x)=0$, hence $B(x)=0$ is a trivial invariant.

For $\om=0$, that is, $\al=\pi/4$ and $u=i$, we have $n=4$, $k=2$ and obtain $B(x)=16(1-x) \Fnn(x)$.  The above invariant lemma, applied to $B(x)-B(0)$, tells us at once that $\Fnn(x)=\Fnn(0)/(1-x)=\Rgf/(1-x)$.

For $\om=1$, that is, $\al=\pi/3$ and $u=e^{2{\pi i}/3}$, we  have $n=3$, $k=2$ and obtain $B(x)=-3\sqrt 3 i \Fnn(x)$. The invariant lemma tells us that $\Fnn(x)=\Fnn(0)= \Rgf$.

For $\om=-1$, that is, $\alpha=\pi/6$ and $u=e^{{\pi i}/3}$, we have $n=6$, $k=3$, and obtain $B(x)=27(x-2) \Fnn(x) (\Fnn(x)/x+2-2x)$. We will express $B(x)$ as a polynomial in $A(x)$ (defined in~\eqref{A-inv}), as described below in the general case.
\qee

\medskip

Let us thus return to the general case, and assume that $\om\neq-2,2$, that is, $k\ge 2$. Then $B(x)$ is a polynomial of degree $k-1$ in $\Fnn(x)$, and a Laurent polynomial of valuation $k-2$ in $x$. Since $A(x)$ has valuation $-1$ in $x$, by expanding $B(x)$ around $x=0$ we can find
% \in \frac{r}{x}+\mathbb{C}[[x,r]]$, this implies that
 unique Laurent series $h_{j}\in\mathbb{C}((t))$ satisfying
\[
  I(x):=B(x)+A(x)^{k-2}h_{k-2}+\cdots+A(x)h_{1}+h_{0}\ \ \in\  x t^{-m}\cs[[x,t]]
\]
for some $m\in \ns$. For instance, in the above example where $\om=-1$, we have $k=3$ and 
  \[
    h_1= -54 \Fnn(0) \qquad \text{and} \qquad h_0 =-27\Fnn(0)(\Fnn(0) - 2\Fnn'(0) - 4).
  \]
  In general, given that $A(x) \sim -\Rgf/x$ and $B(x)\sim c\, \Rgf^{k-1}/x^{k-2}$ around $x=0$, for an explicit constant $c\in \cs$, the series $h_{k-2}$ will always be of the form $c' \Rgf$, for some non-zero constant $c'$. More generally, the $h_j$ have rational expressions in $\Fnn(0)\equiv \Rgf$, $\Fnn'(0)$, \ldots, $\Fnn^{(k-2)}(0)$. Also, since the $h_j$'s only depend on $t$, $I(x)$ is an invariant. By Lemma~\ref{lem:inv} (possibly applied to $t^mI(x)$), we have in fact $I(x)=0$, that is,
 \begin{equation}\label{eq:polyinF}
    B(x)+A(x)^{k-2}h_{k-2}+\cdots+A(x)h_{1}+h_{0}=0.
\end{equation}
The expression on the left-hand side is  an explicit polynomial in $\Fnn(x)$, of degree $k-1$ (only $B(x)$ contributes to this dominant term), and a Laurent polynomial in $x$ of valuation $-(k-2)$ (found both in $B(x)$ and in $A(x)^{k-2}$).  This immediately implies that $\Fnn(x)$ is algebraic in $x$, but over the field generated over $\overline\qs$ by the series $h_j$. For instance, as seen above in the examples, we get $\Fnn(x)=\Rgf/(1-x)$ when $\om=0$, and $\Fnn(x)=\Rgf$ when $\om=1$.

It remains to prove that we can replace the field generated by the
$h_j$'s by $\overline\qs(\Rgf)$. This means that $\Fnn(x)$ is algebraic over $\overline \qs(x,\Rgf)$, or equivalently over $\qs(x,\Rgf)$, as claimed by the proposition.

Recall that the series $\Fnn(x)$ belongs to $t\GK[[x,t]]$ (see for instance Proposition~\ref{prop:F}). Also, expanding $\Fnn(x)$ at first order in $t$ gives 
  \[
    \Fnn(x)= \frac{1-\om x}{1-x} t+ \LandauO(t^2).
 \]
In particular, $\Rgf:=\Fnn(0)=t+ \LandauO(t^2)$. Expanding now $\Fnn(x)$ at first order in $x$ gives
  \beq\label{F-start}
    \Fnn(x)= \Rgf+ x\Rgf \widetilde \Fnn(x),
  \eeq
  for some series $\widetilde \Fnn(x)$ in $x$ and $t$, or equivalently in $x$ and $\Rgf$.
  The final step to proving Proposition~\ref{prop:alg} is to show that~\eqref{eq:polyinF} uniquely defines $\widetilde \Fnn(x)$ and the $h_j$ as series $\Rgf$ and $x$ {with coefficients in $\overline\qs$ (where only $\widetilde \Fnn(x)$ depends on $x$ of course).
  Then Theorem~1.4 in~\cite{popescu} implies that $\widetilde{\Fnn}(x)$ is algebraic over $\overline\qs(x,\Rgf)$, hence so is $\Fnn(x)$\footnote{Let us specialise Theorem~1.4 in~\cite{popescu} to the field $k=\overline \qs$,  two variables $X_1=x$ and $X_2=\Rgf$,  a single polynomial equation in unknown series $y$, $y_0, \ldots, y_{k-2}$ with coefficients in $\overline \qs[x,\Rgf]$. Assume that there exists a solution to this equation -- in our case, this is $(\widetilde \Fnn(x), h_0, \ldots, h_{k-2})$ -- such that $\widetilde \Fnn(x) \in \overline \qs\llbracket x, \Rgf\rrbracket$ and $h_j \in \overline \qs\llbracket \Rgf\rrbracket$ for all~$j$. Then Popescu's theorem tells us that there also exists a solution $(y, y_0, \ldots, y_{k-2}) \in \overline \qs\llbracket x, \Rgf\rrbracket\times \left(\overline \qs\llbracket \Rgf\rrbracket\right)^{k-1}$ that is \emm algebraic, over $\overline \qs(x,\Rgf)$. But since  $(\widetilde \Fnn(x), h_0, \ldots, h_{k-2})$ is the \emm unique, solution in this ring, then it must coincide with the algebraic solution.}. We show the uniqueness of $\widetilde \Fnn(x)$ and the $h_j$ below.

    \begin{Lemma} The series $\widetilde \Fnn(x)\in\mathbb{K}[[x,\Rgf]]$ and $h_j\in\mathbb{K}((\Rgf))$ are uniquely determined by~\eqref{F-start} and~\eqref{eq:polyinF}, and in fact $h_{j}\in\Rgf\mathbb{K}[[\Rgf]]\subset \Rgf \overline\qs[[\Rgf]]$.
    \end{Lemma}
    
\begin{proof}
  We have argued above that the $h_j$ are Laurent series in $t$. Hence they are Laurent series in $\Rgf$. Let us  show that they lie in $\Rgf\mathbb{K}[[\Rgf]]$.  We will show  by induction on $n$ that $[\Rgf^{n}]h_{j}=0$ for each $j$, starting at some negative $n$ and ending at $n=0$. Clearly for sufficiently small $n$ the result holds, so we can proceed to the inductive step: assume that $[\Rgf^{m}]h_{i}=0$ for all $i$ and all $m<n$, for some fixed $n\le 0$. That is, $h_i=\LandauO(\Rgf^n)$ for all~$i$. Observe that it follows from~\eqref{F-start} and from the definition~\eqref{A-inv} of $A(x)$  that for all $i,j\ge 0$,
     \[
      [x^{-j} ] A(x)^i =
      \begin{cases}
        0  & \text{if } i<j,\\
        (-1)^j \Rgf^j & \text{if } i=j,\\
        \LandauO(\Rgf^{j+1}) & \text{if } i>j.
      \end{cases}
    \]
Analogously, it follows from~\eqref{F-start} and the expression of $B(x)$ that for $0\le j \le k-2$, the coefficient of $x^{-j}$ in $B(x)$ is a multiple of $\Rgf^{j+1}$. Let us now fix $j \in \llbracket 0, k-2\rrbracket$.
   Extracting the coefficient of $\Rgf^{n+j}x^{-j}$ on both sides of \eqref{eq:polyinF} yields $[\Rgf^{n}]h_{j}=0$, as required. This completes the induction, so $h_{j}\in\Rgf\mathbb{K}[[\Rgf]]$.

We can now determine $[\Rgf^{n}]h_{j}$ and $[\Rgf^{n-1}]\widetilde F(x)$ by induction on $n\ge 0$. For $n=0$ we have just seen that these coefficients are $0$. For the inductive step, assume that $n\geq 1$ and that $[\Rgf^{m}]h_{j}$ and $[\Rgf^{m-1}]\widetilde F(x)$ are known for all $m<n$. Then each $[\Rgf^{n}]h_{j}$, for $0\le j \le k-2$, can be determined by extracting
    the coefficient of $\Rgf^{n+j}x^{-j}$ from~\eqref{eq:polyinF}.
    Finally, $[\Rgf^{n-1}]\widetilde F(x)$ (or equivalently $[\Rgf^n] \Fnn(x)$) can be determined by analysing the coefficient of $\Rgf^{n}$ on both sides of \eqref{eq:polyinF}: this coefficient is known in each term $A(x)^i h_i$, and in $B(x)$, the only unknown part is $[\Rgf^n] \Fnn(x)$ multiplied by the coefficient of $\Fnn(x)^1$ in $B(x)$ (seen as a polynomial in $x$ and $\Fnn(x)$, and in fact a multiple of $\Fnn(x)$), which we can check to be non-zero as soon as $\om \neq -2,2$. This completes our second induction.
\end{proof}

An explicit polynomial relating $\Fnn(x)$, $x$ and $\Rgf$ can be worked out explicitly for fixed values of $\om$ using the effective strategy of~\cite{mbm-jehanne}.

\medskip

  \noindent{\bf Example.} For $\om=-1$, we have obtained
  \[
    B(x)-54 \Rgf A(x)-27\Rgf (\Rgf - 2\Fnn'(0) - 4)=0.
  \]
  Using the approach of~\cite{mbm-jehanne}, or even Brown's original \emm quadratic method,~\cite[Sec.~2.9]{goulden-jackson}, we first find\footnote{Surprisingly, this series in $\Rgf$ is closely related to the \gf\ of bipartite maps.}
  % mbm There is a +12 in one series and a -12 in the other...
  %
  \[
    \Fnn'(0)= \frac{ (1+8\Rgf)^{3/2}-1-12\Rgf+8\Rgf^2}{16\Rgf},
  \]
  and then, writing $\Rgf=\Sgf(1+2\Sgf)$ with $\Sgf=\LandauO(\Rgf)$, 
  \[
    \Fnn(x)=x^2-x+ \frac{\Rgf +(\Sgf+2x-x^2)\sqrt{(1-x)^2+4 \Sgf+4 \Sgf^2}}{2-x}.
  \]
  Hence $\Fnn(x)$ (resp.~$\Mnn(x)$) has degree $4$ (resp.~$8$) over $\cs(x,\Rgf)$. The coefficient of $x^{-1}$ in $\Mnn(x)$ (which must be $t$) is a D-finite series of $\Sgf$ (and thus of $\Rgf$), for which a linear DE can be automatically computed using creative telescoping:
  \[
    \Sgf (\Sgf+1) (8 \Sgf-1) (4 \Sgf+1) t''(\Sgf)-4 \Sgf (\Sgf+1) (8 \Sgf-1) t'(\Sgf)+2 (4 \Sgf+1)^2 t(\Sgf)=0,
  \]
 and solved in terms of hypergeometric functions:
  \[
    t=\frac{\Sgf(1+\Sgf)(1-8\Sgf)}2\cdot \frac{d}{d\Sgf}
    \left[ {(1-8\Sgf)^{-1/4}}\  _2F_1\left(\frac 1  4,\frac 1 4; 1;- \frac{64 \Sgf^3(1+\Sgf)}{1-8\Sgf}\right)\right],
    %\frac {S(1+S)}{(1-8S)^{1/4}} _2F_1\left(1/4,1/4; 1;- \frac{64 S^3(1+S)}{1-8S}\right)
 %     -\frac{6 S^3(1+S)(1-4S-8S^2)}{(1-8S)^{5/4}} _2F_1\left(5/4,5/4; 2;- \frac{64 S^3(1+S)}{1-8S}\right).
  \]
  where we recall that $\Sgf=(\sqrt{1+8\Rgf}-1)/4$.

%%%%%%%%%%%%%%%%%%%%%%%%%%%%%%%%%%%%%%%%%%%%%%%%%%%%%% 
\section{Differential algebraicity}
\label{sec:DA}
%%%%%%%%%%%%%%%%%%%%%%%%%%%%%%%%%%%%%%%%%%%%%%%%%%%%%% 
In this section, we prove that the series $\Qgf$ is D-algebraic (in all its variables) in the three cases that we have completely solved: $\om=0$, $\om=1$ and $v=1$. We rely on the closure properties of D-algebraic series established in~\cite[Sec.~6]{BeBMRa17}. We also need an additional closure property for implicitly defined series, which we establish in Section~\ref{sec:closure}. In Sections~\ref{sec:DA-om0} and~\ref{sec:DA-om1}, we try to minimise  the orders of differential equations satisfied by $\Qgf$ when $\om=0$ or $\om=1$.

Below, we denote partial derivatives with subscripts, e.g., $F_x:=\partial F/\partial x$. 

% =============================================================
\subsection{D-algebraicity for implicit functions}
\label{sec:closure}
% =============================================================
In Section~\ref{sec:fps} we have defined a series $F(x_1,\ldots, x_d)$ (say, with  coefficients in $\qs$) to be D-algebraic in the $x_i$'s if it satisfies a polynomial differential equation in each $x_i$, with coefficients in $\qs(x_1, \ldots, x_d)$. However,  we can equivalently require to have  a DE in each $x_i$ \emm with coefficients in $\qs$,, see~\cite[Sec.~6.1]{BeBMRa17}.
\begin{Proposition}\label{prop:implicitDA}
  Let $\Psi(x_1, \ldots, x_d,y)$ be a D-algebraic series in $y$ and the $x_i$'s.
 
  Let $F\equiv F(x_1, \ldots, x_d)$ be a formal power series in~$x_1, \ldots, x_d$ with no constant term, satisfying $\Psi( x_1, \ldots, x_d, F)=0$ and $\Psi_y(x_1, \ldots, x_d, F){\neq}0$. Then~$F$ is D-algebraic in $x_1, \ldots, x_d$.
\end{Proposition}

\begin{proof}
  Let us first restrict our attention to the case $d=1$, and write $\Psi(x,F(x))=0$. Let us denote by $\delta_x$ and $\delta_y$ the differentiations with respect to $x$ and $y$, respectively. Let $d_1$ and $d_2$ be the (minimal) orders over $\qs$ of $\Psi(x,y)$ in $x$ and $y$, respectively. As proved in~\cite[Prop.~6.3]{BeBMRa17}, all $x$-derivatives of $\Psi$ belong to $\qs(\Psi, \delta_x \Psi, \ldots, \delta_x^{d_1}\Psi)$. Similarly, all $y$-derivatives of $\Psi$ belong to $\qs(\Psi, \delta_y \Psi, \ldots, \delta_y^{d_2}\Psi)$. It follows that all partial derivatives of $\Psi$ (in $x$ and $y$) belong to the extension of $\qs$ generated by the series $\delta_x^i \delta_y^j\Psi$, for $0\le i\le d_1$ and $0\le j \le d_2$, and in particular to an extension of $\qs$ of finite transcendence degree. By specialization, the field $\GF$ generated over~$\qs$ by the partial derivatives of $\Psi$ evaluated at $(x,F(x))$ has also finite transcendence degree.

  Let us differentiate the identity $\Psi(x,F(x))=0$ in $x$. This gives:
  \[
    \Psi_x + (\delta_xF)\, \Psi_y=0,
  \]
  where  the derivatives of $\Psi$ are evaluated at $(x,F(x))$. Since by assumption $\Psi_y(x,F(x)){\neq}0$, we conclude that $\delta_x F$ belongs to the field $\GF$.
   More generally, by differentiating $i$ times the identity $\Psi(x,F(x))=0$, we conclude by induction on $i$ that $\delta_x^iF$ belongs to that same field, which has finite transcendence degree. This means that $F$ is D-algebraic in $x$.  

  The argument is exactly the same for $d>1$.
\end{proof}

Let us also state two other results  borrowed  from~\cite[Sec.~6]{BeBMRa17}.

\begin{Proposition}\label{prop:composition}
  If $\Lambda(y_1,\ldots, y_k)$ is a D-algebraic series (or function) of $k$ variables, and $G_1(x_1, \ldots, x_d), \ldots, G_k(x_1, \ldots, x_d)$ are D-algebraic in all $x_i$'s, then the composition $H(x_1, \ldots, x_d):=\Lambda(G_1, \ldots, G_k)$, if well defined,  is also D-algebraic in the $x_i$'s.
\end{Proposition}
This is Proposition~6.5 in~\cite{BeBMRa17}. In passing, let us mention that to make the proof of this proposition completely rigorous, one should use the language of fields of finite transcendence degree rather than rational expressions of $H$ and its derivatives, so as to avoid problems with denominators that might cancel after specializing the variables $y_i$ to the $G_i$'s.  

The next lemma also follows from~\cite[Sec.~6.2]{BeBMRa17}.

\begin{Lemma}
  The series $\theta(z)\equiv \theta(z;q)$ defined by~\eqref{theta-def-init} is D-algebraic in $z$ and $q$.
\end{Lemma}
More precisely, it is explained in the proof of Proposition~6.7 in~\cite{BeBMRa17} that the following series,
\[
  \bar\theta(z;\tau):=\sum_{n\in \zs}(-1)^n e^{i(2n+1)z+{\pi i}\tau(n+1/2)^2}
\]
is D-algebraic in $z$ (due to its relation to Weierstrass' $\wp$ function) and $\tau$ (due to the heat equation). (There is a typo in~\cite{BeBMRa17}, where the factor $(-1)^n$ is missing in the definition of $\bar\theta$.) This series is related to our series $\theta(z;q)$ by
\[
  \bar \theta(z; \tau)=2i  e^{{\pi i}\tau/4} \theta(z;q),
\]
with $q= e^{2{\pi i}\tau}$, and the D-algebraicity of $\theta$ follows by a (D-algebraic) change of variables.

\begin{Corollary}
  The specializations of $\Qgf$ at $\om=0$, at $\om=1$, and at $v=1$ (that is, the series $\Qgf_0=2\Ggf$, $\Qgf_1$ and $\tilde \Qgf$ of Theorems~\ref{thm:general},~\ref{thm:quartic} and~\ref{thm:allomega}), are D-algebraic (in $t$ and $v$ for the first two series, in $t$ and $\om$ for the third one).
\end{Corollary}
\begin{proof}
  In all three cases, we define a first bivariate series (either $\Rgf_0$, or $\Rgf_1$, or $q$) by an implicit D-algebraic equation; see~\eqref{t-R0-first},~\eqref{t-R1-first} and~\eqref{t-q}. Proposition~\ref{prop:implicitDA} implies that this series is D-algebraic. Then we express $\Qgf_0$, $\Qgf_1$ or $\tilde \Qgf$ as a D-algebraic series in this implicit series and the two remaining variables. We conclude using Proposition~\ref{prop:composition}.
\end{proof}

% ===================================================
\subsection{The case $\bm{\om=0}$}
\label{sec:DA-om0}
% ===================================================

Let us now be more explicit, and try to obtain DEs of small orders. We refer to the {\sc Maple} session available on~\cite{bmep-ref-arxiv}
% our webpages
for details.  Recall the definition~\eqref{t-Omega0} of $\Rgf\equiv \Rgf_0$, with
\[
  \Omega(r,u)=\sum_{n,k\geq0}\frac{1}{n+1}{2n\choose n}{2n+k\choose k}{2n+k\choose n} u^kr^{n+1}.
\]
This is a D-finite series in two variables $r$ and $u$. It can be checked that its partial derivatives span a vector space of dimension (at most) $3$ over $\qs(r,u)$, generated for instance by $\Omega$, $\Omega_r$ and~$\Omega_u$.
In particular, the second order derivatives $\Omega_{r,r}$, $\Omega_{u,u}$ and $\Omega_{r,u}$ can be expressed as linear combinations of $\Omega$, $\Omega_r$ and $\Omega_u$. More precisely, they are linear combinations of $\Phi:=\Omega-u \Omega_u$ and $\Omega_r$. This means that:
\begin{itemize}
\item we have a first  vector space, generated by $\Omega$, $\Omega_r$ and $\Omega_u$, closed by partial derivations in~$r$ and $u$ --- that is, containing all partial derivatives of $\Omega$,
\item we also have another (a priori smaller)  vector space, generated by $\Phi$ and $\Omega_r$, or by $\Phi$ and $\Phi_r$,  closed by partial derivations.  (Here we use the fact that $\Phi_u=-u \Omega_{u,u}$ only involves a second order derivative of $\Omega$.) This space thus contains all partial derivatives of $\Phi$, as well as all partial derivatives of $\Omega$, except possibly $\Omega$ and $\Omega_u$.
\end{itemize}
Let us now differentiate~\eqref{t-Omega0} with respect to $t$ (remembering that $\Rgf$ is a function of $t$). This gives a relation between $\Omega_r$ and $\Omega_u$, evaluated at $(\Rgf,t(v-1))$:
\[
  1= \Rgf ' \Omega_r + (v-1) \Omega_u.
\]
Here, the derivative of $\Rgf$ is with respect to $t$. Once multiplied by $t$, this can be rewritten as
\beq\label{R1-Omr}
t  \Rgf ' \Omega_r = \Phi,
\eeq
where the functions are evaluated at $(\Rgf, t(v-1))$ again. Recall that $\Omega_r(r,u)$ and $\Phi(r,u)$ generate a vector space over $\qs(r,u)$ that is closed by derivation. Hence,  differentiating the above identity with respect to $t$ gives another identity of the form $A\Omega_r = \Phi$, 
for $A$ a rational function in $t$, $v$, $\Rgf $, $\Rgf '$ and $\Rgf ''$. Comparing the last  two equations proves that $A= t\Rgf '$, so that $\Rgf $ is D-algebraic in $t$ of order (at most) $2$. More precisely, writing $u=t(v-1)$, we have:
\begin{align}\label{ED-R0}
&  t^2 \Rgf  \left(16 \Rgf ^2 +\left(8 u^{2}+20 u -1\right)\Rgf  +u \left(u -1\right)^{3} \right)
  \Rgf ''
  -t^3\left(4  \Rgf -3  u \left(u -1\right)\right) (\Rgf ')^{3}\\
 & -3 t^{2} u \left(8 u +1\right) \Rgf  (\Rgf ')^2
  +3u^2 \Rgf \left(12 t\Rgf  - t \left(u -1\right)^{2} \right) \Rgf '
  -16u^{2} \Rgf  ^{3} +u^{2} \left(4 u -1\right) \left(u -1\right) \Rgf ^{2}=0.
  \nonumber
\end{align}
 The case $v=1$ studied in~\cite{mbm-aep1} (that is, $u=0$) is much simpler, with $\Rgf  (16 \Rgf -1) \Rgf ''=4t (\Rgf ')^3$, but the order remains $2$ only. For the $v$-derivatives of $\Rgf$, the situation is more generic, because the counterpart of~\eqref{R1-Omr}, namely 
\beq\label{R2-Omr}
(v-1) \Rgf_v \Omega_r +t = \Phi,
\eeq
is no longer homogeneous. Consequently, we have to differentiate~\eqref{t-Omega0} three times in $v$ (rather than twice when dealing with $t$-derivatives) and we end up with a third order differential equation satisfied by $\Rgf$.

We now want to work out a differential equation (first, in the main variable $t$) for the \gf\ $\Qgf_0=2\Ggf$, given by~\eqref{Q0-R}. Let us define $L$ and $\Lambda$ by  
\[
  \Lgf:=  t^2\left(  \Qgf_0+v\right)= \Lambda(\Rgf,t(v-1)),
\]
with
\[
  \Lambda(r,u)= \sum_{{n,k\geq0},{~n+k>0}}\frac{1}{n+1}{2n\choose n}{2n+k\choose k}{2n+k-1\choose n} u^kr^{n+1}.
\]
Here we find that $\Lambda$ belongs to the vector space spanned by $1$, $\Omega$, $\Omega_r$ and $\Omega_u$. The form of its expression shows that it will be easier to work at first with
\[
  \overline \Lambda:= \Lambda + \frac{2r} 3 \left(1 + \Omega/u\right),
\]
which is found to lie in the 2D vector space spanned by $\Phi= \Omega- u\Omega_u$ and $\Omega_r$. Recall that this space is closed by
% derivatives,
derivation, and hence contains the partial derivatives of $\overline \Lambda$. Define accordingly
\beq\label{hatL}
\overline \Lgf:=\Lgf+ \frac {2 v\Rgf}{3(v-1)} = \overline \Lambda(\Rgf,t(v-1)).
\eeq
We can now write this function, and all its $t$-derivatives, as linear combinations  of $\Phi$ and $\Omega_r$, evaluated at $(\Rgf,t(v-1)) $, with coefficients in $\qs(\Rgf,\Rgf', \Rgf'', \ldots, t(v-1))$. Using~\eqref{R1-Omr}, this gives expressions of these derivatives as multiples of $\Phi$. In particular, we find
\beq\label{dlog-h}
\frac{\overline \Lgf '}{\overline \Lgf} =
\frac 4 t\cdot  \frac{t^2 (\Rgf')^2+ tu (1-u) \Rgf'+3u^2\Rgf }
{4t \Rgf \Rgf'  +3tu (1-u) \Rgf'+u(1+8u) \Rgf},  
\eeq
with $u=t(v-1)$. Going back to the series $\Lgf$, this means that we have a  polynomial equation, say~$(E_1)$,
   %    of degree~$1$
that relates $\Lgf, \Lgf',\Rgf$ and  $\Rgf'$, with coefficients in $\qs(t,v)$. Now we differentiate $(E_1)$ with respect to $t$: this gives a polynomial equation in $\Lgf, \Lgf', \Lgf'', \Rgf, \Rgf'$ and $\Rgf''$ (and $t$ and $v$). 
  % , of degree $1$ in the first three series.
  Thanks to~\eqref{ED-R0}, we can express~$\Rgf''$ as a rational function of $\Rgf$ and $\Rgf'$. We thus end up with a polynomial equation $(E_2)$ in $\Lgf, \Lgf', \Lgf'', \Rgf$ and $\Rgf'$, not involving $\Rgf''$.
  %, of degree $1$ in each of the first three series.
% and use the second order differential equation satisfied by $\Rgf$ to write the result in terms of $\Lgf, \Lgf', \Lgf'', \Rgf$ and $\Rgf'$ (and $t$ and $v$).
  We repeat the above procedure:  we differentiate $(E_2)$ and eliminate $\Rgf''$. This gives a polynomial equation $(E_3)$ in $\Lgf, \Lgf', \Lgf'', \Lgf''', \Rgf$ and~$\Rgf'$. At this stage we have $3$ polynomial equations $(E_1)$, $(E_2)$, $(E_3)$ that involve $\Rgf$ and~$\Rgf'$, as well as $\Lgf$ and its first three derivatives. We eliminate $\Rgf$ and $\Rgf'$ between these three equations (using resultants, or Gröbner bases): this gives a polynomial relation between  $\Lgf$ and its first three derivatives, with coefficients in $\qs(t,v)$.
  % implies that these four series are {algebraically} related, which
  It means that $\Lgf$, and thus $\Qgf_0$ and $\Ggf$, satisfy a third order differential equation.

Similarly, to obtain a differential equation in $v$, we start from~\eqref{hatL} and differentiate it with respect to $v$ up to fourth order, which yields an expression for $L$ and each of its first four $v$-derivatives in terms of $t,v,\Phi,\Omega_{r}$ and $v$-derivates of $\Rgf$. We then use~\eqref{R2-Omr} and the third order differential equation in $v$ satisfied by $\Rgf$ to write these expressions in terms of the four series $\Rgf, \Rgf_{v}, \Rgf_{v,v}$ and $\Phi$ (as well as $t$ and $v$). This means that $L$ and its first four $v$-derivatives are algebraically related over $\qs(t,v)$, so that $\Lgf$ and $\Qgf_0$ satisfy differential equations of order $4$ in $v$.

% ===================================================
\subsection{The case $\bm{\om=1}$}
\label{sec:DA-om1}
% ===================================================
The derivation of DEs for $\Rgf_1$ and $\Qgf_1$ completely parallels what we have done for $\om=0$. We refer again to our {\sc Maple} session for details. Let us simply give two explicit equations: first, the second order DE in $t$ satisfied now by $\Rgf\equiv \Rgf_1$:
\begin{multline*}
  t^2\Rgf \left(27 \Rgf   ^{2} +\left(36 u -1\right) \Rgf-u \left(4 u -1\right)^{2}  \right) \Rgf''
  -2t^{3} \left(3 \Rgf   -2  u \left(4 u -1\right)\right)(\Rgf')  ^{3}\\
  -6 t^{2} u \left(12 u +1\right) \Rgf   (\Rgf') ^{2}
  +6tu^2 \Rgf \left(18 \Rgf +  \left(4 u -1\right)    \right) \Rgf'
  -54  u^{2}\Rgf   ^{3} -2 u^{2} \left(12 u -1\right) \Rgf   ^{2} =0,
\end{multline*}
and then, an expression of the following log-derivative (again, with respect to $t$):
\[
  \frac{\overline \Lgf '}{\overline \Lgf} =
  \frac 3 t\cdot  \frac{t^2 (\Rgf')^2+ tu (1-4u) \Rgf'+6u^2\Rgf }
  {3t \Rgf \Rgf'  +2tu (1-4u) \Rgf'+u(1+12u) \Rgf},
\]
where $\overline \Lgf$ is defined by
\[
  \overline L:=  t^2\left( \Qgf_1+v\right)+ \frac{v \Rgf}{2(v-1)}.
\]
These expressions imply that $\Qgf_1$ satisfies a third order DE in $t$.

%%%%%%%%%%%%%%%%%%%%%%%%%%%%%%%%%%%%%%%%%%%%%%%%%%%%%%%%%%%%%%%%%%%%% 
\section{Towards more combinatorial proofs}
\label{sec:combin}
%%%%%%%%%%%%%%%%%%%%%%%%%%%%%%%%%%%%%%%%%%%%%%%%%%%%%%%%%%%%%%%%%%%%% 

In this section, we first give simple combinatorial interpretations in terms of trees of the series~$\Qgf$ for the cases $\om=0$ and $\om=1$ (Sections~\ref{sec:trees0} and~\ref{sec:trees1}). This raises the question of finding bijective proofs. Then, in Section~\ref{sec:comb01}, we take $\om=0$ and $v=1$, and explain in this special case how the characterisation of $\Mnn(x)$ given by~\eqref{eq-M} can be proved combinatorially without introducing the second catalytic variable $y$ of Section~\ref{sec:funceq}.

% ============================================
\subsection{A tree interpretation for $\bm{\om=0}$}
\label{sec:trees0}
% ============================================

We begin with combinatorial interpretations of the series $\Rgf_0\equiv \Rgf_0(t,v)$ and $2\Ggf= \Qgf_0(t,v)$ of Theorem~\ref{thm:general}.

\begin{Proposition}\label{prop:trees0}
  Consider rooted plane binary trees with edges of two types (say
  solid and dashed, see Figure~\ref{fig:RG}).  Define the \emm charge, of such a tree to be the number of solid edges minus the number of dashed edges. Call a tree of charge $0$ {\emm balanced,}. Call a leaf  {\em special} if it is the right child of its parent and is incident to a dashed edge.

  Then the series $t-\Rgf_{0}$ of
  Theorem~\ref{thm:general} counts, by leaves ($t$) and special leaves ($v$), balanced trees in which no proper subtree
  is balanced. Moreover, the series $t^2(v+2\Ggf)$ counts the subset of these trees in which the root vertex is joined to its left child by a solid edge.
\end{Proposition}

\begin{figure}[ht]
  \setlength{\captionindent}{0pt}
  \scalebox{0.9}{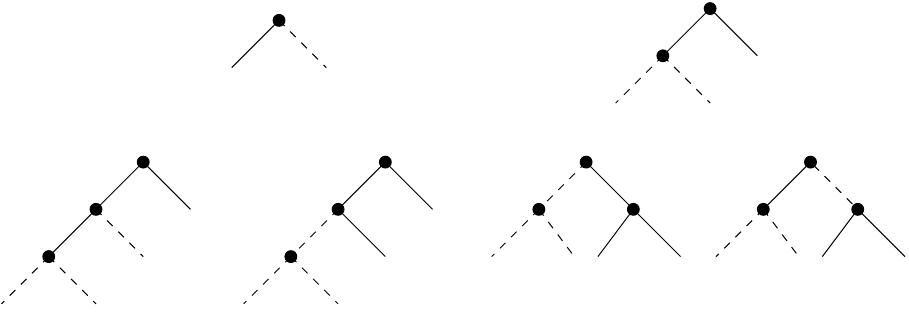}
  \caption{The balanced trees with at most 4  leaves involved in the
    expansion of $t-\Rgf_{0}=(1+v)t^2+(1+3v)t^3+(3+14v+3v^2)t^4+\LandauO(t^5)$, for the series $\Rgf_{0}$ of    Theorem~\ref{thm:general}. The polynomials keep track of      the  various plane embeddings  and the exchange of the two    edge types.}
  \label{fig:RG}
\end{figure}

See Figure~\ref{fig:RG} for an illustration. This combinatorial interpretation refines, with the additional variable $v$, the one given in~\cite[Prop.~9.2]{mbm-aep1}. It would probably be possible to refine the proof of that proposition so as to include~$v$. But  we give here a different proof, which provides in addition a tree interpretation of the series $\Mnn(x)$ defined by~\eqref{M-def}.
\begin{proof}
  Let $U(x)\equiv U(x,t,v)$ be the generating function for rooted plane binary trees with solid and dashed edges, having no proper balanced subtrees, counted by charge (variable $x$), leaves (variable $t$) and special leaves (variable $v$). Let $U_0=[x^{0}]U(x)$. The first part of the proposition is equivalent to the statement that $U_0=t-\Rgf_{0}$. We claim that the following equation holds:
  \begin{equation}\label{eq:Ueq}
    U(x)=(x+1/{x})(t+U(x)-U_0)\big(tx+{tv}/{x}+(x+{1}/{x})(U(x)-U_0)\big).
  \end{equation}
  This follows simply by decomposing a tree counted by $U(x)$ into its left and right subtrees:
  \begin{itemize}
  \item   the initial factor $(x+{1}/{x})$ comes from the choice of the type of the left edge of the root ($x$ if the edge is solid, ${1}/{x}$ otherwise),
  \item the factor $t+U(x)-{U_0}$ counts the possibilities for the left subtree: $t$ for a leaf,
      and $U(x)-{U_0}$ for a (non-leaf) unbalanced tree,
  \item similarly, the term $tx+{tv}/{x}$ corresponds to the case where the right subtree of the root is a leaf (solid or dashed), while $(x+{1}/{x})(U(x)-{U_0})$ counts the possibilities for  a non-leaf unbalanced  right subtree.
  \end{itemize}
  The above equation defines $U(x)$ uniquely as a series in $t$ with no constant term. Observe that~$U(x)$ belongs in fact to $t\qs[v,x,1/x][[t]]$. More precisely, given that $U(-x)$ satisfies the same equation as $U(x)$, it is a series of $t\qs[v,x^2,1/x^2][[t]]$. Now define %writing
  \beq\label{MU}
  M(y):=\frac{tv}{y}+\frac{t}{1-y}+\frac{1}{y(1-y)}\left(U\left(\sqrt{\frac{y}{1-y}}\right)-U_0\right),
  \eeq
  which is a series $t\qs[v,y,1/y, 1/(1-y)][[t]] \subset \qs[v]((y))[[t]]$.  Substituting $x\to \sqrt{\frac{y}{1-y}}$ in~\eqref{eq:Ueq} yields
  \[
    t-U_0=(1-y)(yM(y)-t(v-1))(1-M(y)),
  \]
  which is precisely~\eqref{eq-M} with $R$ replaced by $t-U_0$. %Note that
  This series does not depend on~$y$. Moreover, it follows from $[x^{0}]U(x)=U_0$ that $[y^{-1}]M(y)=tv$. As discussed in Section~\ref{sec:om0}, based on Proposition~\ref{prop:M}, this means that $M(y)$ is the series $\Mnn(y)$ defined by~\eqref{M-def} (for $\om=0$). In particular,  we have $t-U_0=\Rgf_{0}$ as claimed in the proposition. 

  We now consider the generating function $\widetilde{U}(x)$ counting trees in which the left edge of the root is solid. By the same decomposition as above,
  we find
  \[
    \widetilde{U}(x)=x(t+U(x)-U_0)\big(tx+{tv}/{x}+(x+{1}/{x})(U(x)-U_0)\big)
    = \frac{x^2}{1+x^2} U(x).
  \]
  It remains to prove that $t^2(v+2 \Ggf)=[x^{0}]\widetilde{U}(x)$. Given that $\widetilde U(x)$ has coefficients in $\qs[v,x^2,1/x^2]$, we have
  \begin{align*}
    [x^{0}]\widetilde{U}(x)&=[y^{0}]\frac{1}{1-y}\, \widetilde{U}\!\left(\sqrt{\frac{y}{1-y}}\right)\\
                           &=[y^0]\frac{y}{1-y\,}U\!\left(\sqrt{\frac{y}{1-y}}\right)\\
                           &=[y^{-1}]\left(\frac{U_0}{1-y}+yM(y)-tv-\frac{ty}{1-y}\right)  \qquad\text{by~\eqref{MU}}\\
                           &=[y^{-2}]M(y)=t^2(v+\Qgf_0) \qquad \text{by~\eqref{Q-M}.}
  \end{align*}
  To conclude, we recall that $\Ggf=\Qgf_0/2$.
\end{proof}

% ============================================
\subsection{A tree interpretation for $\bm{\om=1}$}
\label{sec:trees1} 
% ============================================

We now move to  the series $\Rgf_1\equiv \Rgf_1(t,v)$ and $\Qgf_1\equiv \Qgf_1(t,v)$ of Theorem~\ref{thm:quartic}. In the following proposition we give an interpretation of $\Rgf_1$ and $\Qgf_1$  in terms of unary/binary trees. This is \emm not, a refinement of the interpretation in our previous article, which involved instead  ternary trees. 

\begin{figure}[ht]
  \setlength{\captionindent}{0pt}
  \scalebox{0.9}{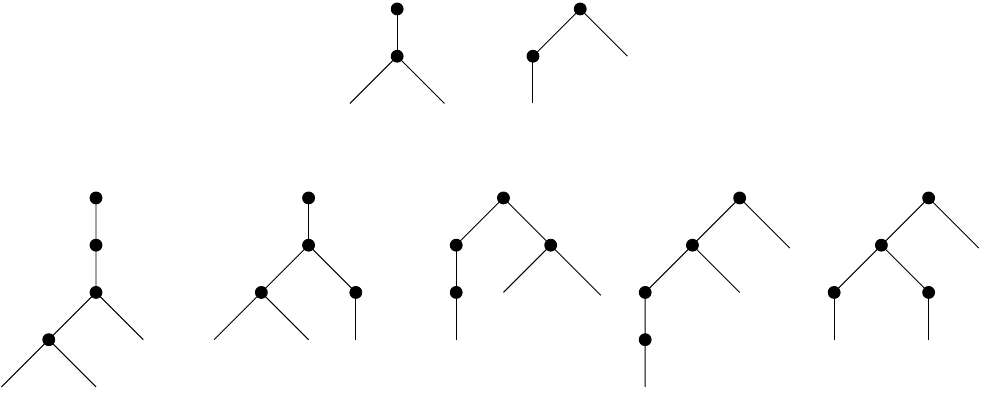} 
  \caption{The balanced trees with $2$ or $3$  leaves involved in the
    expansion of $t-\Rgf_{1}=(1+2v)t^2+2(1+4v+v^2)t^3+\LandauO(t^4)$, for the series $\Rgf_{1}$ of
    Theorem~\ref{thm:quartic}.  The polynomials keep track of
    the  various plane  embeddings.} 
  \label{fig:R4}
\end{figure}

\begin{Proposition}\label{prop:trees1}
  Consider rooted plane trees where each inner vertex has $1$ or $2$ children. We call these unary/binary trees. Define the \emm charge, of such a tree to be the number of binary vertices minus the number of unary vertices. Call a tree of charge $0$ {\emm balanced,}. Call a leaf of such a tree {\em special} if it is the right child of its parent (which has $2$ children).

  Then the series $t-\Rgf_{1}$ of
  Theorem~\ref{thm:quartic} counts, by leaves ($t$) and special leaves ($v$), balanced trees in which no proper subtree is balanced. Moreover, the series $t^2(v+\Qgf_{1})$ counts trees with charge $1$ having no balanced subtrees. 
\end{Proposition}
See Figure~\ref{fig:R4} for an illustration.

\begin{proof}
  Let $U(x)$ be the generating function for   unary/binary trees with no proper balanced subtrees, counted by charge (variable $x$), leaves (variable $t$) and special leaves (variable $v$). Let $U_0=[x^{0}]U(x)$.
  We claim that the following equation holds:
  \begin{equation}\label{eq:Ueqquartic}
    U(x)=x(t+U(x)-U_0)(tv+U(x)-U_0)+\frac{1}{x}(t+U(x)-U_0).
  \end{equation}
  This follows simply by decomposing a tree counted by $U$ at its root.
  In the case that the root vertex is binary, the factor $t+U(x)-U_0$ accounts for all possibilities for the left subtree, while the factor $tv+U(x)-U_0$ accounts for all possibilities for the right subtree. The factor $x$ accounts for the binary root. In the case of a unary vertex there is only the factor $(t+U(x)-U_0)$, and the factor $1/x$ for the unary root. Since a unary/binary tree with $n+1$ leaves has~$n$ binary vertices, the series $U(x)$ belongs to $t\qs[v][[tx,1/x]]$. Now define
  \[
    M(y):=\frac{1}{y}\left(tv+U\!\left(\frac{1}{y}\right)-U_0\right),
  \]
  which is a series in $t/y\qs[v][[t/y,y]]$.   Substituting $x\to \frac{1}{y}$ in~\eqref{eq:Ueqquartic} yields
  \[
    t-U_0=(yM(y)-t(v-1))(1-y-M(y)),
  \]
  which is precisely~\eqref{eq-M1} with $R$ replaced by $t-U_0$. Moreover, it follows from $[x^{0}]U(x)=U_0$ that $[y^{-1}]M(y)=tv$. As discussed in Section~\ref{sec:om1}, based on Proposition~\ref{prop:M}, this means that $M(y)$ is the series $\Mnn(y)$ defined by~\eqref{M-def} (for $\om=1$). In particular, $t-U_0=\Rgf_{1}$. 

  Finally, recall from~\eqref{Q-M} that $t^2(\Qgf_{1}+v)=[y^{-2}]M(y)$. This can be rewritten as $[x^{1}]U(x)$,   which is, by definition of $U(x)$, the generating function for unary/binary trees with charge $1$ having no balanced subtrees.
\end{proof}

%=====================================================
\subsection{A combinatorial way to characterise $\bm{\Mnn(x)}$ (case $\bm{\om=0, v=1}$)}
\label{sec:comb01}
%=====================================================

In Section~\ref{sec:charac}, we defined the series
  $\Mnn(x)$ in one catalytic variable $x$ by
  \beq\label{M-def-bis}
    \Mnn(x):=\frac{1}{x}\Pnn\left(\frac{1}{x}\right)+[y^{1}]\Dnn(x,y),
  \eeq
  and gave an intrinsic characterisation of it; see  Proposition~\ref{prop:M-charac}. However, this characterisation is derived from the equations of Section~\ref{sec:funceq}, which involve the \emm two, catalytic variables $x$ and $y$. It is natural to ask if this second variable $y$ can be avoided when deriving the characterisation of~$\Mnn(x)$. This is what we do in this section, for the special case $\om=0$, $v=1$,
corresponding to the enumeration of Eulerian orientations --- the problem initially posed in \cite{BoBoDoPe}. In this case, we showed in Section~\ref{sec:om0} that the characterisation of $\Mnn(x)\in t\qs((x))[[t]]$  reduces to the fact that the series
  \[
    \Rgf_0:=x(1-x)\Mnn(x)(1-\Mnn(x))
  \]
  is independent of $x$, together  with the condition $t=[x^{-1}]\Mnn(x)$. Here, we explain combinatorially why $x(1-x)\Mnn(x)(1-\Mnn(x))$ does not depend on $x$. Since  a similar property   holds for all~$v$, not just $v=1$, it is natural to ask
  whether an analogous combinatorial explanation  exists in this more general case.

\begin{Proposition}\label{prop:Pnn_om0_v1A}
  When $\om=0$ and $v=1$, we have the following additional identities:
  \[
    [y^{1}]\Dgf(x,y)=\frac{1}{1-x}\Pgf\left(\frac{t}{1-x}\right),
  \]
  and
   \begin{equation}\label{Pnn_om0_v1A}
    \Pnn\left(\frac{1}{x}\right)=t+\frac{1}{x}\Pnn\left(\frac{1}{x}\right)^{2}+2[x^{<0}]\left(\frac{1}{1-x}\Pnn\left(\frac{1}{x}\right)\Pnn\left(\frac{1}{1-x}\right)\right).
  \end{equation}
   The latter equation implies that $x(1-x)\Mnn(x)(1-\Mnn(x))$ does not depend on $x$.
\end{Proposition}
We note that Timothy Budd has independently deduced the second equation \cite{budd-flat}.

\begin{figure}[ht]
  \setlength{\captionindent}{0pt}
  \scalebox{0.8}{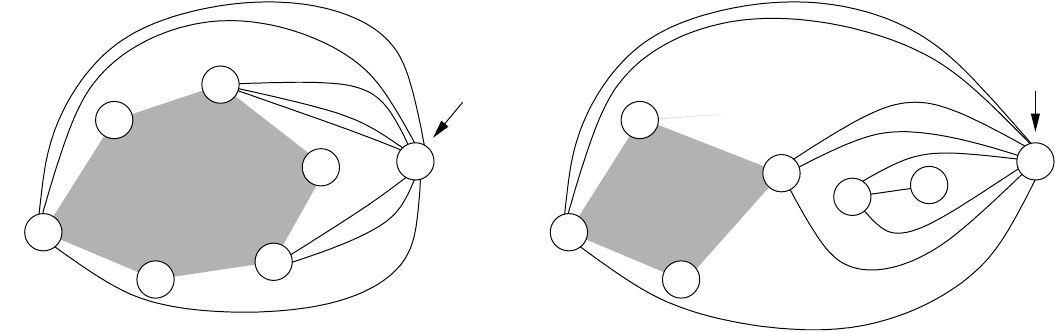}
  \caption{From a D-patch of outer degree $2$ to a $1$-patch $P$. The rightmost D-patch does not contribute as we forbid bicoloured quadrangles.} 
  \label{fig:charac}
\end{figure}

\begin{proof}
  We refer to Definition~\ref{def:patches} and Section~\ref{sec:GF} for the definitions of patches, D-patches and their \gfs. Since we take $\om=0$, we only count maps with no inner bicoloured quadrangle.  Starting with a $D$-patch $D$ with outer degree $2$, let~$v_{0}$ be the root vertex and $c_{1}$ be the  outer corner with label $1$  (Figure~\ref{fig:charac}).
  Remove from $D$ the root vertex and all incident edges: since we forbid bicoloured quadrangles, this leaves a unique connected component $P$, which we root at the corner inherited from $c_{1}$ (Figure~\ref{fig:charac}, left).
   All inner faces of $P$ 
  are quadrangles. Any  outer vertex $v$ of~$P$ comes from a face of~$D$ containing $v_0$: if $v$ is adjacent to $v_0$ in $D$ it has label~$1$ by definition of D-patches, otherwise it has  label~$2$ because we forbid  bicoloured
 quadrangles in~$D$. Hence, $P$ is a $1$-patch. To reconstruct a $D$-patch $D$ starting from $P$, one can choose to add any number of digons at each outer corner of $P$ labelled $1$, making the choice twice for the root corner of $P$. Hence the contribution to $[y^{1}]\Dgf(x,y)$ corresponding to a patch~$P$ with outer degree $2m$ and $n$ inner quadrangles is $\frac{1}{(1-x)^{m+1}}t^{m+n}$. Summing over all patches $P$ yields the first equation of the proposition. After the change of variables~\eqref{change}, it reads
    \beq\label{DPmod}
      [y^{1}]\Dnn(x,y)=\frac{1}{1-x}\Pnn\left(\frac{1}{1-x}\right).
    \eeq
  Substituting $y$ by $1/x$ and then the above identity into~\eqref{Pnn} yields~\eqref{Pnn_om0_v1A}.  Moreover, substituting into our definition~\eqref{M-def-bis} for $\Mnn$, we can write  
  \[
    \Mnn(x)=\frac{1}{x}\Pnn\left(\frac{1}{x}\right)+\frac{1}{1-x}\Pnn\left(\frac{1}{1-x}\right).
  \]
  Given that the series $\Pnn(x)$ has coefficients in $\qs[x]$,  this implies that the coefficient of $t^n$ in $\Mnn(x)$ is a polynomial in $1/(x(1-x))$, with no constant term. Equation~\eqref{Pnn_om0_v1A} is equivalent to
   \[
      [x^{<0}]\, x\!\left(\Mnn(x)-\Mnn(x)^{2}\right)=0.
    \]
    It follows that $[x^{<0}]x(1-x)\left(\Mnn(x)-\Mnn(x)^{2}\right)=0$. Since $x(1-x)\left(\Mnn(x)-\Mnn(x)^{2}\right)$ is a series in $t$ with polynomial coefficients in $1/(x(1-x))$, the fact that it contains no negative power of~$x$ forces it to be independent of $x$, as claimed.
\end{proof}

We believe that this combinatorial explanation
could open the door to deeper results about the structure of random Eulerian orientations. As a specific example, we briefly describe how this could allow one to fairly efficiently generate \emm Boltzmann patches, at criticality with a  prescribed outer degree $2\ell$.
By this, we mean that each  patch $P$ with outer degree
$2\ell$ occurs with probability proportional to $t_{c}^{\ff(P)}$ where $\ff(P)$ 
is the number of faces of $P$ 
and $t_{c}$ is the radius of convergence of the series $[y^\ell] \Pgf(y)$. Recall that $[y^1]\Pgf(y)= 1+\Qgf_0$, so that by Theorem~1.2 in~\cite{mbm-aep1}, $t_{c}={1}/{(4\pi)}$ for $\ell=1$ and moreover, $[y^1]\Pgf(y)$ converges at $t_c$. We expect this to hold  for any $\ell$. Boltzmann models of this form have been considered for several classes of maps, see for example
\cite{BettMier-disksI,gwynne2019convergence,holden2023convergence}. In order to sample such objects, in the spirit of~\cite{boltzmann}, one could use our combinatorial decompositions underlying~\eqref{Pnn_om0_v1A} with the probabilities of the decomposition at each step being determined by the coefficients of $\Pnn(x;t_c)$ (here we have made the dependency of $\Pnn$ in $t$ explicit).  More precisely, let us write 
  \[
    \Pnn(x;t_{c})=\sum_{\ell\ge 0}p_{\ell}x^{\ell},
  \]
where $p_\ell$ is the \gf\ of patches of outer degree $2\ell$, evaluated at $t_c$ (modified by the change of variables~\eqref{change}). In particular, each patch $P$ of outer degree $2\ell$ should occur with probability $t_{c}^{\ell+1 +\ff(P)}/p_\ell$. Then by~\eqref{DPmod},
\[
  [y^1] \Dnn(x,y;t_c)= \frac{1}{1-x}\Pnn\left(\frac{1}{1-x};t_{c}\right)=\sum_{j\ge 0}d_{j}x^{j},
\]
where
\[
  d_{j}=\sum_{k\ge 0}{k+j\choose j}p_{k},
\] 
and  Equations~\eqref{DPmod} and~\eqref{Pnn_om0_v1A} imply that for $\ell>0$,
\[
  p_{\ell}=\sum_{j=0}^{\ell-1}p_{j}p_{\ell-1-j}+2\sum_{j\ge 0}p_{\ell+j}d_{j}
  = \sum_{j=0}^{\ell-1}p_{j}p_{\ell-1-j}+2\sum_{j\ge 0}p_{\ell+j}\sum_{k\ge 0}{k+j\choose j}p_{k}.
\]
Then in order to sample a patch $P$ with outer degree
$2\ell$, one can use the following recursive algorithm.
\begin{itemize}
\item If $\ell=0$ return the atomic patch, that is, the map with only one vertex and no edge.
\item Otherwise, choose the  pair $( 1,j)$ for some $0\leq j\leq \ell-1$
  % n-1
  or $(2,j)$ for $0\leq j$, where $(1,j)$ is chosen with probability ${p_{j}p_{\ell-1-j}}/{p_{\ell}}$ and $(2,j)$ is chosen with probability ${2p_{\ell+j}d_{j}}/{p_{\ell}}$.
\item If the chosen pair is $(1,j)$:
\begin{itemize}
\item recursively use the algorithm to sample patches $P_{1}$ with outer degree $2j$ and $P_{2}$ with outer degree $2(\ell-1-j)$,
\item construct $P$ by
  joining the root vertex of  $P_{1}$ and the co-root vertex of $P_{2}$ by an edge, which is the new root edge (this is the construction that underlies the second term in the expression of Lemma~\ref{NewPequation}; see also the second case in Figure~\ref{fig:newDproof_cases}).
\end{itemize}
\item If the chosen pair is $(2,j)$:
\begin{itemize}
\item recursively use the algorithm to sample a patch $P_{1}$ with outer degree $2(\ell +j)$,
\item choose a non-negative integer $k$   with probability ${{k+j\choose j}p_{k}}/{d_{j}}$,
\item recursively use the algorithm to sample a patch $P_{2}$ with outer degree $2k$,
\item construct from $P_2$ a D-patch $D$ with outer degree $2$ and $j$ inner digons   as described in the proof of Proposition~\ref{prop:Pnn_om0_v1A}, uniformly in one of the ${k+j\choose j}$ possible ways,
\item construct an E-patch $E$ with one outer corner labelled $1$ and $\ell$ outer corners labelled~$-1$ using  $P_{1}$ as subpatch and $D$ as contracted map, as in the proof of Proposition~\ref{prop:symmetry} (see also Figure~\ref{fig:decontract}),
  \item finally, construct $P$ in one of two ways from $E$ (with equal probability) as in the proof of Lemma~\ref{NewPequation}.
\end{itemize}
\end{itemize}
In order to apply this algorithm, the values $p_{j}$ and $d_{j}$ must first be computed to a sufficient precision.
In fact we will obtain exact values by computing $\Pnn(x;t_{c})$ explicitly below.
For fixed $t\in [0, 1/(4\pi)]$, corresponding to $\Rgf_0\in [0,\frac{1}{16}]$, it follows from~\eqref{M-def-bis} that we can determine $\Pnn$ by the integral
\[
 \frac{1}{y}\Pnn\left(\frac{1}{y}\right)=\frac{1}{2\pi i}\oint_{\mathcal{C}}\frac{1}{y-x}\Mnn(x)dx,
\]
where $\mathcal{C}$ is a counterclockwise contour within the annulus where $\Mnn(x)$ converges and $y$ is taken to be outside the contour $\mathcal{C}$. Recall that this annulus shrinks to a circle of radius $y_c=1/2$ centered at the origin when $t\rightarrow t_c$, as discussed below Prediction~\ref{pred:0}, and we take $\mathcal C$ to be this circle.
At the critical point, $\Rgf_0=\frac{1}{16}$, and we can compute the above  integral exactly:
\[
  \Pnn(x;t_{c})=\frac{1}{4x}-\frac{2-x}{4\pi x\sqrt{1-x}}\arccos\left(\frac{x}{2-x}\right).
\]
Rearranging this expression yields
\[
  [y^1]\Dnn\left(x,y;t_{c}\right)=\frac{1}{1-x}\Pnn\left(\frac{1}{1-x};t_{c}\right)
  =\frac{1}{4}-\frac{1-2x}{4\pi\sqrt{x(x-1)}} \arccos\left(\frac 1 {1-2x}\right).
\]
 Both series $\Pnn$ and $\Dnn$ are D-finite, and their
    coefficients satisfy linear recurrences of order $2$ with polynomial coefficients, which allows one to compute them very quickly. One even obtains closed form expressions: for $n>0$,
   \[
     p_n= \frac 1 { \pi 4^{n+1}} \binom{2n}{n-1}
     \left( 2 - \frac \pi  2 -  \sum_{k=1}^{n-1}\frac{2^k (k-1)!k!}{(2k+1)!}
     \right)
   \]
   and
      \[
     d_n= \frac{4^n}{2\pi} \frac{(n-1)!(n+1)!}{(2n+1)!} \sum_{k=0}^{n-1} \frac{(2k+1)!}{2^k k! (k+2)!},
   \]
   with $p_0=1/(4\pi)$ and $d_0=1/4-1/(2\pi)$.

%%%%%%%%%%%%%%%%%%%%%%%%%%%%%%%%%%%%%%%%%%%%%%%%%%%%%%%%%%%%%%%%%%%% 
\section{Final remarks}
\label{sec:final}
%%%%%%%%%%%%%%%%%%%%%%%%%%%%%%%%%%%%%%%%%%%%%%%%%%%%%%%%%%%%%%%%%%%%
In this section we briefly mention some further directions of research and open questions on labelled quadrangulations.

In the last 20 years, a common direction of study on planar map models has been to consider their geometric properties under either a local limit \cite{angel-schramm} or scaling limit \cite{le-gall-topological} as their size goes to infinity. Understanding the geometry of labelled quadrangulations under either of these limits would be very interesting. %We note that in the unweighted case one would expect that labelled quadrangulations converge under a suitable scaling limit to critical Liouville quantum gravity because the central charge of the six vertex model is $1$ (see \cite{zinn-justin-6V-random,gwynne2020liouville}).

An additional parameter that one can study in labelled quadrangulations is the collection of labels themselves. For the analogous model on the square lattice, with $v=1$ and $\omega\geq 0$, there have been several recent advances in this direction \cite{duminil2022logarithmic,glazman2023transition,duminil2024delocalization}.
                             %                              In particular, in a uniform random labelling of the square lattice, the difference between labels at distance $n$ has variance that grows like $\log(n)$ for $\omega\in[1,2]$, although the exact distribution for this difference is not known. For $\omega>2$ however, the labels are localised, meaning that the difference between labels at distance $n$ has bounded variance. Moreover, in this square lattice model it is conjectured that the labels converge under suitable scaling to the Gaussian free field. One could reasonably make the same conjecture for labelled maps.
{While a complete understanding of the limiting behaviour of the labels and the underlying map is probably beyond our reach, we suggest below several observables that one may be able to  determine using our methods or extensions thereof.
\begin{itemize}
\item The distribution of the label of a uniform random vertex in a uniform random labelled quadrangulation with $n$ faces. Indeed, in a subsequent article the second author will derive this distribution for $v=1$ and general $\omega\geq0$.
\item The number of quadrangles in a Boltzmann random patch with a given boundary length (see Section \ref{sec:comb01}) 
\item The length of a {\em level line}, a curve passing through $0$'s in the labelled quadrangulation, or perhaps through edges joining $0$'s to $1$'s. This appears to be particularly natural in E-patches, in connection with the decomposition of E-patches that underlies Proposition~\ref{prop:symmetry}. If such a level line is suitably defined, one would expect it to converge to a certain Schramm-Loewner evolution, namely SLE(4), as is the case of level lines in the Gaussian free field \cite{Schramm2009Contour}. 
\end{itemize}
For each of these observables, one could hope to verify that the limiting distribution is consistent with predictions based on physics, while the dependence of the distribution on $\omega$ and $v$ could be used to identify phase transitions in the model.}

Finally, let us recall that many geometric results on random maps have been proven using bijection with trees, so this aim serves as an additional motivation to find bijections for the results in Sections~\ref{sec:trees0} and~\ref{sec:trees1}. We note that for the case $\omega=1$, Albenque, Duchi and Schabanel have made progress on finding such a bijection (personal communication).

\bigskip
\noindent {\bf Acknowledgments.}
We thank Timothy Budd and Nina Holden for interesting discussions.
We are also grateful to both reviewers for their careful reading and for their suggestions which improved the article.

\appendix

%%%%%%%%%%%%%%%%%%%%%%%%%%%%%%%%%%%%%%%%%%%%%%%%%%%%%%%%%%%%%%% 
\section{Proof of Lemma~\ref{ct-H}}
\label{app:xm1}
%%%%%%%%%%%%%%%%%%%%%%%%%%%%%%%%%%%%%%%%%%%%%%%%%%%%%%%%%%%%%%% 

Recall from the proof of Proposition~\ref{prop:F-construct} that the coefficient of $q^n$ in $M(x)$ is a Laurent series in $x$, of valuation at least $-n$.
Moreover, the value $M(\Xrm(u))=\Xrm^-(u)$ is explicit (see~\eqref{MXu}), and $\Xrm(u)$ lies in $u\GK[[u,q]]$.     By the residue formula, we can write
\begin{align*}
  [x^{-1}] M(x) &= \frac 1 {2{\pi i}} \int M(x) dx
  \\ &= \frac 1 {2{\pi i}} \int M(\Xrm(u))\Xrm'(u) du\\
                &= \frac 1 {2{\pi i}} \int \Xrm^-(u)\Xrm'(u) du ,\\
\end{align*}
where integrals are taken  on small oriented circles
around the origin. These should be considered as formal integrals (i.e., \fps\  in $q$). Returning to the expression~\eqref{X-minus} of $\Xrm^-(u)$, we note that only the part involving $\Trm^-(u)$ contributes to the integral, and conclude that
\beq\label{ctH1}
[x^{-1}] M(x) =- \frac{ \sin (2\al) }{2{\pi i}} \int \sqrt{1-\frac {4q}u}\Yrm(u) \Xrm'(u)du,
\eeq
where we denote
\[
  \Yrm(u):= c \frac{\Trm^-(u)}{\Trm(u)}.
\]

As it happens, we can write explicitly an anti-derivative of the
integrand, up to some simple term.
\begin{Lemma}\label{lem:int}
  The following identity holds:
  \begin{multline*}
    \sqrt{1-\frac {4q}u} \Yrm(u)\Xrm'(u) =-\frac{c_3}{2u \sqrt{1-\frac{4q}u}}
    \\
    +\frac\partial{\partial u}  \left( \sqrt{1-\frac{4q}u} \left(
        \frac 1 2 \Yrm(u)\left(\Xrm(u)+ \frac 1
          {4\sin^2\!\al}\right)+ c_2 \left(1+2u
          \frac{\Trm'(u)}{\Trm(u)}\right)\right)\right) ,   \end{multline*}
  with
  \[
    c_2=\frac{\theta(\al)}{32 \sin^4\!\alpha\, \theta'(\al)} \qquad \hbox{and} \qquad
    c_3=\frac1{64 \sin^4\!\al}\left( \frac{\theta''(\al)}{\theta'(\al)} -
      \frac{\theta(\al) \theta^{(3)}(\al)}{\theta'(\al)^2}\right).
  \]
\end{Lemma}
We defer the proof of this lemma to complete the calculation of
$[x^{-1}]M(x)$. Returning to~\eqref{ctH1}, we find
\[
  [x^{-1}] M(x) = \frac{\sin (2\al) }{2{\pi i}} \int \frac{c_3}{2u
    \sqrt{1-\frac{4q}u}} \,du =
  \sin \al \cos\al \, c_3,
\]
which completes the proof of Lemma~\ref{ct-H}.
\qed

\begin{proof}[Proof of Lemma~\ref{lem:int}.]
  
  \noindent   {\bf 1.} As we prefer to handle power series in $u$ and $q$
  rather than series involving arbitrarily small powers of $u$, we first multiply the identity that we want to prove by $\sqrt{1-4q/u}$. Then, we (partially) evaluate    the derivative occurring in the right-hand side. Now the identity that we need to prove reads:
  \begin{multline*}
    \left( 1-\frac{4q}u\right) \Yrm(u)\Xrm'(u)= -\frac{c_3}{2u} +
    \frac{2q}{u^2}\left(
      \frac 1 2 \Yrm(u)\left(\Xrm(u)- \frac 1
        {4\sin^2\!\al}\right)+ c_2 \left(1+2u
        \frac{\Trm'(u)}{\Trm(u)}\right)\right)
    \\+\left(1-\frac {4q}u\right) \frac{\partial}{\partial u} \left(
      \frac 1 2 \Yrm(u)\left(\Xrm(u)- \frac 1
        {4\sin^2\!\al}\right)+ c_2 \left(1+2u
        \frac{\Trm'(u)}{\Trm(u)}\right)\right).
  \end{multline*}

  \noindent    {\bf 2. From $\bm u$ to $\bm z$.} Now, recall that $\Xrm(u)$ satisfies
  $\Xrm(4q\sin^2\!z)=\chi(z)$, where $\chi(z)$ is given by~\eqref{eq:chi_formula}. Analogously, it follows from~\eqref{T-theta} and~\eqref{Tpm-theta} that
  \[
    \Yrm (4q \sin^2\! z)= \Upsilon(z):= -c\,\frac{\sin z}{\cos z}\cdot \frac{\theta(z+\alpha)-\theta(z-\alpha)}{\theta(z)}.
  \]
  Moreover, if two series $\Brm(u)$ and $\beta(z)$ are
  related by   $\Brm(4q\sin^2\!z)=\beta(z)$, then
  \[
    8q \sin z \cos z\, \Brm'(4q\sin^2\! z)=\beta'(z).
  \]
  Combining this with the fact that $\Trm(4q\sin^2\!z)=\theta(z)/\sin z$, we find in particular that,  for $u=4q\sin^2\!z$, 
    \[
      1+2u
      \frac{\Trm'(u)}{\Trm(u)}= \frac{\sin z}{\cos z} \cdot \frac{\theta'(z)}{\theta(z)}.
    \]
  Hence the
  formula that we want to prove, written at $u=4q\sin^2\!z$, simplifies into
  \[
    \frac{\cos z}{\sin z} \Upsilon(z) \chi'(z)= c_3+  \frac{\partial}{\partial z} \left(
      \frac 1 2  \frac{\cos z}{\sin z} \Upsilon(z)\left(\chi(z)- \frac 1
        {4\sin^2\!\al}\right)+ c_2   \frac{\theta'(z)}{\theta(z)}\right)
    .\]

  \noindent{\bf 3. Back to the series $\theta$.} We now use the definitions of
  $\chi$ and $\Upsilon$ in terms of the series $\theta$. Let us also denote
  \[
    \theta_+(z)=\theta(z+\alpha)+\theta(z-\alpha) \qquad \text{and} \qquad \theta_-(z)=\theta(z+\alpha)-\theta(z-\alpha).
  \]
  Then, upon dividing by $c^2$, the above identity reads
  \[
    - \frac{\theta_-(z)}{\theta(z)} \frac{\partial}{\partial z}\left(
      \frac{\theta_+(z)}{\theta(z)}\right)
    =
    \tilde c_3 +\frac{\partial}{\partial
      z}\left(-\frac{\theta_-(z)\theta_+(z)}{2\theta(z)^2} +\tilde c_2 \frac{\theta'(z)}{\theta(z)}\right),
  \]
  where  
  \beq\label{c-tilde}
  \tilde c_2= \frac{c_2}{c^2} =\frac{2\theta(\al) \theta'(\al)}{\theta'(0)^2}
  \qquad\hbox{and}\qquad
  \tilde c_3= \frac{c_3}{c  ^2} =\frac 1 {\theta'(0)^2}
  \left( \theta'(\al) \theta''(\al)- \theta(\al)\theta^{(3)}(\al)\right).
  \eeq
  We now add $\frac{\partial}{\partial
    z}\left(\frac{\theta_-\theta_+}{2\theta^2}\right)$ to both sides, and express
  $\theta_+$ and $\theta_-$ in terms of $\theta$:  the identity that we want to prove is now:
  \beq\label{id4}
  \frac{ \theta(z-\al)\theta'(z+\al)-  \theta(z+\al)\theta'(z-\al)}{\theta(z)^2}
  = \tilde c_3+ \tilde c_2 \frac{\partial}{\partial
    z}\left( \frac{\theta'(z)}{\theta(z)}\right).
  \eeq
  
  \noindent {\bf 4. An elliptic function identity.} Recall that $\theta$ is
  closely related to Jacobi's theta function. Let us set $q=e^{2i \gamma}$, with $\gamma$ in the upper half of the
  complex plane. Then $\theta(z)$ is an entire function of $z$, which satisfies one periodicity property and one quasi-periodicity
  property (already used in Section~\ref{sec:ansatz}): 
  \[
    \theta(z+\pi)= -\theta(z) \qquad \hbox{and} \qquad  \theta(z+\gamma)=
    -e^{-i\gamma -2iz} \theta(z).
  \]
  From this, it follows that
  \[
    \frac{\theta'(z+\pi)}{\theta(z+\pi)} =  \frac{\theta'(z)}{\theta(z)} \qquad \hbox{and} \qquad 
    \frac{\theta'(z+\gamma)}{\theta(z+\gamma)} =  \frac{\theta'(z)}{\theta(z)}
    -2i.
  \]
  In particular, this implies that $\frac{\partial}{\partial
    z}\left(\frac{\theta'(z)}{\theta(z)}\right)$ is doubly periodic (or: elliptic). Note that
  it has a double pole at $z=0$, and no other poles in its fundamental domain $\{ a\pi + b\gamma: a,b \in [0, 1)\}$.

  The left-hand side of~\eqref{id4} can be written
  \[
    \frac{\theta(z+\al)\theta(z-\al)}{\theta(z)^2}
    \left( \frac {\theta'(z+\al)}{\theta(z+\al)} - \frac
      {\theta'(z-\al)}{\theta    (z-\al)}\right).
  \]
  Under this form, one easily checks that it is also elliptic (because each factor is elliptic). Its only pole in the fundamental domain is a double pole at $0$. But it is known that two elliptic
  functions with a double pole at $0$ and no other pole are linearly related. So there
  must indeed exist constants $\tilde c_2$ and $\tilde c_3$
  satisfying~\eqref{id4}. Their expressions~\eqref{c-tilde} are obtained by
  expanding~\eqref{id4} around $z=0$ up to the constant term.
\end{proof}

%%%%%%%%%%%%%%%%%%%%%%%%%%%%%%%%%%%%%%%%%%%%%%%%%%%%%%%%%%%%%%%%%%%%%%
\section{Analytic transformation of $\Mnn(x)$}
\label{app:complex}
%%%%%%%%%%%%%%%%%%%%%%%%%%%%%%%%%%%%%%%%%%%%%%%%%%%%%%%%%%%%%%%%%%%%%%

In this appendix we move to a complex analytic setting, and convert the characterisation of $\Mnn(x)$ given in Proposition~\ref{prop:M-charac} to a system of functional equations satisfied by a meromorphic
function~$\Xz$ related to $\Mnn$ by $\Mnn(\Xz(z))=\Xz(z-\gamma)$. This form was precisely our Ansatz in  Section~\ref{sec:ansatz} in  the case $v=1$ (see~\eqref{ansatz}). So this appendix can be seen as
 an inspiration for this Ansatz. As we will see in the  main result of this section (Proposition~\ref{prop:Xz_characterisation}), such a function $\chi$ exists for  general~$v$. 

In the rest of the paper we have only handled formal power series. Here all our series will be seen as functions of their (real or complex) variables. More precisely, $v$ and $\om$ are arbitrary real numbers with $v>0$, while the main variable $t$  is fixed to a sufficiently small (real and positive) value. Then the series $\Fnn(x)$ and $\Mnn(x)$ become functions of the complex variable $x$.  We first establish some of their basic properties, including a ``one-cut'' property for $\Mnn$. This property is often \emm assumed, to hold in the physics literature on maps~\cite{kostov,BDG-planaires,eynard-bonnet-potts}.

 \begin{Proposition}\label{prop:analytic}
   Fix $v,\omega\in\mathbb{R}$ with $v>0$. There is some $\epsilon>0$ such that the following holds: for fixed $t\in(0,\epsilon)$, there is a simply connected bounded domain  $\Gamma\subset\mathbb{C}$, containing $0$ and closed under complex conjugation, with the following properties:
\begin{enumerate}[label={\rm\roman*)},ref={\rm\roman*)}] 
\item \label{item:F} the series $\Fnn(x)$ converges uniformly to a holomorphic function for $x\in\Gamma$,
\item \label{item:M} the series $\Mnn(x)$ converges uniformly to $\overline{x}$ for $x\in\partial\Gamma$, the boundary of $\Gamma$,
\item \label{item:int} we have
   \[
    \frac{1}{2\pi i}\oint_{\vec{\partial\Gamma}}
    \Mnn(x)dx=tv,
  \]
  where the integral is taken counterclockwise along the boundary of $\Gamma$,
\item \label{item:cut} the function $\Lnn(x):=\left(x-\om x^2 - t(v-1)\right)^2- 4x\Fnn(x)$ has exactly two roots~$c_{0}$ and $c_{1}$ in~$\Gamma$, which are both simple and satisfy $0\leq c_{0}<c_{1}$,
\item \label{item:ext} the function $x\Mnn(x)$ extends to a bounded holomorphic function on $\Gamma\setminus[c_{0},c_{1}]$, which is related to $\Lnn(x)$ by
  \[
    \Lnn(x)=\left(x(1-\omega x)+t(v-1)-2x\Mnn(x)\right)^{2},
  \]
\item \label{item:cut-diff} for $x\in[c_0,c_1]$, we have
  \[
    \lim_{\epsilon\to 0^+}\Mnn(x+i\epsilon)+\Mnn(x-i\epsilon)=1-\omega x+\frac{t(v-1)}{x}.
  \]
\end{enumerate}
% Other information: c_0=0 if and only if v=1 and \Mnn(x) has a pole at 0 if and only if v>1, in which case \Mnn(x)~t(v-1)/x. 
\end{Proposition}

\begin{proof} Recall from Section~\ref{sec:charac} that for fixed $v,\omega\in\mathbb{R}$, we have $\Mnn(x)\in \frac{t}{x}\mathbb{R}[[x,\frac{t}{x}]]$. For $j,k\in\mathbb{Z}$, denote
  % define
  $m_{j,k}:=[t^{j}x^{k}]\Mnn(x)$. This number is non-zero only if $j\geq 1$ and $j+k\geq 0$. Similarly, denote $f_{j,k}:=[t^{j}x^{k}]\Fnn(x)$, which is non-zero only if $j\geq 1$ and $k\ge 0$ (Proposition~\ref{prop:F}).
  Using the definition~\eqref{M-def} of $\Mnn$ in terms of patches, and crude bounds on the number of planar maps with a given number of edges, one can show that the values $|m_{j,k}|$ grow at most exponentially in both $j$ and $|k|$. From Proposition~\ref{prop:M} and~\eqref{FM} respectively, this implies that the coefficients of $\Znn(x,y)$ and $\Fnn(x)$ (seen as \fps\  in $t$, $x$ and $y$) also grow at most exponentially.
  Consequently,
  we can choose $\sigma$ sufficiently small to ensure that $\Fnn(\sigma;\sigma)$ and $\Znn(\sigma,\sigma;\sigma)$ converge absolutely, where we have exceptionally  indicated in the notation the dependency in $t$ (as the final variable). Moreover, for all sufficiently small $\sigma>0$, we have
  \beq\label{ineqM1}
    \sum_{(j,k)\neq(1,-1)}(|k|+1)|m_{j,k}|\sigma^{2j+k-1}<\min\left(\frac{v}{2},\frac{1}{2}\right),
  \eeq
  and
  \[\sum_{(j,k)\neq(1,0)}|f_{j,k}|\sigma^{j+k-1}<\frac{v}{2(v+1)},\]
  because  the exponent of $\sigma$ in {each} non-zero summand is always positive.
   Fix $\sigma>0$ such that the above properties hold,
  and moreover 
  \beq\label{conds-sigma}
    \sigma< \min \left\{\frac{1}{2}, \frac{v}{2}, \frac{1}{2v}, \frac{1}{v+1}, \frac 1{|\om|}\right\},
  \eeq
  where $1/|\omega|=\infty$ when $\omega=0$. Such a $\sigma$ exists
  % which is possible
  due to our assumption that $v>0$.
  Recall that $[x^{-1}]\Mnn(x)=tv$. It follows from~\eqref{ineqM1} 
  that for $|x|,|\frac{t}{x}|<\sigma$ we have
  \beq\label{ineq-M1}
    \left|\Mnn(x)-\frac{vt}{x}\right|=\left|\sum_{(j,k)\neq(1,-1)}m_{j,k}x^{j+k} (t/x)^{j}\right|<\sum_{(j,k)\neq(1,-1)}|m_{j,k}t/x|\sigma^{2j+k-1}<\frac{v|t|}{2|x|}
  \eeq
and
 \beq\label{ineq-M2}
 \left|x\Mnn'(x)+\frac{vt}{x}\right|=\left|\sum_{(j,k)\neq(1,-1)}km_{j,k}x^{j+k} (t/x)^{j}
  \right|<\sum_{(j,k)\neq(1,-1)}|km_{j,k}t/x|\sigma^{2j+k-1}<\frac{v|t|}{2|x|},
\eeq
where $\Mnn'(x)$ denotes the derivative of $\Mnn(x)$ with respect to $x$. In particular, for fixed $t\in[0,\sigma^{2})$, $\Mnn(x)$ converges absolutely in the annulus $A_1:=\{x: |x|,|\frac{t}{x}|<\sigma\}$, and in fact for $|x|,|\frac{t}{x}|\le\sigma$. Similarly, for $|x|,|t|<\sigma$, we have
\beq\label{ineq-F1}
    \left|\Fnn(x)-t\right|=\left|\sum_{(j,k)\neq(1,0)}f_{j,k}x^{k} t^{j}\right|<\sum_{(j,k)\neq(1,0)}|f_{j,k}t|\sigma^{j+k-1}<\frac{v|t|}{2(v+1)}.
  \eeq
Now let $\epsilon=\sigma^{9}$ and fix $t\in[0,\epsilon)$. For $1\le j \le 4$, let us define more generally the annulus $A_{j}$ by:
\[
  A_{j}=\{x\in\mathbb{C}:t\sigma^{-j}<|x|<\sigma^{j}\}.
\]
Observe that $A_4 \subset A_3 \subset A_2 \subset A_1$. In particular, we have  $|x|,|\frac{t}{x}|<\sigma$ for any $x$ in these annuli, so~\eqref{ineq-M1} and~\eqref{ineq-M2} both hold. Then for $x\in A_{j}$ with $j>1$, it follows from~\eqref{ineq-M1} and~\eqref{conds-sigma} that 
\beq\label{ineq-M3}
  t\sigma^{1-j}<\frac{vt}{2}\sigma^{-j}<\frac{vt}{2|x|}<|\Mnn(x)|<\frac{3vt}{2|x|}<\frac{3v}{2}\sigma^{j}<\sigma^{j-1},
\eeq
that is, $\Mnn(x)\in A_{j-1}$. We will use this to show that for $x\in A_{4}$ we have $\Mnn(\Mnn(x))=x$. 
 Consider the series $\Znn(x,y)$ (see Proposition~\ref{prop:M}). As discussed, $\Znn(\sigma,\sigma;\sigma)$ converges absolutely, so $\Znn(x,y;t)$ converges absolutely for $|x|,|y|<\sigma$ (recall that $t<\sigma^9<\sigma$). In particular, for $x\in A_{4}$, we have $|\Mnn(x)|<\sigma$, and thus $\Znn(x,y)$ converges at $y= \Mnn(x)$. So there can be no pole in $\Znn(x,y)$ as $y\to \Mnn(x)$, and the numerator of $\Znn(x,y)$ must vanish at $y=\Mnn(x)$. That is,
 \[
   \left(1-\frac{1}{x}\Mnn(\Mnn(x))\right)\left(1-x-\omega\Mnn(x)+\frac{t}{\Mnn(x)}(v-1)-\Mnn(\Mnn(x))\right)=0.
 \]
Since $x\in A_{4}$, we have $\Mnn(x)\in A_{3}$ and $\Mnn(\Mnn(x))\in A_{2}$, hence the modulus of the second term is bounded from below by $1-\sigma^{4}-|\omega|\sigma^{3}-|v-1|\sigma^{3}-\sigma^2$. From \eqref{conds-sigma} we have $|v-1|\sigma<\max\{v,1\}\sigma<\frac{1}{2}$ and $|\omega|\sigma^{3}<\sigma^2<\frac{1}{4}$, which  implies that this expression is positive. Hence, for $x \in A_4$,  we have
\[
  \Mnn(\Mnn(x))=x.
\]
Now let $\widetilde{\Gamma}=\{x\in \overline{A_{4}}:|\Mnn(x)|>|x|\}$. Since $\Mnn(x)$ is a series with real coefficients that converges in this region, $\widetilde{\Gamma}$ is closed under conjugation. Since $t<\sigma^9$, it follows from~\eqref{ineq-M3} and~\eqref{conds-sigma} that~$\widetilde{\Gamma}$ contains the (non-empty)  annulus
  \[
  A_5:=  \left\{x \in \cs: t \sigma^{-4} \le |x| <\sqrt{\frac{vt}2}\right\}.
  \]
  The inequalities~\eqref{ineq-M3} and~\eqref{conds-sigma} also imply that no point of the outer boundary of $A_4$, where $|x|=\sigma^4$,  is  in~$\widetilde \Gamma$. Let~$\Gamma$ be the union of
  the connected component of $\widetilde \Gamma$ that contains $A_5$ and all the (bounded) regions that it contains. Then $\Gamma$ is a simply connected bounded domain, closed by conjugation, containing $0$,
  whose boundary $\partial\Gamma$
forms a non-contractible loop in $A_{4}$ on which $|\Mnn(x)|=|x|$ (see Figure~\ref{fig:M_regions}, right). Since any point~$x$ of~$\Gamma$ satisfies $|x|<\sigma$, Point~\ref{item:F} of the proposition holds. Moreover, since $\partial \Gamma\subset A_1$, the series $\Mnn(x)$ is analytic in a neighbourhood of~$\partial \Gamma$, proving convergence in Point~\ref{item:M}.

\begin{figure}[ht]
  \centering
  \includegraphics[scale=0.7]{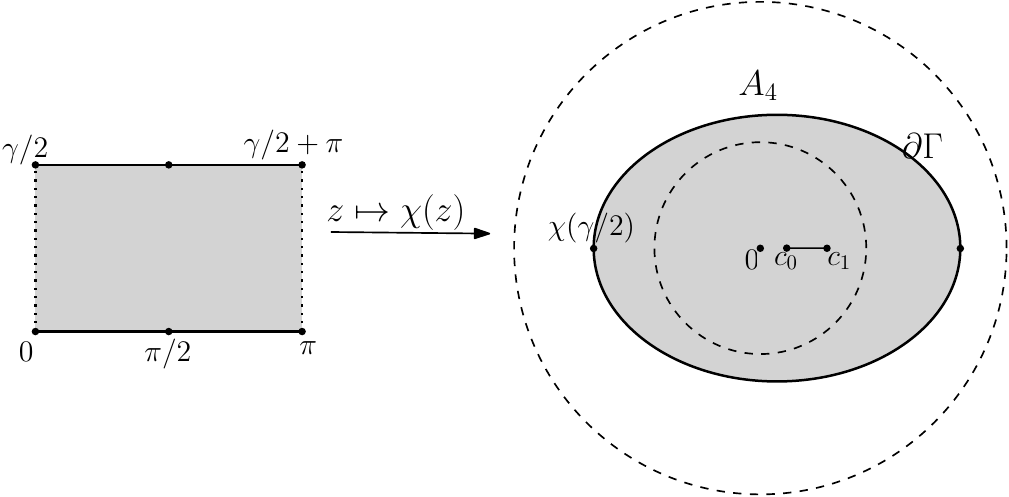}
  \caption{{\em Right:} the region $\Gamma$ (shaded) and the annulus $A_{4}$ (with dashed boundary), where $\Fnn(x)$ and $\Mnn(x)$ are respectively holomorphic, along with roots $c_{0}$ and $c_{1}$ of $\Lnn(x)$. {\em Left:} the cylinder on which $\Xz$ is defined. We fix  $\Xz(0)=c_{0}$ and this implies that $\Xz(\pi/2)=c_{1}$.}
  \label{fig:M_regions}
\end{figure}

Let us complete the proof of~\ref{item:M} by showing that for $x\in\partial\Gamma$, we have $\Mnn(x)=\overline{x}$. Suppose the contrary, that there is some $x\in\partial\Gamma$ with $\Mnn(x)\neq\overline{x}$, and write $y=\overline{\Mnn(x)}$. Since $x$ is on the boundary of $\Gamma$, we have $|y|=|x|$. Moreover, since $x \in A_4$, we have $\Mnn(\Mnn(x))= \Mnn(\bar y)=x$, thus
$\Mnn(y)=\overline{x}$ (because the series $\Mnn$ has real coefficients).
Writing $r=|x|=|y|$, let $\ell$ be the shortest arc of the circle of radius $r$ around $0$ between $x$ and $y$. Then the length of $\ell$ is at most $\frac{\pi}{2}|x-y|$ (this is equivalent to the inequality $\theta\le\pi \sin \frac \theta 2$ for $0\le \theta \le \pi$). Moreover, for $z$ on this arc, we have, according to~\eqref{ineq-M2}:
\[
  \left|\Mnn'(z)+\frac{tv}{z^2}\right|<\frac{vt}{2|z^2|}=\frac{vt}{2r^2}.
\]
Hence, taking the integral along this arc, we have
\begin{multline*}
    \frac{vt}{2r^2}\frac{\pi}{2}|x-y|>
    \left|\int_{x}^{y}\left(\Mnn'(z)+\frac{tv}{z^{2}}\right) dz\right|=\left|\Mnn(y)-\Mnn(x)-\frac{tv}{y}+\frac{tv}{x}\right|\\
    = \left| \bar x -\bar y-\frac{tv}{y}+\frac{tv}{x} \right| =|x-y|\left|1+\frac{tv}{r^2}\right|,
\end{multline*}
which is a contradiction (because $v$ and $t$ are positive and $\pi<4$).
Therefore, for $x\in\partial\Gamma$, we have $\Mnn(x)=\overline{x}$. This concludes the proof of~\ref{item:M}.

Let us now address Point~\ref{item:int}. Since $\partial\Gamma$ is in the interior of $A_4$, where $\Mnn(x)$ converges absolutely, this is precisely Cauchy's formula, as $[x^{-1}]\Mnn(x)=tv$.

\medskip
We now embark on the proof of Point~\ref{item:cut}.
As a formal power series, $\Fnn(x)$ is related to $\Mnn(x)$ by~\eqref{FM}:
\[
  \Fnn(x):=(x\Mnn(x)-t(v-1))(1-\om x-\Mnn(x)).
\]
Moreover, as already discussed, $\Fnn(\sigma;\sigma)$ converges absolutely, so $\Fnn(x;t)$ converges for $|x|<\sigma$.
In particular, for $x\in A_{1}$, the series $\Fnn(x;t)$ and $\Mnn(x;t)$ both converge, so they must be related by the equation above. We now define the function $\Lnn(x)$, for $|x|<\sigma$, as the discriminant of the above quadratic equation in $\Mnn(x)$:
\[
  \Lnn(x):=\left(x(1-\omega x)-t(v-1)\right)^{2}-4x\Fnn(x).
\]
It is also holomorphic for $|x|<\sigma$, and in particular in $\Gamma$.
Moreover, for $x\in A_1$, we also have:
\beq\label{LM}
  \Lnn(x)=\left(x(1-\omega x)+t(v-1)-2x\Mnn(x)\right)^{2}.
  \eeq 
By the argument principle, the number of roots of $\Lnn(x)$ in $\Gamma$, counted with multiplicity, is
\[
  \frac{1}{2\pi i}\oint_{\vec{\partial\Gamma}}\frac{\Lnn'(x)}{\Lnn(x)}dx,
\]
assuming that $\Lnn$ does not vanish on the contour, which we will now check. Writing $\Knn(x)=1-\omega x+\frac{t(v-1)}{x}-2\Mnn(x)$, we have $\Lnn(x)=x^{2}\Knn(x)^{2}$ and for $x\in{\partial\Gamma}\subset A_{4}$, we have $|\Mnn(x)|=|x|$, so $|1-\Knn(x)|<|\omega|\sigma^{4}+|v-1|\sigma^4+2\sigma^{4}<1$, where the second inequality is due to \eqref{conds-sigma}. In particular, $\Lnn$ does not vanish on~$\partial \Gamma$, and moreover the function $\log(\Knn(x))$, defined with the principal value of $\log$, is analytic on a neighborhood of $\partial\Gamma$. Therefore
\beq\label{even-int}
  \frac{1}{2\pi i}\oint_{\vec{\partial\Gamma}}\frac{\Lnn'(x)}{\Lnn(x)}dx=
  \frac{1}{2\pi i}\oint_{\vec{\partial\Gamma}}\left(\frac{2}{x}+\frac{2\Knn'(x)}{\Knn(x)}\right)dx=2.
\eeq
Hence $\Lnn(x)$ has exactly two roots in $\Gamma$, counted with multiplicity.
We will now show that these lie on the real line. To do this we note that $0$ and $\sigma^{-2}t$ belong to $\Gamma$, and we will show that $\Lnn(0),\Lnn({\sigma^{-2}t})\geq 0$ while $0<(v+1)t<{\sigma^{-2}t}$ and $\Lnn((v+1)t)<0$, as then it follows from the intermediate value theorem that $\Lnn(x)$ has roots $c_{0}\in[0,(v+1)t)$ and $c_{1}\in ((v+1)t,{\sigma^{-2}t}]$, since for $x\in\mathbb{R}\cap\Gamma$ we have $\Lnn(x)\in\mathbb{R}$. 
The first three inequalities are fairly simple: $\Lnn(0)=t^2(v-1)^2\ge 0$, while ${\sigma^{-2}t}\in A_1$ where~\eqref{LM} holds, so $\Lnn({\sigma^{-2}t})\geq 0$. The inequality $0<(v+1)t<{\sigma^{-2}t}$ follows from \eqref{conds-sigma}. Finally, combining the definition of $\Lnn$ in terms of $\Fnn$ with \eqref{ineq-F1} yields
  \[\Lnn((v+1)t)\leq (2t-\omega t^{2}(v+1)^2)^{2}-4t(v+1)\left(t-\frac{vt}{2(v+1)}\right)=t^{2}\left(\left(2-\omega t(v+1)^2\right)^{2}-4-2v\right)
  \]
  (note that~\eqref{ineq-F1} applies because $t(v+1)<\sigma$).
  Then the fact that $\Lnn((v+1)t)<0$ follows from $|\omega t(v+1)^2|<\sigma^{6}$ and $(2+\sigma^{6})^2<4+2v$, which both follow from \eqref{conds-sigma} and $t<\sigma^{9}$. Hence by the intermediate value theorem, $\Lnn(x)$ has roots $c_{0}\in[0,(v+1)t)$ and $c_{1}\in ((v+1)t,{\sigma^{-2}t}]$, so these are the only two roots of $\Lnn(x)$ in $\Gamma$ (see Figure~\ref{fig:M_regions}, right). This completes the proof of~\ref{item:cut}.

\medskip
We now show that $\Mnn(x)$, which we have so far proved to be analytic in $A_1$, has a meromorphic continuation to $\Gamma\setminus[c_{0},c_{1}]$. From \eqref{LM}, it suffices to show that $\Lnn(x)$ has a holomorphic square root in $\Gamma\setminus[c_{0},c_{1}]$ which extends
  \beq\label{sq-def}
    \sqrt{\Lnn(x)}:=x(1-\omega x)+t(v-1)-2x\Mnn(x),
  \eeq 
  defined on $A_1$. We apply  Lemma \ref{lem:rootexist} below with $\Omega=\Gamma\setminus[c_{0},c_{1}]$, $\Pi=\Gamma \cap A_1$, $h=\Lnn$ and $s$ the right-hand side of~\eqref{sq-def}. This lemma implies that
    such a square root exists if and only if for any smooth closed curve~$\gamma$ in $\Gamma\setminus[c_{0},c_{1}]$,
\[ %\begin{equation}\label{eq:2Z}
  \frac{1}{2\pi i}\oint_{\gamma}\frac{\Lnn'(x)}{\Lnn(x)}dx \in 2\mathbb{Z}.
\] %\end{equation}
    %     Moreover, \eqref{eq:2Z}
This holds in our case because the left-hand side
    %     integral is $4\pi i$
is $2$ when $\gamma=\vec{\partial\Gamma}$ (see~\eqref{even-int}) and the integrand 
    %     in \eqref{eq:2Z}
has no singularities in $\Gamma\setminus[c_{0},c_{1}]$, so the left-hand side above %integral
    %     in \eqref{eq:2Z}
is
    %     equal to
twice  the number of times $\gamma$ winds counterclockwise around the cut $[c_{0},c_{1}]$. We thus conclude that $\sqrt{\Lnn(x)}$ extends analytically to $\Gamma\setminus[c_0,c_1]$.

Since $\Lnn(x)$ has roots at $c_{0}$ and $c_{1}$, the function $\sqrt{\Lnn(x)}$ must have branch cuts at these points. Moreover, $\Lnn(x)$ has no other roots in $\Gamma$, so $\sqrt{\Lnn(x)}$ has no other branch cuts in $\Gamma$. This implies that for $x_c\in[c_{0},c_{1}]$, the limit of $\sqrt{\Lnn(x)}$ as $x$ approaches $x_{c}$ from above the cut and the same limit as $x$ approaches $x_{c}$ from below the cut must sum to $0$. The relation between $\sqrt{\Lnn(x)}$ and $\Mnn(x)$ implies
\[
  \lim_{\epsilon\to 0^+}\Mnn(x+i\epsilon)+\Mnn(x-i\epsilon)=1-\omega x+\frac{t(v-1)}{x}.
\]
This completes the proof of~\ref{item:cut-diff}, and of the whole proposition.
\end{proof}

Below we give an elementary proof of the result used at the end of the last proof about the existence of an analytic square root of a given analytic function.
  % describing the cases when an analytic function has an analytic square-root.
We found an alternative explanation using algebraic topology in a blog post by Evan Chen \cite{Chen_roots} (see Theorem 2). We searched several textbooks for this result, but were not able to find it. Nonetheless, we expect it exists somewhere.
  
\begin{Lemma}\label{lem:rootexist}
    Let $\Pi\subset\Omega\subset\mathbb{C}$ be path connected regions and let $h:\Omega\to\mathbb{C}\setminus\{0\}$ and $s:\Pi\to\mathbb{C}\setminus\{0\}$ be an analytic functions satisfying $s(z)^{2}=h(z)$ for $z\in\Pi$. Then $s$ extends to an analytic function on $\Omega$ satisfying $s(z)^{2}=h(z)$ for $z\in\Omega$ if and only if, for every smooth closed curve $\gamma$ in $\Omega$, the following holds
    \begin{equation}
      \frac{1}{2\pi i}\oint_{\gamma}\frac{h'(z)}{h(z)}\in 2\mathbb{Z}.\label{eq:rootexist}
    \end{equation}
\end{Lemma}

\begin{proof}
We first prove the following: for any $x,y\in\Omega$ we have
\begin{equation}  \label{eq:logexist}
  h(y)=h(x)\exp\left(\int_{x}^{y}\frac{h'(z)}{h(z)}dz\right),
\end{equation}
where the integral is taken along any smooth curve $\rho:[0,1]\to\Omega$ with $\rho(0)=x$ and $\rho(1)=y$ (the integral is well defined as $h$ does not vanish in $\Omega$). Writing $\xi(u):=h(\rho(u))$ and $\beta(u):=\int_0^{u}\frac{\xi'(v)}{\xi(v)}dv$ for $u\in[0,1]$, the above equation is equivalent to
\[
  \xi(1)=\xi(0)\exp\left(\beta(1)\right).
\]
From the definition of $\beta(u)$, we have immediately $\beta'(u)= \frac{\xi'(u)}{\xi(u)}$, which implies that
\[
  \frac{d}{du}
  \left( \xi(u)e^{-\beta(u)}\right)=\xi'(u)e^{-\beta(u)}-\xi(u)\beta'(u)e^{-\beta(u)}=0.\]
Hence $\xi(u)e^{-\beta(u)}$ is a constant for $u\in[0,1]$, and, in particular
\[
  \xi(1)e^{-\beta(1)}=\xi(0)e^{-\beta(0)}=\xi(0),
\]
as required.

We will now use \eqref{eq:logexist} to prove the lemma. First, if $s(z)$ does extend to $\Omega$, we can apply \eqref{eq:logexist} with $h$ replaced by $s$. Applying this to a smooth closed curve $\gamma$ starting and ending at some point $x\in\Omega$ then dividing both sides by $s(x)=s(y)$ yields
\[
  1=\exp\left(\oint_{\gamma}\frac{s'(z)}{s(z)}dz\right)=\exp\left(\oint_{\gamma}\frac{h'(z)}{2h(z)}dz\right).
\]
Hence 
\[
  \oint_{\gamma}\frac{h'(z)}{2h(z)}dz\in 2\pi i\mathbb{Z},
\]
as required.

For the other direction, assume that \eqref{eq:rootexist} holds for any smooth closed curve $\gamma$ in $\Omega$. Fix $x\in\Pi$. By~\eqref{eq:logexist}, with $h$ replaced by $s$ and $\Omega$ replaced by $\Pi$, for any $y\in\Pi$ we have
\beq\label{s-prop}
  s(y)=s(x)\exp\left(\int_{x}^{y}\frac{s'(z)}{s(z)}dz\right)=s(x)\exp\left(\int_{x}^{y}\frac{h'(z)}{2h(z)}dz\right),
\eeq
where the integral may be along any smooth curve in $\Pi$ from $x$ to $y$. For
    %     any
$y\in\Omega$, let us now define $\tilde s(y)$ by
  \[
    \tilde s(y):=s(x)\exp\left(\int_{x}^{y}\frac{h'(z)}{2h(z)}dz\right),
  \]
where the integral is taken along any smooth curve in $\Omega$ from $x$ to $y$. Note that by \eqref{eq:rootexist}, any two integrals for the same points $x$ and $y$ differ by a multiple of $2\pi i$, so the value of $\tilde s(y)$ defined above does not depend on the chosen curve from $x$ to $y$. Hence this defines an analytic function~$\tilde s$ on $\Omega$ which coincides with the original function $s$ on $\Pi$ (see~\eqref{s-prop}).  Finally squaring both sides in the definition of $\tilde s(y)$ and comparing to \eqref{eq:logexist} yields $\tilde s(y)^2=h(y)$, as required. 
\end{proof}

We will now parametrise the domain $\Gamma\setminus [c_{0},c_{1}]$ using the  following classical result (see for example~\cite[Chap.~5, Sec.~1]{Goluzin1969geometric} for an equivalent statement with ``cylinder'' replaced by ``annulus'').

\begin{Theorem}\label{thm:uniformisation}
  Any doubly connected domain $B$ other than the punctured disk and punctured plane is conformally equivalent to some cylinder $\Cyl:=\{x+ciy:x\in\mathbb{R}, y\in(0,1)\}/\pi\mathbb{Z}$ with $c>0$. Moreover, the value of $c$ is uniquely determined by $B$.  The function $\chi$ mapping this cylinder onto $B$ is uniquely determined once we choose the point $\lim_{z\rightarrow 0} \chi(z)$ on the boundary of $B$.
  \end{Theorem}

We shall apply this to the domain $B:=\Gamma\setminus[c_0,c_1]$, with $\chi(0)=c_0$. See Figure~\ref{fig:M_regions} for an illustration of  the parametrising function $\Xz$, with  $\gamma=2ic$. We will extend $\Xz$ meromorphically from the cylinder to $\cs$.
The characterisation of $\Mnn(x)$, once rewritten at $x=\chi(z)$, will translate into the following simpler characterisation of $\chi$, which justifies and generalises (to the case $v\neq1$) the Ansatz of Section~\ref{sec:ansatz}.

\begin{Proposition}\label{prop:Xz_characterisation}
  Fix $v,\omega\in\mathbb{R}$ with $v>0$. For $t>0$, sufficiently small, there is a value $\gamma\in i\mathbb{R}_{>0}$ and a meromorphic function $\Xz:\mathbb{C}\to\mathbb{C}$ satisfying the following equations for $z\in\mathbb{C}$:
\begin{align*}
\Xz(z)&=\Xz(z+\pi),\\
\Xz(z)&=\Xz(-z),\\
\Xz(\gamma+z)+\omega\Xz(z)+\Xz(z-\gamma)&=1+\frac{t(v-1)}{\Xz(z)},\\
\int_0^{\pi}\Xz\left(\frac{\gamma}{2}-z\right)\Xz'\left(\frac{\gamma}{2}+z\right)dz&=-2{\pi i} tv.
\end{align*}
Moreover, $\Xz$ has one root and no poles in the closed cylinder 
$\overline \Cyl:=\{x+y\gamma/2:x\in\mathbb{R}, y\in[0,1]\}/\pi\mathbb{Z}$, while $\Xz(z)\Xz(\gamma-z)$ has no poles in this region.

Finally, for $z$ in the strip $\Strip:=\{x+y\gamma/2:x\in\mathbb{R}, y\in(0,1)\}$, we have
 \beq\label{relations-strip}
  \Mnn(\Xz(z))=\Xz(\gamma-z).
\eeq
\end{Proposition}
 
\begin{proof}
  The region $B:=\Gamma\setminus [c_{0},c_{1}]$ is doubly connected, and it is neither
  the punctured disk nor the punctured plane, so by Theorem~\ref{thm:uniformisation}, there is a unique value $\gamma\in i\mathbb{R}_{>0}$ (where $\gamma=2ic$ in the notation of the theorem) and a unique bijective conformal map
$ \Xz$ from the cylinder $\Cyl:=\Strip/\pi\mathbb{Z}$ to~$\Gamma$
satisfying $\Xz(0)=c_{0}$ when $\Xz$ is extended by continuity to the closed cylinder  $\overline \Cyl$ (which is $\overline \Strip/\pi\mathbb{Z}$). 
  The map $\Xz$ sends the top boundary $(\mathbb{R}+\gamma/2)/\pi\mathbb{Z}$ of the cylinder
  bijectively to  $\partial\Gamma$.
  The bottom boundary $\mathbb{R}/\pi\mathbb{Z}$ is sent to $[c_{0},c_{1}]$, with $c_{0}$ and $c_{1}$ each having a unique inverse while each point in $(c_{0},c_{1})$ has two inverses accounting for the limits coming from above and below. Thus $\Xz$ is 1-to-1 from $\overline \Cyl \setminus (0, \pi)$ to
  $\overline{\Gamma}\setminus(c_{0},c_{1})$,
  and $2$-to-$1$ from $(0, \pi)$ to $(c_0,c_1)$.
  In particular, since $c_{0}\geq 0$, (see Proposition~\ref{prop:analytic}~\ref{item:cut}) the function $\Xz$ has a unique root in $\overline \Cyl$, as claimed in the proposition. The fact that it has no pole is obvious, since the closed cylinder is sent by $\chi$ to $\overline \Gamma$.

  For convenience we will now see $\Xz$ as a function on $\overline \Strip\subset\mathbb{C}$ rather than $\overline \Cyl=\overline \Strip/\pi\mathbb{Z}$, satisfying  $\Xz(z)=\Xz(z+\pi)$. We notice that the function sending $z$ to $\overline{\Xz(-\overline{z})}$ is also a $\pi$-periodic function defined on $\Strip$, and a conformal map from $\Cyl$ to $\Gamma$ (which is closed by conjugation) sending $0$ to~$c_0$.
   By Theorem~\ref{thm:uniformisation} this function must actually be equal to $\Xz$, that is,
  \beq\label{unique-chi}
    \Xz(-\overline{z})=\overline{\Xz(z)} \qquad \text{ for all } z\in\overline{\Strip}.
    \eeq
    In particular, any point $z\in \rs$ is sent by $\chi$ {to a point} in $[c_0,c_1]$, and then we have
    \[
      \Xz(-z)=\Xz(z)=\Xz(\pi-z),
    \]
    which implies in particular that $\Xz(\pi/2)=c_1$. 
    Then, for $z\in\mathbb{R}+\gamma/2$, we have
\beq\label{M-prop-gamma}
  \Mnn(\Xz(\gamma-z))=\Mnn(\Xz(-\overline{z}))=\Mnn\left(\overline{\Xz(z)}\right)=\Xz(z),
\eeq
where we have first used the fact that $\gamma -z=-\overline z$, then the identity~\eqref{unique-chi} on $\chi$, and finally Proposition~\ref{prop:analytic}~\ref{item:M}.

We will now extend $\chi(z)$ to a meromorphic function on $\mathbb{C}$. Our construction is illustrated in Figure \ref{fig:Strip_extension}).
  Using the rule $\Xz(z)=\Mnn(\Xz(\gamma- z))$, which we have just seen to hold on the top boundary of $\Strip$,
  we first extend $\chi$ to the doubled closed strip $\{x+y\gamma:x\in\mathbb{R}, y\in[0,1]\}$.
Subsequently, we use the rule $\Xz(-z)=\Xz(z)$, which holds on the real axis, to extend $\Xz$ to the  strip $\{x+y\gamma:x\in\mathbb{R}, y\in[-1,1]\}$, of height $2|\gamma|$.
By construction,  the extension satisfies,
  for $z\in\overline{\Strip}$:
\[
   \Xz(z)=\Xz(-z)~~~~~\text{and}~~~~~\Mnn(\Xz(z))=\Xz(\gamma-z).
   \]
   This establishes in particular~\eqref{relations-strip}.

   \begin{figure}[h]
  \centering
  \includegraphics[scale=0.85]{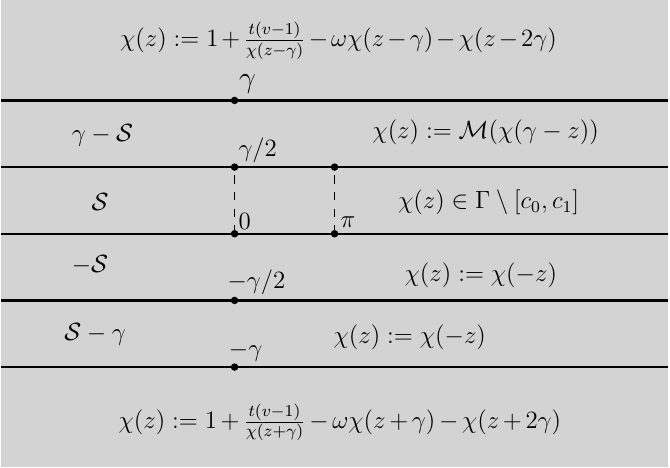}
  \caption{The function $\Xz$ is first defined on $\overline{\Strip}$ by the parametrisation, then on $\gamma-\overline{\Strip}$, then on $-\overline{\Strip}$ and $\overline{\Strip}-\gamma$, and finally on the outer regions using the rules shown. The resulting function is meromorphic on $\mathbb{C}$ because all rules used hold on the boundaries between the regions.}
  \label{fig:Strip_extension}
\end{figure}
Now for $z\in\mathbb{R}$ and $\epsilon>0$ sufficiently small we have $z+i\epsilon\in\Strip$ and $-z+i\epsilon\in\Strip$, so
\begin{align*}
  \Xz(\gamma-z)+\Xz(\gamma+z)&=\lim_{\epsilon\to 0^+}\Xz(\gamma-z-i\epsilon)+\Xz(\gamma+z-i\epsilon)\\
&=\lim_{\epsilon\to 0^+}\Mnn(\Xz(z+i\epsilon))+\Mnn(\Xz(-z+i\epsilon))\\
&=\lim_{\epsilon\to 0^+}\Mnn(\Xz(z+i\epsilon))+\Mnn\left(\overline{\Xz(z+i\epsilon)}\right) \hskip 10mm \text{by~\eqref{unique-chi}}\\
&=1-\omega\Xz(z)+\frac{t(v-1)}{\Xz(z)}  \hskip 10mm \text{by Proposition~\ref{prop:analytic}~\ref{item:cut-diff}},
\end{align*}
where for the last equality we use the fact that $\Xz(z+i\epsilon)$ and $\overline{\Xz(z+i\epsilon)}$ approach $\Xz(z)$ from opposite sides of the cut $[c_{0},c_{1}]$.
That is, for $z\in \rs$
\[
  1+\frac{t(v-1)}{\Xz(z)}=\Xz(\gamma+z)+\omega\Xz(z)+\Xz(z-\gamma).
\]
Now for $\Im(z)\geq0$ we use this equation to define $\Xz(\gamma+z)$ in terms of $\Xz(z)$ and $\Xz(z-\gamma)$: this first gives the value of $\chi$  on the strip $\{ z: |\gamma| \le \Im(z) \le 2|\gamma|\}$, then on the strip of height $|\gamma|$ just above this one, and so on. Analogously,  for $\Im(z)\leq0$ we use this equation to define $\Xz(z-\gamma)$ in terms of $\Xz(z)$ and $\Xz(z+\gamma)$. This extends $\Xz$ to a meromorphic function defined on all of $\mathbb{C}$, noting that poles can appear in $\Xz(z+\gamma)$ either following a root of $\Xz(z)$ or a pole of $\Xz(z)$ or $\Xz(z-\gamma)$, and similarly for $\Xz(z-\gamma)$. By the meromorphic extension, this function $\Xz$  satisfies the first three equations of the proposition, which originally hold in $\Strip$ because of the way we chose to extend~$\Xz$.

Now consider the integral for $\Xz$ occurring in the proposition.
It can be related to an integral for $\Mnn$ by substituting $x=\Xz(\frac{\gamma}{2}+z)$ (so that $x$ runs clockwise along $\partial \Gamma$ as $z$ moves from $0$ to $\pi$), which yields
\[
  \int_0^{\pi}\Xz\left(\frac{\gamma}{2}-z\right)\Xz'\left(\frac{\gamma}{2}+z\right)dz=-\oint_{\vec{\partial\Gamma}}\Mnn(x)dx=-2{\pi i} tv,
\]
where we have used~\eqref{M-prop-gamma} to write $\Mnn(x)=\chi(\gamma/2-z)$ and then Proposition~\ref{prop:analytic}~\ref{item:int}.

Let us finally prove the statement of the proposition dealing with the poles of $\Xz(z)\Xz(\gamma-z)$ in $\overline\Cyl$.  For $z\in\Strip$, we have $\Xz(z)\Xz(\gamma-z)=\Xz(z)\Mnn(\Xz(z))$ (by~\eqref{relations-strip}), which is bounded by Proposition~\ref{prop:analytic}~\ref{item:ext} since $\Xz(z)\in\Gamma\setminus[c_0,c_1]$. Hence $z\mapsto \Xz(z)\Xz(\gamma-z)$ cannot have a pole in $\overline{\Strip}$. 
\end{proof}
In the case $v=1$, the third equation of Proposition~\ref{prop:Xz_characterisation} specialises to~\eqref{diff}, which we can solve in terms of Jacobi's theta function as in Section~\ref{sec:ansatz}. Then the series $\Mnn(x)$, seen as an analytic function, is characterised implicitly by~\eqref{relations-strip}.

\bigskip

%%%%%%%%%%%%%%% 
% Bibliography
%%%%%%%%%%%%%%% 

\bibliographystyle{abbrv}%{plain}
\bibliography{oe}

\end{document}